\RequirePackage{fix-cm}
\RequirePackage{amsmath}
\RequirePackage{amssymb}
\documentclass[12pt]{article}
\usepackage{mathrsfs}
\usepackage{bm}
\usepackage{amsthm}
\usepackage[colorlinks,linkcolor=blue,anchorcolor=red,citecolor=blue]{hyperref}
 \usepackage{geometry}
\newcommand{\qe}{}
\newtheorem{Thm}{Theorem}[section]

\newtheorem{Lem}[Thm]{Lemma}
\newtheorem{Coro}[Thm]{Corollary}

\newcommand{\1}{\mathbf{1}}
\newcommand{\R}{\mathbb{R}}
\newcommand{\Rd}{{\mathbb{R}^3}}

\renewcommand{\P}{\mathbf{P}}
\renewcommand{\Re}{\text{Re}}

\newcommand{\F}{\mathscr{F}}
\newcommand{\N}{\mathbb{N}}

\newcommand{\E}{\mathcal{E}}
\newcommand{\D}{\mathcal{D}}

\usepackage{cite}

\allowdisplaybreaks
\renewcommand{\Re}{\text{Re}}
\renewcommand{\S}{\mathbb{S}}
\newcommand{\<}{\langle}
\renewcommand{\>}{\rangle}

\newcommand{\I}{\mathbf{I}}
\newcommand{\II}{\mathbf{I}_{\pm}}
\renewcommand{\P}{\mathbf{P}}
\newcommand{\PP}{\mathbf{P}_{\pm}}

\newcommand{\vertiii}[1]{{\left\vert\kern-0.25ex\left\vert\kern-0.25ex\left\vert #1 \right\vert\kern-0.25ex\right\vert\kern-0.25ex\right\vert}}

\title{Smoothing Estimates of the Vlasov-Poisson-Landau System}
\author{Dingqun DENG 
	\thanks{email: dingqdeng2-c@my.cityu.edu.hk, Department of Mathematics, City University of Hong Kong, ORCID: 0000-0001-9678-314X } }

\begin{document}

\maketitle

\begin{abstract}
	In this work, we consider the smoothing effect of Vlasov-Poisson-Landau system for both hard and soft potential. In particular, we prove that any classical solutions becomes immediately smooth with respect to all variables. We also give a proof on the global existence to Vlasov-Poisson-Landau system with optimal large time decay. These results give the regularity to Vlasov-Poisson-Landau system. The proof is based on the time-weighted energy method building upon the pseudo-differential calculus. 

	\paragraph{Keywords} Vlasov-Poisson-Landau system, regularity, smoothing effect.
\end{abstract}

\tableofcontents

\section{Introduction}


\paragraph{Model and Equation.}The Vlasov-Poisson-Landau system is an important physical model to describe the time evolution of charged dilute particles of two species (e.g. ions and electrons). We consider the Vlasov-Poisson-Landau system of two species in the whole space $\R^3$:
\begin{equation}\begin{aligned}\label{1}
	\partial_tF_+ + v\cdot\nabla_xF_+ +\frac{e_+}{m_+}E\cdot\nabla_vF_+ = Q(F_+,F_+) + Q(F_-,F_+),\\
	\partial_tF_- + v\cdot\nabla_xF_- -\frac{e_-}{m_-}E\cdot\nabla_vF_- = Q(F_-,F_-) + Q(F_+,F_-).
\end{aligned}
\end{equation}
Here $F_\pm(t,x,v)$ represents the density distribution for the ions (+) and electrons (-) respectively, at time $t\ge 0$ with position $x\in \R^3$ and velocity $v\in\R^3$. $e_\pm$, $m_\pm$ are the magnitude of their charges and masses.
The self-consistent electrostatic field is taken as $E(t,x) = -\nabla_x\phi$, with the electric potential $\phi$ given by 
\begin{align}\label{2}
	-\Delta_x\phi = \int_{\Rd}(e_+F_+-e_-F_-)\,dv, 
\end{align}
The initial data of the system is 
\begin{align}\label{3}
	F_\pm(0,x,v) = F_{\pm,0}(x,v). 
\end{align}
The bilinear collision term $Q(F,G)$ on the right hand side of \eqref{1} is given by 
\begin{align}
	Q(G_1,G_2)(v) = \frac{c_{12}}{m_1}\nabla_v\cdot\int_{\Rd}\Phi(v-v_*)\Big(\frac{G_1(v_*)\nabla_vG_2(v)}{m_1}-\frac{G_2(v)\nabla_vG_1(v_*)}{m_2}\Big)\,dv_*.
	\end{align} 
The Landau (Fokker-Planck) collision kernel $\Phi$, cf. \cite{Guo2002a}, is a non-negative symmetric matrix-valued function defined for $0\neq v\in\R^3$ as:
\begin{align}\label{5}
\Phi(v) = |v|^{\gamma+2}\Big(I-\frac{v\otimes v}{|v|^2}\Big).
\end{align}
	where $-3\le \gamma\le 1$ is a parameter determined by the interaction potential between particles. It's convenient to call hard potential when $\gamma+2\ge 0$ and soft potential when $-1\le\gamma+2<0$. The original Landau collision operator for the Coulombic interaction corresponds to the case $\gamma=-3$. The constant $c_{12} = 2\pi e^2_1e^2_2\ln\Lambda$, $\ln\Lambda=\ln(\frac{\lambda_D}{b_0})$, where $\lambda_D=(\frac{T}{4\pi n_ee^2})^{1/2}$ is the Debye shielding distance and $b_0=\frac{e^2}{3T_0}$ is the typical 'distance of closest approach' for a thermal particle \cite{Hilton1983}. 

For notational simplicity and without loss of generality, we normalize all constants in the Vlasov-Poisson-Landau system to be one. Now we linearize the Vlasov-Poisson-Landau system around a normalized Maxwellian 
\begin{align*}
\mu(v) = (2\pi)^{-3/2}e^{-\frac{|v|^2}{2}},
\end{align*}
with the standard perturbation $f(t,x,v)$ to $\mu$ as 
\begin{align*}
F_\pm = \mu + \mu^{1/2}f_\pm.
\end{align*}
Consider the vector $f=(f_+,f_-)$, the Vlasov-Poisson-Landau system for the perturbation takes the form:
\begin{equation}\label{7}
	\partial_tf_\pm + v\cdot\nabla_xf_\pm \pm \frac{1}{2}\nabla_x\phi\cdot vf_\pm  \mp\nabla_x\phi\cdot\nabla_vf_\pm \pm \nabla_x\phi\cdot v\mu^{1/2} - L_\pm f = \Gamma_{\pm}(f,f),
\end{equation}
\begin{equation}\label{8}
	-\Delta_x \phi = \int_{\Rd}(f_+-f_-)\mu^{1/2}\,dv,
\end{equation}
with initial data 
\begin{align}\label{9}
	f_\pm(0,x,v) = f_{\pm,0}(x,v). 
\end{align}
The linear operator $L=(L_+,L_-)$ and $\Gamma = (\Gamma_+,\Gamma_-)$ are gives as 
\begin{equation*}
	L_\pm f = \mu^{-1/2}\Big(2Q(\mu,\mu^{1/2}f_\pm) + Q(\mu^{1/2}(f_\pm+f_\mp),\mu)\Big),
\end{equation*}
\begin{equation}\label{9a}
	\Gamma_\pm(f,g) = \mu^{-1/2}\Big(Q(\mu^{1/2}f_\pm,\mu^{1/2}g_\pm) + Q(\mu^{1/2}f_\mp,\mu^{1/2}g_\pm)\Big).
\end{equation}

In this paper, we are going to establish the global existence and smoothing effect of the classical solutions to Cauchy problem \eqref{1}-\eqref{3} of the Vlasov-Poisson-Landau system near the global Maxwellian equilibrium. For global existence, \cite{Guo2002a} firstly investigate Landau system in a periodic box. Later, \cite{Guo2012} investigate Vlasov-Poisson-Landau system in a periodic box and \cite{Strain2013} prove the global existence in the whole space; see also \cite{Strain2007}. 
For smoothing effect for solution of Landau equation, one can refer to \cite{Chen2009} for classical solution and \cite{Henderson2019} for weak solution. The solution found in \cite{Guo2002a} becomes immediately smooth with respect to all variables. 

In this work, we will prove the smoothing effect directly from the initial data and require only weighted estimate on the solution for soft potential. It turns out that for $\gamma+2>0$, requiring only the same regularity assumption on initial data, the global solutions we found becomes immediately smooth for any positive time $t>0$. For $-1\le\gamma+2\le0$, we will require additionally $\|\<v\>^Cf\|_{L^2_{v,x}}<\infty$ on the solution, and the solution will become smooth immediately. Our approach is based on the time-weighted energy method building upon pseudo-differential calculus. We found the smoothness on $v$ from dissipation rate and smoothness on $x$ from Poisson bracket $\{v\cdot\nabla_x,\theta\}$, where $\theta$ is defined by \eqref{107}. The method in this work should be applicable to torus $\mathbb{T}^3_x$ case. 

\paragraph{Notations.}
Through the paper, $C$ denotes some positive constant (generally large) and $\lambda$ denotes some positive constant (generally small), where both $C$ and $\lambda$ may take different values in different lines. 
For any $v\in\Rd$, we denote $\<v\>=(1+|v|^2)^{1/2}$. For multi-indices $\alpha=(\alpha_1,\alpha_2,\alpha_3)$ and $\beta=(\beta_1,\beta_2,\beta_3)$, write 
\begin{align*}
	\partial^\alpha_\beta = \partial^{\alpha_1}_{x_1}\partial^{\alpha_2}_{x_2}\partial^{\alpha_3}_{x_3}\partial^{\beta_1}_{v_1}\partial^{\beta_2}_{v_2}\partial^{\beta_3}_{v_3}.
\end{align*}The length of $\alpha$ is $|\alpha|=\alpha_1+\alpha_2+\alpha_3$. 
The notation $a\approx b$ (resp. $a\gtrsim b$, $a\lesssim b$) for positive real function $a$, $b$ means there exists $C>0$ not depending on possible free parameters such that $C^{-1}a\le b\le Ca$ (resp. $a\ge C^{-1}b$, $a\le Cb$) on their domain. $\mathscr{S}$ denotes the Schwartz space. $\Re (a)$ means the real part of complex number $a$. $[a,b]=ab-ba$ is the commutator between operators. $\{a(v,\eta),b(v,\eta)\} =  \partial_\eta a_1\partial_va_2 - \partial_va_1\partial_\eta a_2$ is the Poisson bracket. $\Gamma=|dv|^2+|d\eta|^2$ is the admissible metric and $S(m)=S(m,\Gamma)$ is the symbol class. 
For pseudo-differential calculus, we write $(x,v)\in \Rd\times\Rd$ to be the space-velocity variable and $(y,\eta)\in \Rd\times\Rd$ to be the corresponding variable in frequency space (the variable after Fourier transform). We introduce a weight function of $v$ as 
\begin{equation}\label{10w}
w=w(v)=\left\{\begin{aligned}
&\<v\>,\quad\quad\text{ if }\gamma+2\ge 0,\\
&\<v\>^{-\gamma},\text{ if }-1\le\gamma+2<0. 
\end{aligned}\right.
\end{equation}

(i) As in \cite{Guo2012,Guo2003a}, the null space of $L$ is given by 
\begin{equation*}
	\ker L  = \text{span}\Big\{(1,0)\mu^{1/2},(0,1)\mu^{1/2},(1,1)v_i\mu^{1/2}(1\le i\le 3),(1,1)|v|^2\mu^{1/2}\Big\}. 
\end{equation*}
We denote $\PP$ to be the orthogonal projection from $L^2_v\times L^2_v$ onto $\ker L$, which is defined by 
\begin{equation}\label{10}
	\P f = \Big(a_+(t,x)(1,0)+a_-(t,x)(0,1)+v\cdot b(t,x)(1,1)+(|v|^2-3)c(t,x)(1,1)\Big)\mu^{1/2},
\end{equation}or equivalently by 
\begin{equation*}
	\PP f = \Big(a_\pm(t,x)+v\cdot b(t,x)+(|v|^2-3)c(t,x)\Big)\mu^{1/2}.
\end{equation*}
Then for given $f$, one can decompose $f$ uniquely as 
\begin{equation*}
	f = \P f+ (\I-\P)f. 
\end{equation*}
The function $a_\pm,b,c$ are given by 
\begin{align*}
	a_\pm &= (\mu^{1/2},f_\pm)_{L^2_v} = (\mu^{1/2},\PP f)_{L^2_v},\\
	b_j&= \frac{1}{2}(v_j\mu^{1/2},f_++f_-)_{L^2_v} = (v_j\mu^{1/2},\PP f)_{L^2_v},\\
	c&=\frac{1}{12}((|v|^2-3)\mu^{1/2},f_++f_-)_{L^2_v} = \frac{1}{6}((|v|^2-3)\mu^{1/2},\PP f)_{L^2_v}. 
\end{align*}

(ii) To describe the behavior of linearized Landau collision operator, we define the following velocity weighted norms; cf. \cite{Guo2002a}:
\begin{align*}
|g|^2_{\sigma,l} &= \sum_{i,j=1}^3\int_{\R^3}w^{2l}\Big(\sigma^{ij}\partial_{v_i}f\partial_{v_j}\bar{f}+\sigma_{ij}v_iv_j|g|^2\Big)\,dv,\\
\|g\|^2_{\sigma,l} &= \sum_{i,j=1}^3\int_{\R^3\times\R^3}w^{2l}\Big(\sigma^{ij}\partial_{v_i}f\partial_{v_j}\bar{f}+\sigma_{ij}v_iv_j|g|^2\Big)\,dvdx,
\end{align*}
where $\sigma^{ij}=\sigma^{ij}(v)=\int_{\R^3}\Phi^{ij}(v-v_*)\mu(v_*)\,dv_*$, with $\Phi$ defined by \eqref{5}. For any vector-valued function $h(v)=(h_1(v),h_2(v),h_3(v))$, we define projection 
\begin{align}
P_vh_i = \sum^3_{j=1}h_jv_j\frac{v_i}{|v|^2},\ j\in\{1,2,3\}.
\end{align}Then by \cite[Lemma 3 and Corollary 1]{Guo2002a}, we have 
\begin{align}\label{13b}
|g|^2_{\sigma,l}\approx |w^l\<v\>^{\gamma/2}P_v\partial_ig|^2_{L^2_v}+|w^l\<v\>^{(\gamma+2)/2}(I-P_v)\partial_ig|^2_{L^2_v}+|w^l\<v\>^{(\gamma+2)/2}g|^2_{L^2_v}.
\end{align}The upper bound was not written down in \cite{Guo2002a}, but the proof is the same; see also \cite[Lemma 5]{Strain2007}.

We write $|\cdot|_{L^2_v}$ to be the $L^2_v(\R^3)$ norm and $\|\cdot\|_{L^2_{v,x}}$ to be the $L^2_{v,x}(\R^3\times\R^3)$ norm. $L^2(B_C)$ is the $L^2_v$ space on Euclidean ball $B_C$ of radius $C$ at the origin. For usual Sobolev space, we will use notation 
\begin{align*}
	\|f\|_{H^k_vH^m_x} = \sum_{|\beta|\le k,|\alpha|\le m}\|\partial^\alpha_\beta f\|_{L^2_{v,x}},
\end{align*}for $k,m\ge 0$. 
We also define the standard velocity-space mixed Lebesgue space $Z_1=L^2(\R^3_v;L^1(\R^3_x))$ with the norm
\begin{equation*}
	\|f\|_{Z_1} = \Big\|\|f\|_{L^1_x}\Big\|_{L^2_v}.
\end{equation*}
In this paper, we write Fourier transform on $x$ as 
\begin{align*}
	\widehat{f}(y) = \int_{\R^3}f(x)e^{-ix\cdot y}\,dx. 
\end{align*}

For pseudo-differential calculus, one may refer to the appendix as well as \cite{Lerner2010} for more information. Let $\Gamma=|dv|^2+|d\eta|^2$ be an admissible metric. 
Define
\begin{align}\label{11a}
	\tilde{a}(v,\eta):=\<v\>^\gamma(1+|\eta|^2+|v|^2)+K_0\<v\>^{\gamma+2}
\end{align}to be a $\Gamma$-admissible weight, where $K_0>0$ is chosen as the following. 
Applying Lemma 2.1 and 2.2 in \cite{Deng2020a}, there exists $K_0>0$ such that the Weyl quantization $\tilde{a}^w:H(\tilde{a}c)\to H(c)$ and $(\tilde{a}^{1/2})^w:H(\tilde{a}^{1/2}c)\to H(c)$ are invertible, with $c$ being any $\Gamma$-admissible metric. The weighted Sobolev space $H(c)$ is defined by \eqref{sobolev_space}. 
By using the equivalence \eqref{13b}, we have 
\begin{align}\label{144}
	|(\tilde{a}^{1/2})^wf|_{L^2_v}&\lesssim |\<v\>^{\frac{\gamma+2}{2}}f|_{L^2_v}+|\<v\>^{\frac{\gamma}{2}}\nabla_vf|_{L^2_v}
	\lesssim|f|_{\sigma,0},
\end{align}for any suitable $f$. 
In order to extract the smoothing effect on $x$, we define a symbol $\tilde{b}$ by 
\begin{align}\label{11}
\tilde{b}(v,y) = \<v\>^{l_0}|y|^{\delta_1}, 
\end{align}
where $l_0\in\R$, $\delta_1=\delta_1(\alpha,\beta)>0$ are defined by \eqref{106b}.

\paragraph{Main results.}
To state the result of the paper, we let $K\ge 0$ to be the total order of derivatives on $v,x$.
In order to extract the smoothing effect, we define a useful coefficient  
\begin{equation*}
\psi_k=\left\{\begin{aligned}
	1, \text{  if $k\le 0$},\\
	\psi^k, \text{ if $k> 0$}, 
\end{aligned}\right.
\end{equation*}
where $\psi = 1$ in Theorem \ref{main1} and $\psi = t^N$ with $N=N(\alpha,\beta)>0$ large in Section \ref{sec5} and Theorem \ref{main2}.
When considering $\psi=t^N$ in proving regularity, we always assume $0\le t\le 1$, since regularity is a local-in-time property. In any cases, we have $\psi\le 1$. 
Corresponding to given $f=f(t,x,v)$, we introduce the instant energy functional $\E_{K,l}(t)$
 and the instant high-order energy functional $\E^{h}_{K,l}(t)$ to be functionals
satisfying the equivalent relation
\begin{align}\label{Defe}
	\E_{K,l}(t)\notag &\approx \sum_{|\alpha|\le K}\|\psi_{|\alpha|-3}\partial^\alpha E\|^2_{L^2_x}+\sum_{|\alpha|\le K}\|\psi_{|\alpha|-3}\partial^\alpha\P f\|^2_{L^2_{v,x}}\\
	&\qquad+\sum_{\substack{|\alpha|+|\beta|\le K}}\|\psi_{|\alpha|+|\beta|-3}w^{l-|\alpha|-|\beta|}\partial^\alpha_\beta(\I-\P) f\|^2_{L^2_{v,x}}.
\end{align}
\begin{align}\label{Defeh}
	\E^h_{K,l}(t)\notag &\approx \sum_{|\alpha|\le K}\|\psi_{|\alpha|-3}\partial^\alpha E\|^2_{L^2_x}+\sum_{1\le|\alpha|\le K}\|\psi_{|\alpha|-3}\partial^\alpha\P f\|^2_{L^2_{v,x}}\\
	&\qquad+\sum_{\substack{|\alpha|+|\beta|\le K}}\|\psi_{|\alpha|+|\beta|-3}w^{l-|\alpha|-|\beta|}\partial^\alpha_\beta(\I-\P) f\|^2_{L^2_{v,x}}.
\end{align}
Also, we define the dissipation rate functional $\D_{K,l}$ by 
\begin{align}\label{Defd}
	\D_{K,l}(t) \notag&= \sum_{|\alpha|\le K-1}\|\psi_{|\alpha|-3}\partial^\alpha E\|^2_{L^2_x}+\sum_{1\le|\alpha|\le K}\|\psi_{|\alpha|-3}\partial^\alpha\P f\|^2_{L^2_{v,x}}\\
	&\qquad+\sum_{\substack{|\alpha|+|\beta|\le K}}\|\psi_{|\alpha|+|\beta|-3}\partial^\alpha_\beta(\I-\P) f\|^2_{\sigma,l-|\alpha|-|\beta|}.
\end{align}
Here $E=E(t,x)$ is determined by $f(t,x,v)$ in terms of $E=-\nabla_x\phi$ and \eqref{8}. 
 The main result of this paper is stated as follows.

\begin{Thm}\label{main1}
	Let $\gamma+2\ge -1$, $K\ge 3$, $\psi=1$. 
	Assume $l_0\ge K$ satisfies that 
	\begin{equation}\label{20b}l_0\ge\left\{\begin{aligned}
			& \frac{3(\gamma+2)}{-4\gamma}+3,\  \text{ if }-1\le\gamma+2<0,\\
			& 2(\gamma+2)+3,\quad\text{ if }\gamma+2\ge0.
		\end{aligned}\right.
	\end{equation}Define 
\begin{equation}\label{88}
X(t) = \left\{\begin{aligned}
	&\sup_{0\le\tau\le t}(1+\tau)^{3/2}\E_{K,l_0}(\tau)+\sup_{0\le\tau\le t}(1+\tau)^{5/2}\E^h_{K,l_0}(\tau),\text{ if }\gamma+2\ge 0,\\
	\sup_{0\le\tau\le t}\E_{K,l_0+l_1}&(\tau)+\sup_{0\le\tau\le t}(1+\tau)^{3/2}\E_{K,l_0}(\tau)+\sup_{0\le\tau\le t}(1+\tau)^{3/2+p}\E^h_{K,l_0}(\tau),\text{ if }-1\le\gamma+2< 0,
\end{aligned}\right.
\end{equation}and
\begin{equation*}
\epsilon_0 = (\E_{K,l_0+l_1}(0))^{1/2}+\|w^{l_2}f_0\|_{Z_1}+\|E_0\|_{L^1},
\end{equation*}
where $E_0(x)=E(0,x)$, $l_1=\frac{5(\gamma+2)}{4(1-p)\gamma}$, $l_2=\frac{5(\gamma+2)}{4\gamma}$ for soft potential $-1\le\gamma+2< 0$ and $l_1=l_2=0$ for hard potential $\gamma+2\ge 0$.
Let $f_0(x,v)=(f_{0,+}(x,v),f_{0,-}(x,v))$ satisfying $F_\pm(0,x,v)=\mu(v)+(\mu(v))^{1/2}f_{0,\pm}(x,v)\ge 0$.

	If $\epsilon_0$	is sufficiently small, then there exists a unique global solution $f(t,x,v)$ to the Cauchy problem \eqref{7}-\eqref{9} of the Vlasov-Poisson-Boltzmann system such that $F_\pm(t,x,v)=\mu(v)+(\mu(v))^{1/2}f_\pm(t,x,v)\ge 0$ and 
	\begin{equation}\label{15a}
	X(t)\lesssim \epsilon_0^2.
	\end{equation}
	for any $t\ge 0$. 
\end{Thm}
This gives the global existence to the Vlasov-Poisson-Landau system with the optimal large time decay for $\gamma\ge 3$, which generalizes the index $\gamma=-3$ in \cite{Strain2013}.
Notice that in \eqref{88}, the high-order energy functional $\E^h_{K,l}$ has time decay rate $(1+t)^{-5/2}$ for hard potential and $(1+t)^{-3/2-p}$ for soft potential. $\E_{K,l}$ has time decay rate $(1+t)^{-3/2}$. They are all optimal as in the Boltzmann equation case \cite{Strain2012} and the Vlasov-Maxwell-Boltzmann system case \cite{Duan2011}. 
Let $\delta_0>0$ and the $a$ $priori$ assumption to be 
\begin{align}
	\label{priori}\sup_{0\le t\le T}X(t)\le \delta_0. 
\end{align}
Our main approach is to prove the following closed $a$ $priori$ estimate:
\begin{align*}
	X(t)\lesssim \epsilon^2_0+X^{3/2}(t)+X^2(t).
\end{align*}

Next Theorem concerns the smoothing effect. 
\begin{Thm}\label{main2}Let $\gamma+2\ge -1$, $0<\tau<T\le \infty$. 
For any $K\ge 3$ and multi-indices $|\alpha|+|\beta|\le K$, assume $\psi=t^N$ with $N>0$ large when $|\alpha|\le 3$ and $N=N(\alpha)>0$ defined by \eqref{106a} when $|\alpha|>3$. 
Fix $l\ge K$ and let $(f,E)$ be the solution to \eqref{7}-\eqref{9} satisfying that
\begin{align}\label{15aa}
\epsilon_1 = (\E_{3,l}(0))^{1/2}
\end{align}
is sufficiently small. If $-1\le\gamma+2\le0$, we assume additionally that any for $C>0$, 
 \begin{align}
 \sup_{0\le t\le T}\|\<v\>^{C}f(t)\|_{L^2_{v,x}} <\infty.
 \end{align}
 Then the followings holds true.

(1) For $|\alpha|+|\beta|\le K$, $T<\infty$,
\begin{align}\label{19a}
\sup_{\tau\le t\le T}\|w^{l-|\alpha|-|\beta|}\partial^\alpha_\beta f\|^2_{L^2_{v,x}}\le \epsilon^2_1C_{\tau,T}<\infty,
\end{align}
where $C_{\tau,T}>0$ depends on $\tau,T$.
Moreover, if additionally  
\begin{align*}
\sup_{l^*\ge 3}\E_{3,l^*}(0)
\end{align*}is sufficiently small, then for any $l_0\ge K$, $|\alpha|+|\beta|\le K$, $k\ge 0$, $T<\infty$, we have 
\begin{align}\label{19b}
\sup_{\tau\le t\le T}\|w^{l_0-|\alpha|-|\beta|}\partial^\alpha_\beta\partial^{k}_tf\|^2_{L^2_{v,x}}\le C_{\tau,T,k}<\infty,
\end{align}where $C_{\tau,T,k}$ is a constant depending on $\tau$, $T$, $k$. 
Consequently, $f\in C^\infty(\R^+_t;C^\infty(\R^3_x;\mathscr{S}(\R^3_v)))$. 

(2) If additionally, the solution $(f,E)$ satisfies that 
\begin{align}\label{19aa}
	\epsilon_0 = (\E_{3,l_0+l_1}(0))^{1/2}+\|w^{l_2}f_0\|_{Z_1}+\|E_0\|_{L^1_x},
\end{align}
is sufficiently small, where $l_0$ is defined by \eqref{20b}, $l_1=\frac{5(\gamma+2)}{4(1-p)\gamma}$, $l_2=\frac{5(\gamma+2)}{4\gamma}$ for soft potential $-1\le\gamma+2< 0$ and $l_1=l_2=0$ for hard potential $\gamma+2\ge 0$. Then the constants in (1) can be chosen independent of $T$ and $T$ can take the value $\infty$.
	 
\end{Thm}
This result shows that the smoothing effect proven (for example in \cite{Chen2009}) for the Landau equation extends to Vlasov-Poisson-Landau system for all variables $t,x,v$. 
This behavior was also observed for Boltzmann equation \cite{Alexandre2010,Chen2012} and Vlasov-Poisson-Boltzmann system \cite{Deng2020c}. Notice that the smoothing property holds uniformly in Theorem \ref{main2} when the time goes to infinity.

Now we illustrate several technical points in the proof of Theorem \ref{main1} and Theorem \ref{main2}. 
We use the velocity weight $w^{l-|\alpha|-|\beta|}$ to deal with the term $v\cdot\nabla_x\phi f$ when doing estimate on  
\begin{align*}
	\big(\sum_{\alpha_1\le\alpha}v\cdot\nabla_x\partial^{\alpha_1}\phi \partial^{\alpha-\alpha_1} f,e^{\pm\phi}w^{2l-2|\alpha|-2|\beta|}\partial^\alpha f\big)_{L^2_{v,x}},
\end{align*}where $|\alpha|\le K$. By using the term corresponding to $v\cdot\nabla_xf$, the case $\alpha_1=0$ will be eliminated as in \cite{Guo2012}. This is why we add the term $e^{\pm\phi}$. The case $\alpha_1\neq 0$ can be bounded by $\E^{1/2}_{K,l}\D_{K,l}$ due to the weight $w^{l-|\alpha|-|\beta|}$, since $vw^{l-|\alpha|-|\beta|}$ can be absorbed in $w^{l-|\alpha-\alpha_1|-|\beta|}$. 
For the term $\nabla_x\phi\cdot\nabla_vf$, one will need to estimate 
\begin{align*}
	\big(\sum_{\alpha_1\le\alpha}\nabla_x\partial^{\alpha_1}\phi \partial^{\alpha-\alpha_1}\nabla_v f,e^{\pm\phi}w^{2l-2|\alpha|-2|\beta|}\partial^\alpha f\big)_{L^2_{v,x}}.
\end{align*}
If $\alpha_1=0$, we can use integration by parts to move $\nabla_v$ to the weight $w^{2l-2|\alpha|-2|\beta|}$ and it becomes a lower order term, while if $\alpha_1\neq 0$, this term will transfer one derivative from $x$ to one derivative on $v$. The total order on the first $f$ is less or equal to $K$ and hence can be controlled by our energy functional $\E_{K,l}$ and $\D_{K,l}$. As in \cite{Guo2012,Duan2013}, one has to bound the term 
\begin{align*}
	\|\partial_t\phi\|_{L^\infty_x}\E_{K,l},
\end{align*}which cannot be absorbed by the energy dissipation norm. But observing that $\|\partial_t\phi\|_{L^\infty_x}$ is bounded by the high-order energy functional $\E^h_{K,l}$ and hence integrable as shown in \cite{Guo2012}, one can use the Gronwall's inequality to close the $a$ $priori$ estimate. 
The next technical point concerns the choice of $\psi=t^N$ in Theorem \ref{main2} and the usage of $\tilde{b}$, $\psi_{|\alpha|+|\beta|-3}$ is Section \ref{sec5}. 
Firstly, denote \begin{equation*}
\psi_k=\left\{\begin{aligned}
1, \text{  if $k\le 0$},\\
\psi^k, \text{ if $k> 0$}. 
\end{aligned}\right.
\end{equation*} Whenever $|\alpha|+|\beta|> 3$, $\psi_{|\alpha|+|\beta|-3} = t^{N(|\alpha|+|\beta|-3)}$ is equal to $0$ at $t=0$ and hence, $\E_{K,l}(0)=\E_{3,l}(0)$. Plugging this into energy estimate, the higher order derivatives are canceled at $t=0$ and one can control the higher order instant energy by lower order initial data. Then one can deduce the smoothing effect locally in time. By using the global energy control obtained in Theorem \ref{main1}, the local-in-time regularity becomes uniform when time goes to infinity. Notice that we use $-3$ to eliminate the index arising from Sobolev embedding $\|\cdot\|_{H^1_vL^\infty_x}\lesssim \|\cdot\|_{H^1_vH^2_x}$, where the latter has derivatives of order three. However, after adding $\psi_{|\alpha|+|\beta|-3}$, one need to control the term 
\begin{align}\label{22}
	\big(\partial_t(\psi_{|\alpha|+|\beta|-3})\partial^\alpha_\beta f,e^{\pm\phi}w^{2l-2|\alpha|-2|\beta|}\partial^\alpha_\beta f\big)_{L^2_{v,x}}.
\end{align}
This is where we make use of $\tilde{b}$. By choosing $N=N(\alpha)$ properly, one has interpolation 
\begin{align*}
\psi_{|\alpha|-3-\frac{1}{2N}}
&\lesssim \delta\,\tilde{b}^{1/2}+C_{0,\delta}\<v\>^{\frac{-l_0|\alpha|}{\delta_1}}|y|^{-|\alpha|}.
\end{align*} 
The first term can be absorbed by instant energy involving $\tilde{b}^{1/2}$ while the second term eliminate $\alpha$ derivatives on $x$:
 Applying a similar interpolation on $v$ with $\tilde{a}$, we can control \eqref{22} by a high-order term and a algebraic decay term if $-1\le\gamma+2\le 0$:
\begin{align*}
	\delta^2\|\psi_{|\alpha|+|\beta|-3}\tilde{b}^{1/2}w^{l-|\alpha|-|\beta|}(\partial^\alpha_\beta f)^\wedge(v,y)\|^2_{L^2_{v,y}}+\delta^2\D_{K,l}+C_\delta\|\<v\>^{C_{K,l}}f\|^2_{L^2_{v,x}},
\end{align*}
while for $\gamma+2>0$, we have 
\begin{align*}
	\delta^2\|\psi_{|\alpha|+|\beta|-3}\tilde{b}^{1/2}w^{l-|\alpha|-|\beta|}(\partial^\alpha_\beta f)^\wedge(v,y)\|^2_{L^2_{v,y}}+\delta^2\D_{K,l}+C_\delta\E_{K,l}.
\end{align*}
Defining $\theta$ by \eqref{107}, using the equation \eqref{7}-\eqref{9} and Poisson bracket $\{v\cdot y,\theta\}$, one can control the high-order term $\|\psi_{|\alpha|+|\beta|-3}\tilde{b}^{1/2}w^{l-|\alpha|-|\beta|}(\partial^\alpha_\beta f)^\wedge(v,y)\|^2_{L^2_{v,y}}$ by using functional $\E_{K,l}$ and $\D_{K,l}$, where $\delta_1$ in $\tilde{b}$ should be chosen properly, i.e. Lemma \ref{Lem53}. Hence, we can obtain a closed energy estimate locally. 
Here, when dealing with case $-1\le\gamma+2\le 0$, there occurs a algebraic decay term in $v$: $\|\<v\>^{C_{K,l}}f\|_{L^2_{v,x}}$ and we will need to assume such norms are bounded initially, as observed in the Landau equation and Boltzmann equation; cf. \cite{Chen2009,Chen2012}. 
After obtaining the local regularity, we can combine it with the global energy control from global existence Theorem \ref{main1} and one can find the regularity globally in time.

The rest of the paper is arranged as follows. In Section 2, we present some estimates on $L, \Gamma$, and some tricks in energy estimates. In Section 3, we present macroscopic estimate. In Section 4, we prove the global existence to Vlasov-Poisson-Landau system. In Section 5, we give the proof for regularity. 
The Appendix is devoted to pseudo-differential calculus.

\section{Preliminaries}

In this section, we list several basic lemmas corresponding to linearized Landau collision operator $L_\pm$ and the bilinear Landau collision operator $\Gamma_\pm$.

In the following lemma, we collect some known basic estimates for linearized Landau operator $L_\pm$ from \cite{Guo2002a,Guo2012}.  
\begin{Lem}\label{lemmaL}For any $l\in\R$, multi-indices $\alpha,\beta$, we have the followings. 
	
	(i) It holds that 
	\begin{equation*}
		\sum_\pm(-L_\pm g,g_\pm)_{L^2_v}\gtrsim |(\I-\P)g|^2_{\sigma,0}.
	\end{equation*}

(ii) There exists $C>0$ such that 
\begin{align*}
	-\sum_\pm(w^{2l}L_\pm g,g_\pm)_{L^2_v}\gtrsim |g|^2_{\sigma,l}-C|g|^2_{L^2_v(B_C)}.
\end{align*}

(iii) For any $\eta>0$, $|\beta|>0$,
\begin{align*}
	-\sum_\pm(w^{2l-2|\alpha|-2|\beta|}\partial^\alpha_\beta L_\pm g,\partial^\alpha_\beta g_\pm)_{L^2_v}&\gtrsim |\partial^\alpha_\beta g|^2_{\sigma,l-|\alpha|-|\beta|} - \eta\sum_{|\beta_1|\le|\beta|}|\partial^\alpha_{\beta_1}g|^2_{\sigma,l-|\alpha|-|\beta_1|}-C_\eta|\partial^\alpha g|^2_{L^2(B_{C_\eta})}.
\end{align*}
\end{Lem}
\begin{proof}
	To apply the previous work, we decompose $L_\pm f=A_\pm f+K_\pm f$ with 
	\begin{align*}
	A_\pm f &= 2\mu^{-1/2}Q(\mu,\mu^{1/2}f_\pm),\\
	K_\pm f &= \mu^{-1/2}Q(\mu^{1/2}(f_\pm+f_\mp),\mu).
	\end{align*}
	Then the proof of (i) is the same as \cite[Lemma 5]{Guo2002a}. The proof of (ii) is the same as \cite[Lemma 3 and Lemma 5]{Guo2002a} and the proof of (iii) is the same as \cite[Lemma 6]{Guo2002a}. 
\end{proof}

In order to obtain a suitable norm estimate of $\Gamma$ on $x$. We will use the following Sobolev inequalities and Gagliardo-Nirenberg interpolation inequality frequently throughout our analysis. 
\begin{Lem}
	For any $u,v\in H^2_x$, we have 
	\begin{align*}
		\|u\|_{L^\infty}&\lesssim \|\nabla_xu\|^{1/2}\|\nabla^2_xu\|^{1/2}\lesssim \|\nabla_xu\|_{H^1},\\
		\|u\|_{L^6}&\lesssim\|\nabla_xu\|_{L^2},\\
		\|u\|_{L^3}&\lesssim \|u\|_{H^1}.
	\end{align*}Consequently,
\begin{align}
		\label{13}
		\|uv\|_{L^2_x}&\lesssim \min\{\|\nabla_xu\|_{H^1_x}\|v\|_{L^2_x}, \|\nabla_xu\|_{L^2_x}\|v\|_{H^1_x}\}.
	\end{align}
\end{Lem}
Inequality \eqref{13} shows that we can always give at least one $x$ derivative on $u$ and give the other one $x$ derivative to $u$ or $v$ by our wish. 

The next lemma concerns the estimates on the nonlinear collision operator $\Gamma_\pm$, which comes from \cite[Proposition 2.2 and Lemma 5.1]{Strain2013}.
\begin{Lem}\label{23a}
	For any $l\in\R$, we have the upper bound 
	\begin{align}\label{12}
		|(w^{2l}\partial^\alpha_\beta\Gamma_\pm(f,g),\partial^\alpha_\beta h)_{L^2_{v,x}}|\lesssim\sum_{\alpha_1\le\alpha,\bar{\beta}\le\beta_1\le\beta}|w^{-C}\partial^{\alpha_1}_{\bar{\beta}}f|_{L^2_v}|\partial^{\alpha-\alpha_1}_{\beta-\beta_1}g|_{\sigma,l}|\partial^\alpha_\beta h|_{\sigma,l},
	\end{align}
where $C>0$ can be arbitrarily large. Also, 
\begin{align}\label{12a}
	|w^l\tilde{\Gamma}(f,g)|_{L^2_v}\lesssim |f|_{H^1_v}|w^{l}\<v\>^{{2\gamma+4}}g|_{H^1_v}+|f|_{L^2_v}|w^{l}\<v\>^{{2\gamma+4}}\nabla^2_vg|_{L^2_v}+|f|_{H^2_v}|w^{l}\<v\>^{{2\gamma+4}}g|_{L^2_v}.
\end{align}
Define $l_3=1$ for $\gamma+2\ge0$ and $l_3=-\gamma$ for $-1\le\gamma+2<0$. Then, for $l_0\ge \max\{l+3+\frac{2\gamma+4}{l_3},3\}$, we have  
\begin{align}\label{12b}
	\|w^l{\Gamma}(f,f)\|_{L^2_{v}H^1_x}+\|w^l{\Gamma}(f,f)\|_{Z_1}\lesssim\E_{3,l_0}.
\end{align}
\end{Lem}
\begin{proof}
	The proof of \eqref{12} is the same as \cite[Proposition 2.2]{Strain2013}. Although \cite[Proposition 2.2]{Strain2013} uses $\gamma=-3$, the proof is also applicable to any $\gamma\ge-3$; see also \cite[Theorem 3]{Guo2002a}. To prove \eqref{12a}, we will use a similar argument as in \cite[Lemma 5.1]{Strain2013}. Using \eqref{9a}, it suffices to estimate 
	\begin{align*}
		\tilde{\Gamma}(f,g):&=\partial_{v_i}\big((\Phi^{ij}*(\mu^{1/2}f))\partial_{v_j}g\big)-(\Phi^{ij}*(v_i\mu^{1/2}f))\partial_{v_j}g\\
		&\quad-\partial_{v_i}\big((\Phi^{ij}*(\mu^{1/2}\partial_{v_j}f))g\big)+(\Phi^{ij}*(v_i\mu^{1/2}\partial_{v_j}f))g.
	\end{align*}
This expansion can be found in \cite[Lemma 1]{Guo2002a}. Above and below we implicitly sum over indices $i$ and $j$ when they are repeated. Since $|\Phi^{ij}|\lesssim |v|^{\gamma+2}$, we have 
\begin{align*}
	\Phi^{ij}*(\mu^{\lambda}f)\lesssim ((\Phi^{ij})^2*\mu^\lambda)^{1/2}|\mu^{\lambda/2}f|_{L^2_v}\lesssim\<v\>^{2\gamma+4}|\mu^{\lambda/2}f|_{L^2_v}.
\end{align*}
Thus recalling \eqref{10w}, we have 
\begin{align*}
		|w^l\tilde{\Gamma}(f,g)|_{L^2_v}\lesssim |f|_{H^1_v}|w^{l}\<v\>^{{2\gamma+4}}g|_{H^1_v}+|f|_{L^2_v}|w^{l}\<v\>^{{2\gamma+4}}\nabla^2_vg|_{L^2_v}+|f|_{H^2_v}|w^{l}\<v\>^{{2\gamma+4}}g|_{L^2_v}.
\end{align*}
Now we let $l_0\ge \max\{l+3+\frac{2\gamma+4}{l_3},3\}$. Then by \eqref{13}, we have 
\begin{align*}
	\|w^l\tilde{\Gamma}(f,f)\|_{L^2_{v,x}}&\lesssim\|f\|_{H^1_vH^1_x}\|w^{l+\frac{2\gamma+4}{l_3}}f\|_{H^1_vH^1_x}+\|f\|_{L^2_vH^2_x}\|w^{l+\frac{2\gamma+4}{l_3}}f\|_{H^2_vL^2_x}+\|f\|_{H^2_vL^2_x}\|w^{l+\frac{2\gamma+4}{l_3}}f\|_{L^2_vH^2_x}\\
	&\lesssim \E_{3,l_0}.
\end{align*}
Similarly, we have 
\begin{align*}
	|w^l\nabla_x\tilde{\Gamma}(f,g)|_{L^2_v}
	&\lesssim |\nabla_xf|_{H^1_v}|w^{l+\frac{2\gamma+4}{l_3}}g|_{H^1_v}+|\nabla_xf|_{L^2_v}|w^{l+\frac{2\gamma+4}{l_3}}\nabla^2_vg|_{L^2_v}+|\nabla_xf|_{H^2_v}|w^{l+\frac{2\gamma+4}{l_3}}g|_{L^2_v}\\
	&\qquad+|f|_{H^1_v}|w^{l+\frac{2\gamma+4}{l_3}}\nabla_xg|_{H^1_v}+|f|_{L^2_v}|w^{l+\frac{2\gamma+4}{l_3}}\nabla_x\nabla^2_vg|_{L^2_v}+|f|_{H^2_v}|w^{l+\frac{2\gamma+4}{l_3}}\nabla_xg|_{L^2_v},
\end{align*}
and by \eqref{13} and $l_0\ge \max\{l+3+\frac{2\gamma+4}{l_3},3\}$, 
\begin{align*}
	&\quad\,\|w^l\nabla_x\tilde{\Gamma}(f,f)\|_{L^2_{v,x}}\\&\lesssim\|f\|_{H^1_vH^2_x}\|w^{l+\frac{2\gamma+4}{l_3}}f\|_{H^1_vH^2_x}+\|f\|_{L^2_vH^3_x}\|w^{l+\frac{2\gamma+4}{l_3}}f\|_{H^2_vH^1_x}+\|f\|_{H^2_vH^1_x}\|w^{l+\frac{2\gamma+4}{l_3}}g\|_{L^2_vH^3_x}\\
	&\lesssim \E_{3,l_0}.
\end{align*}
Finally, using \eqref{12a}, we estimate $\|\tilde{\Gamma}(f,f)\|_{Z_1}$ as follows:
\begin{align*}
	\|w^l\tilde{\Gamma}(f,f)\|_{Z_1}&\lesssim\|f\|_{H^1_vL^2_x}\|w^{l+\frac{2\gamma+4}{l_3}}f\|_{H^1_vL^2_x}+\|f\|_{L^2_vL^2_x}\|w^{l+\frac{2\gamma+4}{l_3}}f\|_{H^2_vL^2_x}+\|f\|_{H^2_vL^2_x}\|w^{l+\frac{2\gamma+4}{l_3}}f\|_{L^2_vL^2_x}\\
	&\lesssim \E_{3,l_0}.
\end{align*}
This completes the proof of Lemma \ref{23a}.

\end{proof}

The following Corollary gives the behavior of nonlinear terms in Vlasov-Poisson-Landau system. 
\begin{Coro}\label{Coro1} 
	Assume $l_0\ge 3$ satisfies that 
	\begin{equation}\label{31b}l_0\ge\left\{\begin{aligned}
			& \frac{3(\gamma+2)}{-4\gamma}+3,\  \text{ if }-1\le\gamma+2<0,\\
			& 2(\gamma+2)+3,\quad\text{ if }\gamma+2\ge0.
		\end{aligned}\right.
	\end{equation}
	Then, there exists $l_*>\frac{5(\gamma+2)}{4\gamma}$ when $-1\le\gamma+2<0$ and $l_*=0$ when $\gamma+2\ge0$ such that 
	\begin{align*}
	\|w^{l_*}g_\pm\|_{Z_1}+\|w^{l_*}g_\pm\|_{L^2_{v}H^1_x}\lesssim \E_{3,l_0},
	\end{align*}where $g_\pm = \pm\nabla_x\phi\cdot\nabla_vf_\pm\mp\frac{1}{2}\nabla_x\phi\cdot vf_\pm+\Gamma_\pm(f,f)$. 
\end{Coro}
\begin{proof}
Let $l_3=1$ for $\gamma+2\ge0$ and $l_3=-\gamma$ for $-1\le\gamma+2<0$. The estimate on $\Gamma_\pm(f,f)$ has already been proved by \eqref{12b} and the restriction is $l_0\ge l_*+3+\frac{2\gamma+4}{l_3}$ and $l_0\ge 3$.
For the left terms, we have 
	\begin{align*}
		\|w^{l_*}\nabla_x\phi\cdot\nabla_vf_\pm\|_{Z_1}&\lesssim \|\nabla_x\phi\|_{L^2_x}\|w^{l_*}\nabla_vf\|_{L^2_{v,x}}\lesssim\E_{3,l_0},\\
		\|w^{l_*}v\cdot\nabla_x\phi f_\pm\|_{Z_1}&\lesssim\|\nabla_x\phi\|_{L^2_x}\|w^{l_*}vf_\pm\|_{L^2_{v,x}}\lesssim\E_{3,l_0},
	\end{align*}whenever $l_0\ge l_*+\max\{1/l_3,1\}=l_*+1$. 
	By \eqref{13}, 
	\begin{align*}
		\|w^{l_*}\nabla_x(\nabla_x\phi\cdot\nabla_vf_\pm)\|_{L^2_{v,x}}&\lesssim \|\nabla_x\phi\|_{H^2_x}\|w^{l_*}f_\pm\|_{H^1_{v}H^1_x}\lesssim\E_{3,l_0}\\
		\|w^{l_*}\nabla_x(v\cdot\nabla_x\phi f_\pm)\|_{L^2_{v,x}}&\lesssim \|\nabla_x\phi\|_{H^1_x}\|w^{l_*}vf_\pm\|_{L^2_{v}H^1_x}\lesssim\E_{3,l_0},
	\end{align*}whenever $l_0\ge l_*+1+\max\{1/l_3,1\}=l_*+2$. 
	Now we verify that such $l_*$ exists. From the restriction above, we need to choose $l_*$ such that when $-1\le\gamma+2<0$, 
	\begin{align*}
		\frac{-5(\gamma+2)}{4l_3}< l_* \le l_0-\frac{2\gamma+4}{l_3}-3,\quad l_* \le l_0-2,
	\end{align*}and when $\gamma+2\ge0$, 
\begin{align*}
	0=l_* \le l_0-\frac{2\gamma+4}{l_3}-3,\quad l_* \le l_0-2.
\end{align*}
Such choice exists, according to \eqref{31b}. 
\qe\end{proof}

With the help of the above lemmas, we can control the nonlinear term $(\partial^\alpha_\beta\Gamma_\pm(f,g),w^{2l-2|\alpha|-2|\beta|}\partial^\alpha_\beta h)_{L^2_{v,x}}$. 
\begin{Lem}\label{lemmat}
	Let $K\ge 2$. For any multi-indices $|\alpha|+|\beta|\le K$ and real number $l\ge K$, we have \begin{equation}\begin{aligned}\label{15}
		\Big|&(\psi_{2|\alpha|+2|\beta|-6}w^{2l-2|\alpha|-2|\beta|}\notag\partial^\alpha_\beta\Gamma_\pm(f,g),\partial^\alpha_\beta h)_{L^2_{v,x}}\Big|\\&\lesssim \bigg(\sum_{|\alpha|+|\beta|\le K}\|\psi_{|\alpha|+|\beta|-3}w^{-C}\partial^{\alpha}_\beta f\|_{L^2_{v,x}}\sum_{\substack{|\alpha|\ge 1\\ |\alpha|+|\beta|\le K}}\|\psi_{|\alpha|+|\beta|-3}\partial^{\alpha}_{\beta}g\|_{\sigma,l-|\alpha|-|\beta|}\\
		&\qquad+\sum_{\substack{|\alpha|\ge 1\\ |\alpha|+|\beta|\le K}}\|\psi_{|\alpha|+|\beta|-3}w^{-C}\partial^{\alpha}_\beta f\|_{L^2_{v,x}}
		\sum_{|\alpha|+|\beta|\le K}\|\psi_{|\alpha|+|\beta|-3}\partial^{\alpha}_{\beta}g\|_{\sigma,l-|\alpha|-|\beta|}\bigg)
		\\
		&\qquad\qquad\qquad\qquad\qquad\qquad\qquad\qquad\qquad\times\|\psi_{|\alpha|+|\beta|-3}\partial^\alpha_\beta h\|_{\sigma,l-|\alpha|-|\beta|},\end{aligned}
	\end{equation}
	where $C>0$ can be arbitrarily large and we restrict $t\in[0,1]$ when considering $\psi=t^N$ as in Theorem \ref{main2}.
\end{Lem}
\begin{proof}
Using the estimate \eqref{12}, we have 
\begin{align*}
&\quad\,\big|(\psi_{2|\alpha|+2|\beta|-6}w^{2l-2|\alpha|-2|\beta|}\partial^\alpha_\beta\Gamma_\pm(f,g),\partial^\alpha_\beta h)_{L^2_{v,x}}\big| \\
&\lesssim \sum_{\alpha_1\le\alpha,\bar{\beta}\le\beta_1\le\beta}\Big\|\psi_{|\alpha|+|\beta|-3}|w^{-C}\partial^{\alpha_1}_{\bar{\beta}}f|_{L^2_v}|\partial^{\alpha-\alpha_1}_{\beta-\beta_1}g|_{\sigma,l-|\alpha|-|\beta|}\Big\|_{L^2_x}\|\psi_{|\alpha|+|\beta|-3}\partial^\alpha_\beta h\|_{\sigma,l-|\alpha|-|\beta|},
\end{align*}by choosing $l-|\alpha|-|\beta|$ to be the $l$ in \eqref{12}, where $C>0$ can be arbitrarily large. 
For brevity we denote the first term in the norm $\|\cdot\|_{L^2_x}$ inside the summation $\sum_{\alpha_1\le\alpha,\bar{\beta}\le\beta_1\le\beta}$ on the right hand side to be $I$ and discuss its value in several cases. 

If $2\le|\alpha_1|+|\beta_1|\le K$, then $|\alpha_2|+|\beta_2|\le |\alpha|+|\beta|-2$ and $l-|\alpha|-|\beta|\le l-|\alpha_2+\alpha'|-|\beta_2|$ for any $1\le|\alpha'|\le2$. Notice that in this case, $\psi_{|\alpha|+|\beta|-3}\le \psi_{|\alpha_1|+|\beta_1|-3}\psi_{|\alpha_2+\alpha'|+|\beta_2|-3}$. By using \eqref{13}, we have 
\begin{align}\notag
	I&\lesssim\psi_{|\alpha|+|\beta|-3}\|w^{-C}\partial^{\alpha_1}_{\beta_1}f\|_{L^2_{v,x}}\big\||\partial^{\alpha_2}_{\beta_2}g|_{\sigma,l-|\alpha|-|\beta|}\big\|_{L^\infty_x}\\
	&\notag\lesssim\|\psi_{|\alpha_1|+|\beta_1|-3}w^{-C}\partial^{\alpha_1}_{\beta_1}f\|_{L^2_{v,x}}\sum_{1\le|\alpha'|\le2}\|\psi_{|\alpha_2+\alpha'|+|\beta_2|-3}\partial^{\alpha_2+\alpha'}_{\beta_2}g\|_{\sigma,l-|\alpha_2+\alpha'|-|\beta_2|}\\
	&\lesssim\sum_{|\alpha|+|\beta|\le K}\|\psi_{|\alpha|+|\beta|-3}w^{-C}\partial^{\alpha}_\beta f\|_{L^2_{v,x}}\sum_{\substack{|\alpha|\ge 1\\ |\alpha|+|\beta|\le K}}\|\psi_{|\alpha|+|\beta|-3}\partial^{\alpha}_{\beta}g\|_{\sigma,l-|\alpha|-|\beta|}.\label{33aa}
\end{align}
Secondly, if $|\alpha_1|+|\beta_1|=1$, then $|\alpha_2|+|\beta_2|\le |\alpha|+|\beta|-1$. Notice that $\psi\le1$ and $\psi_{|\alpha|+|\beta|-3}\le \psi_{|\alpha_1+\alpha'_1|+|\beta_1|-3}\psi_{|\alpha_2+\alpha'_2|+|\beta_2|-3}$, for any $|\alpha'_1|= 1$, $|\alpha'_2|\le 1$.  Using \eqref{13} to give one $x$ derivative to $f$, we have 
\begin{align*}
	I&\lesssim \sum_{|\alpha'|= 1}\|\psi_{|\alpha_1+\alpha'|+|\beta_1|-3}w^{-C}\partial^{\alpha_1+\alpha'}_{\beta_1}f\|_{L^2_{v,x}}\sum_{|\alpha'|\le 1}\|\psi_{|\alpha_2+\alpha'|+|\beta_2|-3}\partial^{\alpha_2+\alpha'}_{\beta_2}g\|_{\sigma,l-|\alpha_2+\alpha'|-|\beta_2|}\\
	&\lesssim \sum_{\substack{|\alpha|\ge 1\\|\alpha|+|\beta|\le K}}\|\psi_{|\alpha|+|\beta|-3}w^{-C}\partial^{\alpha}_\beta f\|_{L^2_{v,x}}\sum_{|\alpha|+|\beta|\le K}\|\psi_{|\alpha|+|\beta|-3}\partial^{\alpha}_{\beta}g\|_{\sigma,l-|\alpha|-|\beta|}.
 \end{align*}
Thirdly, if $|\alpha_1|+|\beta_1|=0$, using \eqref{13} to give two $x$ derivatives to $f$, we have 
\begin{align}\label{33bb}\notag
	I&\lesssim\sum_{1\le|\alpha'|\le2}\|\psi_{|\alpha_1+\alpha'|+|\beta_1|-3}w^{-C}\partial^{\alpha_1+\alpha'}_{\beta_1}f\|_{L^2_{v,x}}\|\psi_{|\alpha_2|+|\beta_2|-3}\partial^{\alpha_2}_{\beta_2}g\|_{\sigma,l-|\alpha_2|-|\beta_2|}\\
	&\lesssim \sum_{\substack{|\alpha|\ge 1\\ |\alpha|+|\beta|\le K}}\|\psi_{|\alpha|+|\beta|-3}w^{-C}\partial^{\alpha}_\beta f\|_{L^2_{v,x}}
	\sum_{|\alpha|+|\beta|\le K}\|\psi_{|\alpha|+|\beta|-3}\partial^{\alpha}_{\beta}g\|_{\sigma,l-|\alpha|-|\beta|}.
\end{align}
Here we used $\psi_{|\alpha|+|\beta|-3}\le \psi_{|\alpha_1+\alpha'|+|\beta_1|-3}\psi_{|\alpha_2|+|\beta_2|-3}$, for any $|\alpha'|\le2$.  
Combining the above estimate, we have the desired result for $I$:
\begin{align*}
	I&\lesssim \sum_{|\alpha|+|\beta|\le K}\|\psi_{|\alpha|+|\beta|-3}w^{-C}\partial^{\alpha}_\beta f\|_{L^2_{v,x}}\sum_{\substack{|\alpha|\ge 1\\ |\alpha|+|\beta|\le K}}\|\psi_{|\alpha|+|\beta|-3}\partial^{\alpha}_{\beta}g\|_{\sigma,l-|\alpha|-|\beta|}\\
	&\qquad+\sum_{\substack{|\alpha|\ge 1\\ |\alpha|+|\beta|\le K}}\|\psi_{|\alpha|+|\beta|-3}w^{-C}\partial^{\alpha}_\beta f\|_{L^2_{v,x}}
	\sum_{|\alpha|+|\beta|\le K}\|\psi_{|\alpha|+|\beta|-3}\partial^{\alpha}_{\beta}g\|_{\sigma,l-|\alpha|-|\beta|}.
\end{align*}
Similar decomposition on $|\alpha_1|+|\beta_1|$ and $\psi_{|\alpha|+|\beta|-3}$ will be used frequently used later. 
\qe\end{proof}

A direct consequence of Lemma \ref{lemmat} is the following estimate.
\begin{Lem}\label{lemmag}
	Let $K\ge 2$, $|\alpha|+|\beta|\le K$, $l\ge K$. Then,
	\begin{equation*}
		|(\partial^\alpha\Gamma_\pm(f,f),\psi_{2|\alpha|-6}\partial^\alpha f_\pm)_{L^2_{v,x}}|\lesssim\E^{1/2}_{K,l}\D_{K,l}(t),
	\end{equation*}
\begin{equation*}
|(\partial^\alpha_\beta\Gamma_\pm(f,f),\psi_{2|\alpha|+2|\beta|-6}\partial^\alpha_\beta(\II-\PP) f)_{L^2_{v,x}}|\lesssim\E^{1/2}_{K,l}\D_{K,l}(t),
\end{equation*}
and
\begin{equation}\label{25}
	|(w^{2l-2|\alpha|-2|\beta|}\partial^\alpha_\beta\Gamma_\pm(f,f),\psi_{2|\alpha|+2|\beta|-6}\partial^\alpha_\beta f)_{L^2_{v,x}}|\lesssim \E^{1/2}_{K,l}\D_{K,l}(t)+\E_{K,l}\D^{1/2}_{K,l}(t).
\end{equation}
Also, for any smooth function $\zeta(v)$ satisfying $|\zeta(v)|\approx e^{-\lambda|v|^2}$ with some $\lambda>0$, we have 
\begin{align*}
	(\partial^\alpha\Gamma_\pm(f,f),\zeta(v))_{L_{v,x}}\lesssim \E^{1/2}_{K,l}\D^{1/2}_{K,l}(t).
\end{align*}
\end{Lem}
\begin{proof}
	For brevity, we only give the proof of \eqref{25}. Notice that 
	\begin{align*}
		&\quad\,(w^{2l-2|\alpha|-2|\beta|}\partial^\alpha_\beta\Gamma_\pm(f,f),\partial^\alpha_\beta f_\pm)_{L^2_{v,x}} \\
		&= (w^{2l-2|\alpha|-2|\beta|}\partial^\alpha_\beta\Gamma_\pm(f,f),\partial^\alpha_\beta (\II-\PP) f)_{L^2_{v,x}}+(w^{2l-2|\alpha|-2|\beta|}\partial^\alpha_\beta\Gamma_\pm(f,f),\partial^\alpha_\beta \PP f)_{L^2_{v,x}}.
	\end{align*}
	The first term on the right hand, by directly using Lemma \ref{lemmat} and the definition of $\E_{K,l}$ and $\D_{K,l}$, is bounded above by $\E^{1/2}_{K,l}\D_{K,l}(t)$, since there's zero $x$ derivative on $(\I-\P)f$ in the definition of $\D_{K,l}$. But there's no such term for $\P f$ in $\D_{K,l}$. Hence, the second right-hand term can only be bounded above by $\E_{K,l}\D^{1/2}_{K,l}(t)$. 
\qe\end{proof}

For later use, we need the following estimate on $v\cdot\nabla_x\phi f_\pm$ and $\nabla_x\phi\cdot\nabla_vf_\pm$. Here we assume that $\|\phi\|_{L^\infty_x}\approx 1$ and hence, $|e^{\pm\phi}|\approx 1$. 
\begin{Lem}\label{Lem26}Let $l\ge K\ge 3$, $1\le|\alpha|\le K$ and $|\alpha|+|\beta|\le K$. Then, for $\alpha_1\le\alpha,\beta_1\le\beta$ with $|\alpha_1|\ge 1$, it holds that 
	\begin{equation}\label{333}
		|(v_i\partial^{\alpha_1+e_i}\phi\partial^{\alpha-\alpha_1}f_\pm,\psi_{2|\alpha|-6}e^{\pm\phi}w^{2l-2|\alpha|}\partial^\alpha f_\pm)_{L^2_{v,x}}|\lesssim \E^{1/2}_{K,l}\D_{K,l}, 
	\end{equation}
	\begin{equation}\label{33a}
		|(\partial_{\beta_1}v_i\partial^{\alpha_1+e_i}\phi\partial^{\alpha-\alpha_1}_{\beta-\beta_1}f_\pm,\psi_{2|\alpha|-6}e^{\pm\phi}w^{2l-2|\alpha|-2|\beta|}\partial^\alpha_\beta f_\pm)_{L^2_{v,x}}|\lesssim \E^{1/2}_{K,l}\D_{K,l}.
	\end{equation}
\begin{equation}\label{33b}
|(\partial_{\beta_1}v_i\partial^{\alpha_1+e_i}\phi\partial^{\alpha-\alpha_1}_{\beta-\beta_1}(\II-\PP)f,\psi_{2|\alpha|+2|\beta|-6}e^{\pm\phi}w^{2l-2|\alpha|-2|\beta|}\partial^\alpha_\beta (\II-\PP)f)_{L^2_{v,x}}|\lesssim \E^{1/2}_{K,l}\D_{K,l}.
\end{equation}
\end{Lem}
\begin{proof}
We only give a detailed proof of \eqref{333}.
For $|\alpha_1|\ge 1$ with $\alpha_1\le\alpha$, by noticing \eqref{10w}, we have $|v_i|w^{l-|\alpha|}\le \<v\>^{\gamma+2}w^{l-|\alpha|+1}$ and hence, 
\begin{align}\notag
	&\quad\,|(v_i\partial^{\alpha_1+e_i}\phi\partial^{\alpha-\alpha_1}f_\pm,\psi_{2|\alpha|-6}e^{\pm\phi}w^{2l-2|\alpha|}\partial^\alpha f_\pm)_{L^2_{v,x}}|\\
	&\lesssim\label{26} \|\psi_{|\alpha|-3}\partial^{\alpha_1}\nabla_x\phi\<v\>^{\frac{\gamma+2}{2}}w^{l-|\alpha|+1}\partial^{\alpha-\alpha_1}f_\pm\|_{L^2_{v,x}}\|\psi_{|\alpha|-3}\<v\>^{\frac{\gamma+2}{2}}w^{l-|\alpha|}\partial^{\alpha}f_\pm\|_{L^2_{v,x}}. 
\end{align}
For the first term on the right hand of \eqref{26}, we discuss its value as the following. 
If $\alpha_1<\alpha$, then $1\le|\alpha_1|\le K-1$ and there's at least one derivative on $f_\pm$ with respect to $x$. Then by the same discussion on the value of $|\alpha_1|$ as \eqref{33aa}-\eqref{33bb}, one has 
\begin{align}\label{37bb}
	&\quad\,\|\psi_{|\alpha|-3}\partial^{\alpha_1}\nabla_x\phi\<v\>^{\frac{\gamma+2}{2}}w^{l-|\alpha|+1}\partial^{\alpha-\alpha_1}f_\pm\|_{L^2_{v,x}}
	\lesssim \E^{1/2}_{K,l}\D^{1/2}_{K,l}, 
\end{align}where we used $\|\<v\>^{(\gamma+2)/2}(\cdot)\|_{L^2_{v,x}}\lesssim \|\cdot\|_{\sigma,0}$. 
If $\alpha_1=\alpha$, then we decompose $f_\pm=\PP f+(\II-\PP)f$ and give one derivative to $\PP f$ with respect to $x$ by using \eqref{13}. 
That is, 
\begin{align*}
	&\quad\,\|\psi_{|\alpha|-3}\partial^{\alpha}\nabla_x\phi\<v\>^{\frac{\gamma+2}{2}}w^{l-|\alpha-\alpha_1|}\PP f\|_{L^2_{v,x}}\notag\\
	&\lesssim\|\psi_{|\alpha|-3}\partial^\alpha\nabla_x\phi\|_{L^2_x}\sum_{1\le|\alpha'|\le 2}\|\psi_{|\alpha'|-3}\partial^{\alpha'}\PP f\|_{L^2_{v,x}}\\
	&\lesssim \E^{1/2}_{K,l}\D_{K,l}^{1/2}.
\end{align*}
For the part $(\II-\PP)f$, we will use \eqref{13} to give two derivatives to $(\II-\PP)f$ when $|\alpha|\ge 3$, one derivative to $(\II-\PP)f$ when $|\alpha|=2$ and give nothing to $(\II-\PP)f$ when $|\alpha|=1$. That is, 
\begin{align}\notag
	&\quad\,\|\psi_{|\alpha|-3}\partial^{\alpha}\nabla_x\phi\<v\>^{\frac{\gamma+2}{2}}w^{l-|\alpha-\alpha_1|}(\II-\PP)f\|_{L^2_{v,x}}\notag\\
	&\lesssim\notag \sum_{3\le|\alpha|\le K}\|\psi_{|\alpha|-3}\partial^{\alpha}\nabla_x\phi\|_{L^2_x}\sum_{1\le|\alpha'|\le2}\|\psi_{|\alpha'|-3}\<v\>^{\frac{\gamma+2}{2}}w^{l-|\alpha|+1}\partial^{\alpha'}(\II-\PP)f\|_{L^2_{v,x}}\\\notag
	&\qquad\notag+\sum_{|\alpha|=2}\sum_{|\alpha'|\le1}\|\psi_{|\alpha+\alpha'|-3}\partial^{\alpha+\alpha'}\nabla_x\phi\|_{L^2_x}\sum_{|\alpha_1'|=1}\|\psi_{|\alpha_1'|-3}\<v\>^{\frac{\gamma+2}{2}}w^{l-|\alpha|+1}\partial^{\alpha_1'}(\II-\PP)f\|_{L^2_{v,x}}\\
	&\qquad\notag+\sum_{|\alpha|=1}\sum_{|\alpha'|\le2}\|\psi_{|\alpha+\alpha'|-3}\partial^{\alpha+\alpha'}\nabla_x\phi\|_{L^2_x}\|\<v\>^{\frac{\gamma+2}{2}}w^{l}(\II-\PP)f\|_{L^2_{v,x}}\\
	&\lesssim \E^{1/2}_{K,l}\D^{1/2}_{K,l},\label{29}
\end{align}where we used $-3$ in $\psi$ to assure that the decomposition on $\psi_{|\alpha|-3}$ above is valid.
Thus, when $\alpha_1=\alpha$, 
\begin{align}
	&\quad\,\|\psi_{|\alpha|-3}\partial^{\alpha_1}\nabla_x\phi\<v\>^{\frac{\gamma+2}{2}}w^{l-|\alpha|+1}\partial^{\alpha-\alpha_1}f_\pm\|_{L^2_{v,x}}
	\lesssim \E^{1/2}_{K,l}\D^{1/2}_{K,l}.\label{27c}
\end{align}
Plugging the above estimate into \eqref{26}, we have 
\begin{align*}
	|(v_i\partial^{\alpha_1+e_i}\phi\partial^{\alpha-\alpha_1}f_\pm,\psi_{2|\alpha|-6}e^{\pm\phi}w^{2l-2|\alpha|}\partial^\alpha f_\pm)_{L^2_{v,x}}|\lesssim \E^{1/2}_{K,l}\D_{K,l}. 
\end{align*}

Similarly, for $|\beta|\le K$ and $\beta_1\le \beta$, we have $|\partial_{\beta_1}v_i|\le \<v\>$ and hence, 
\begin{align}\label{28}
	&\quad\,|(\partial_{\beta_1}v_i\partial^{\alpha_1+e_i}\phi\partial^{\alpha-\alpha_1}_{\beta-\beta_1}f_\pm,\psi_{2|\alpha|+2|\beta|-6}e^{\pm\phi}w^{2l-2|\alpha|-2|\beta|}\partial^\alpha_\beta f_\pm)_{L^2_{v,x}}|\\
	&\lesssim \|\psi_{|\alpha|+|\beta|-3}\partial^{\alpha_1}\nabla_x\phi\<v\>^{\frac{\gamma+2}{2}}w^{l-|\alpha|+1-|\beta-\beta_1|}\partial^{\alpha-\alpha_1}_{\beta-\beta_1} f_\pm\|_{L^2_{v,x}}\|\psi_{|\alpha|+|\beta|-3}\<v\>^{\frac{\gamma+2}{2}}w^{l-|\alpha|-|\beta|}\partial^\alpha_\beta f_\pm\|_{L^2_{v,x}}.\notag
\end{align}
For the first term on the right hand of \eqref{28}, we use the same argument as in \eqref{26}-\eqref{27c} to find its upper bound $\E^{1/2}_{K,l}\D^{1/2}_{K,l}$. Hence, \eqref{33a} is bounded above by $\E^{1/2}_{K,l}\D_{K,l}$. The proof of \eqref{33b} is similar and we omit it for brevity.
\qe\end{proof}

\begin{Lem}\label{Lem27}
	Let $|\alpha|+|\beta|\le K$, $l\ge K\ge 3$. Then, for $\alpha_1\le\alpha,\beta_1\le\beta$, it holds that 
	\begin{equation}\label{30a}
		|(\partial^{\alpha_1+e_i}\phi\partial^{\alpha-\alpha_1}_{e_i}f_\pm,\psi_{2|\alpha|-6}e^{\pm\phi}w^{2l-2|\alpha|}\partial^\alpha f_\pm)_{L^2_{v,x}}|\le \E^{1/2}_{K,l}\D_{K,l}, 
	\end{equation}
\begin{equation}\label{30b}
	|(\partial^{\alpha_1+e_i}\phi\partial^{\alpha-\alpha_1}_{\beta+e_i}f_\pm,\psi_{2|\alpha|+2|\beta|-6}e^{\pm\phi}w^{2l-2|\alpha|-2|\beta|}\partial^\alpha_\beta f_\pm)_{L^2_{v,x}}|\le \E^{1/2}_{K,l}\D_{K,l}.
\end{equation}
\begin{equation}\label{30c}
|(\partial^{\alpha_1+e_i}\phi\partial^{\alpha-\alpha_1}_{\beta+e_i}(\II-\PP)f,\psi_{2|\alpha|+2|\beta|-6}e^{\pm\phi}w^{2l-2|\alpha|-2|\beta|}\partial^\alpha_\beta (\II-\PP)f)_{L^2_{v,x}}|\le \E^{1/2}_{K,l}\D_{K,l}.
\end{equation}
\end{Lem}
\begin{proof}
We firstly prove \eqref{30a}.
When $\alpha_1=0$, by integration by parts and $\partial_{e_i}w^{2l-2|\alpha|}\lesssim w^{2l-2|\alpha|}\<v\>^{\gamma+2}$, we have 
\begin{align}
	&\quad\,|(\partial^{e_i}\phi\partial^{\alpha}_{e_i}f_\pm,\psi_{2|\alpha|-6}e^{\pm\phi}w^{2l-2|\alpha|}\partial^\alpha f_\pm)_{L^2_{v,x}}|\notag\\
	&\lesssim |(\partial^{e_i}\phi\partial^{\alpha}f_\pm,\psi_{2|\alpha|-6}e^{\pm\phi}(\partial_{e_i}w^{2l-2|\alpha|})\partial^\alpha f_\pm)_{L^2_{v,x}}|\notag\\
	&\lesssim \|\psi_{|\alpha|-3}\nabla_x\phi \<v\>^{\frac{\gamma+2}{2}}w^{l-|\alpha|}\partial^{\alpha}f_\pm\|_{L^2_{v,x}}\|\psi_{|\alpha|-3}w^{l-|\alpha|}\partial^\alpha f_\pm\|_{L^2_{v,x}}\notag\\
	&\lesssim \sum_{|\alpha'|\le 2}\|\psi_{|\alpha'|-3}\partial^{\alpha'}\nabla_x\phi\|_{L^2_x}\sum_{1\le|\alpha|\le K}\|\psi_{|\alpha|-3}\<v\>^{\frac{\gamma+2}{2}}w^{l-|\alpha|}\partial^\alpha f_\pm\|_{L^2_{v,x}}\sum_{|\alpha|\le K}\|\psi_{|\alpha|-3}w^{l-|\alpha|}\partial^\alpha f_\pm\|_{L^2_{v,x}}\notag\\
	&\lesssim \E^{1/2}_{K,l}\D_{K,l}, \label{31a}
\end{align}where we use \eqref{13} to assure that there's always at least one derivative on the first $f_\pm$. 
When $|\alpha_1|\ge 1$, we have $|\alpha|\ge 1$. Then we decompose $f_\pm=\PP f+(\II-\PP)f$ to obtain 
\begin{align}
	&\quad\,(\partial^{\alpha_1+e_i}\phi\partial^{\alpha-\alpha_1}_{e_i}f_\pm,\psi_{2|\alpha|-6}e^{\pm\phi}w^{2l-2|\alpha|}\partial^\alpha f_\pm)_{L^2_{v,x}}=I+J,
\end{align}with 
\begin{align*}
	I &= (\partial^{\alpha_1+e_i}\phi\partial^{\alpha-\alpha_1}_{e_i}\PP f,\psi_{2|\alpha|-6}e^{\pm\phi}w^{2l-2|\alpha|}\partial^\alpha f_\pm)_{L^2_{v,x}},\\
	J&= (\partial^{\alpha_1+e_i}\phi\partial^{\alpha-\alpha_1}_{e_i}(\II-\PP)f,\psi_{2|\alpha|-6}e^{\pm\phi}w^{2l-2|\alpha|}\partial^\alpha f_\pm)_{L^2_{v,x}}.
\end{align*}
Now we estimate $I$ and $J$ as the followings. For $I$, noticing there's exponential decay in $v$, we have 
\begin{align*}
	|I|&\lesssim\|\psi_{|\alpha|-3}\partial^{\alpha_1+e_i}\phi\partial^{\alpha-\alpha_1}\PP f\|_{L^2_{v,x}}\|\psi_{|\alpha|-3}\<v\>^{\frac{\gamma+2}{2}}w^{l-|\alpha|}\partial^\alpha f_\pm\|_{L^2_{v,x}}\\
	&\lesssim \sum_{|\alpha_1|\le K}\|\psi_{|\alpha_1|-3}\partial^{\alpha_1}\nabla_x\phi\|_{L^2_x}\sum_{1\le|\alpha|\le K}\|\psi_{|\alpha|-3}\partial^\alpha\PP f\|_{L^2_{v,x}}\|\psi_{|\alpha|-3}\partial^\alpha f_\pm\|_{\sigma,l-|\alpha|}\\
	&\lesssim \E^{1/2}_{K,l}\D_{K,l},
\end{align*}where we used same discussion on the value of $|\alpha_1|$ as \eqref{33aa}-\eqref{33bb} and give at least one derivative to $\PP f$. 
For $J$, in the case of soft potential $-1\le\gamma+2< 0$, we use the trick in \cite{Guo2012}. 
\begin{align*}
	J &= \big(\partial^{\alpha_1+e_i}\phi \partial_{e_i}\big(w^{-\frac{1}{2}}w^{l-|\alpha|+1}\partial^{\alpha-\alpha_1}(\II-\PP)f\big),\psi_{2|\alpha|-6}w^{-\frac{1}{2}}e^{\pm\phi}w^{l-|\alpha|}\partial^\alpha f_\pm\big)_{L^2_{v,x}}\\
	&\qquad - \big(\partial^{\alpha_1+e_i}\phi \partial_{e_i}\big(w^{-\frac{1}{2}}w^{l-|\alpha|+1}\big)\partial^{\alpha-\alpha_1}(\II-\PP)f,\psi_{2|\alpha|-6}w^{-\frac{1}{2}}e^{\pm\phi}w^{l-|\alpha|}\partial^\alpha f_\pm\big)_{L^2_{v,x}}\\
	&=: J_1+J_2.
\end{align*}
For the term $J_2$, we will use the trick of \eqref{37bb}-\eqref{29}. That is, by \eqref{13}, if $\alpha_1< \alpha$, we use discussion on the value of $|\alpha_1|$ as \eqref{33aa}-\eqref{33bb}. If $\alpha_1=\alpha$, we use the same argument as \eqref{29}. Then
\begin{align*}
	|J_2|
	&\lesssim \E^{1/2}_{K,l}\D_{K,l}. 
\end{align*}where we used $\|\<v\>^{\gamma/2}(\cdot)\|_{L^2_{v}}\lesssim\|\cdot\|_{\sigma,0}$. 
For $J_1$, since $|\<D_v\>^{\frac{1}{2}}\<v\>^{\frac{\gamma}{2}}(\cdot)|_{L^2_v}\le|\cdot|_{\sigma,0}$. By the discussion as \eqref{33aa}-\eqref{33bb}, one has 
\begin{align*}
	|J_1|&\lesssim \big|\big(\partial^{\alpha_1+e_i}\phi \mathbf{i}\eta_i\F_v\big(w^{-\frac{1}{2}}w^{l-|\alpha|+1}\partial^{\alpha-\alpha_1}(\II-\PP)f\big),\psi_{2|\alpha|-6}\F_v\big(w^{-\frac{|\alpha_1|}{2}}e^{\pm\phi}w^{l-|\alpha|}\partial^\alpha f_\pm\big)\big)_{L^2_{v,x}}\big|\\
	&\lesssim \int_{\R^3}\psi_{2|\alpha|-6}|\partial^{\alpha_1}\nabla_x\phi|\|\<D_v\>^{\frac{1}{2}}\<v\>^{\frac{\gamma}{2}}w^{l-|\alpha-\alpha_1|}\partial^{\alpha-\alpha_1}(\II-\PP)f\|_{L^2_{v}}\|\<D_v\>^{\frac{1}{2}}\<v\>^{\frac{\gamma}{2}}w^{l-|\alpha|}\partial^\alpha f_\pm\|_{L^2_v}\,dx\\
	&\lesssim \sum_{|\alpha_1|\le K}\|\psi_{|\alpha_1|-3}\partial^{\alpha_1}\nabla_x\phi\|_{L^2_x}\sum_{|\alpha|\le K}\|\psi_{|\alpha|-3}\partial^{\alpha}(\II-\PP)f\|_{\sigma,l-|\alpha|}
	\sum_{1\le|\alpha|\le K}\|\psi_{|\alpha|-3}\partial^\alpha f_\pm\|_{\sigma,l-|\alpha|}\\
	&\lesssim \E^{1/2}_{K,l}\D_{K,l}, 
\end{align*}
where $\F_v$ is the Fourier transform with respect to $v$, $\mathbf{i}$ is the pure imaginary unit, $\eta_i$ is the variable after Fourier transform $\F_v$. 

For the term $J$ in the case of hard potential $\gamma+2\ge 0$, we use the following argument.
\begin{align*}
	|J|&\le \|\psi_{|\alpha|-3}\partial^{\alpha_1+e_i}\phi w^{l-|\alpha|+1-|e_i|}\partial^{\alpha-\alpha_1}_{e_i}(\II-\PP)f\|_{L^2_{v,x}}\|\psi_{|\alpha|-3}w^{l-|\alpha|}\partial^{\alpha}f_\pm\|_{L^2_{v,x}}.
\end{align*}
As in \eqref{33aa}-\eqref{33bb}, the first term is estimate by 
\begin{align*}
	&\quad\,\sum_{|\alpha_1|=1}\sum_{|\alpha'_1|\le 2}\|\psi_{|\alpha_1+\alpha'_1|-3}\partial^{\alpha_1+\alpha_1'+e_i}\phi\|_{L^2_{x}}\|\psi_{|\alpha-\alpha_1|+|e_i|-3}w^{l-|\alpha-\alpha_1|-|e_i|}\partial^{\alpha-\alpha_1}_{e_i}(\II-\PP)f\|_{L^2_{v,x}}\\
	&+\sum_{|\alpha_1|=2}\sum_{|\alpha'_1|\le 1}\|\psi_{|\alpha_1+\alpha'_1|-3}\partial^{\alpha_1+\alpha_1'+e_i}\phi\|_{L^2_{x}}\sum_{|\alpha'|\le1}\|\psi_{|\alpha+\alpha'-\alpha_1|+|e_i|-3}w^{l-|\alpha+\alpha'-\alpha_1|-|e_i|}\partial^{\alpha+\alpha'-\alpha_1}_{e_i}(\II-\PP)f\|_{L^2_{v,x}}\\
	&+\sum_{3\le|\alpha_1|\le K}\|\psi_{|\alpha_1|-3}\partial^{\alpha_1+e_i}\phi\|_{L^2_{x}}\sum_{|\alpha'|\le2}\|\psi_{|\alpha+\alpha'-\alpha_1|+|e_i|-3}w^{l-|\alpha+\alpha'-\alpha_1|-|e_i|}\partial^{\alpha+\alpha'-\alpha_1}_{e_i}(\II-\PP)f\|_{L^2_{v,x}}\\
	&\lesssim \E^{1/2}_{K,l}\D^{1/2}_{K,l}. 
\end{align*}
Thus $|J|$ is bounded above by $\E^{1/2}_{K,l}\D_{K,l}$.

Collecting all the above estimates, we obtain \eqref{30a}. 
The proof of \eqref{30b} and \eqref{30c} are the same as \eqref{30a}, and the details are omitted for brevity.

\qe\end{proof}

\section{Macroscopic Estimate}
In this section, we assume $\psi=1$. We will analyze the macroscopic dissipation by taking the projection $\P$ on the equation \eqref{7}. Since we are dealing with Vlasov-Poisson-Landau system, the idea here is similar to the Boltzmann equation case \cite{Gressman2011}. Similar macroscopic estimate can be found in \cite{Strain2013,Deng2020c}. Notice that the calculation in this section is valid for both hard potential $\gamma+2\ge 0$ and soft potential $\gamma+2<0$. 

Consider the homogeneous linearized system 
\begin{equation}\label{39}
	\left\{\begin{aligned}
		&\partial_tf_\pm +v\cdot\nabla_x f_\pm\pm \mu^{1/2}v\cdot\nabla_x\phi + Lf_\pm = 0,\\
		&-\Delta_x\phi = \int_{\R^3}(f_+-f_-)\mu^{1/2}\,dv,\\
		&f_\pm|_{t=0}=f_{0,\pm},
	\end{aligned} \right.
\end{equation}which is the homogeneous system of \eqref{7}-\eqref{9}. We write the formal solution to Cauchy problem \eqref{39} to be 
\begin{equation}\label{42b}
	f = e^{tB}f_0,
\end{equation}where $e^{tB}$ denotes the solution operator. For later use, we will analyze the large time behavior of system \eqref{39}. The idea here follows from \cite{Strain2012}.
\begin{Thm}\label{homogen}
	Let $f= e^{tB}f_0$ be the solution to \eqref{39}, $m\ge 0$ be an integer and time decay rate index to be 
	\begin{equation*}
		\sigma_m = \frac{3}{4}+\frac{m}{2}.
	\end{equation*}
Let $l\ge 0$, $t\ge 0$, $l_*>\sigma_m\frac{\gamma+2}{\gamma}$. Then for the case of hard potential $\gamma+2\ge 0$, we have
	\begin{equation}
		\|w^l\nabla^m_xf(t)\|_{L^2_{v,x}}+\|\nabla_x^mE(t)\|_{L^2_x} \lesssim (1+t)^{-\sigma_m}\big(\|w^lf_0\|_{Z_1}+\|E_0\|_{L_1}+\|w^l\nabla_x^mf_0\|_{L^2_{v,x}}\big),
	\end{equation}while for the case of soft potential $-1\le\gamma+2< 0$, we have 
\begin{align}
	\|\nabla^m_xw^lf\|_{L^2_{v,x}}+\|\nabla^m_xE\|_{L^2_x}\lesssim (1+t)^{-\sigma_m}\big(\|w^{l+l_*}f_0\|_{Z_1}+\|E_0\|_{L_1}+\|w^{l+l_*}\nabla^m_xf_0\|_{L^2_{v,x}}).
\end{align}
\end{Thm}
In order to prove this Theorem, we need the following preparation.

Recall the projection notation in \eqref{10}. By multiplying the equation \eqref{7} with $\mu^{1/2}, v_j\mu^{1/2}(j=1,2,3)$, $\frac{1}{6}(|v|^2-3)\mu^{1/2}$, $(v_iv_j-1)\mu^{1/2}$, and $\frac{1}{10}(|v|^2-5)v_j\mu^{1/2}$, and then integrating them over the $\R^3_v$, we have 
\begin{equation}\label{17}\left\{\begin{aligned}
&\partial_ta_\pm + \nabla\cdot b + \nabla_x\cdot(v\mu^{1/2},(\II-\PP)f)_{L^2_v} =0,\\
&\partial_t\big(b_j+(v_j\mu^{1/2},(\II-\PP)f)_{L^2_v}\big)+\partial_j(a_\pm+2c)\mp E_j+(v_j\mu^{1/2},v\cdot\nabla_x(\II-\PP)f)_{L^2_v}\\&\qquad\qquad\qquad\qquad\qquad\qquad\qquad\qquad\qquad = (L_\pm f+g_\pm,v_j\mu^{1/2})_{L^2_v},\\
&\partial_t\Big(c+\frac{1}{6}((|v|^2-3)\mu^{1/2},(\II-\PP)f)_{L^2_v}\Big)+\frac{1}{3}\nabla_x\cdot b + \frac{1}{6}((|v|^2-3)\mu^{1/2},v\cdot\nabla(\II-\PP)f)_{L^2_v} \\&\qquad\qquad\qquad\qquad\qquad\qquad\qquad\qquad\qquad= \frac{1}{6}(L_\pm f+g_\pm,(|v|^2-3)\mu^{1/2})_{L^2_v},\\
&\partial_t\big(\Theta_{jj}((\II-\PP)f)+2c\big) + 2\partial_jb_j = \Theta_{jj}(g_\pm+h_\pm),\\
&\partial_t\Theta_{jk}((\II-\PP)f)+\partial_jb_k+\partial_kb_j + \nabla_x\cdot(v\mu^{1/2},(\II-\PP)f)_{L^2_v} \\&\qquad\qquad\qquad\qquad\qquad\qquad\qquad\qquad\qquad= \Theta_{jk}(g_\pm+h_\pm)+(\mu^{1/2},g_\pm)_{L^2_v},\ j\neq k,\\
&\partial_t\Lambda_j((\II-\PP)f)+\partial_jc = \Lambda_j(g_\pm+h_\pm),
\end{aligned}\right.
\end{equation}
where for brevity, we denote $I=(I_+,I_-)$ with $I_\pm f=f_\pm$ and 
\begin{align}\label{22a}
g_\pm &= \pm\nabla_x\phi\cdot\nabla_vf_\pm\mp\frac{1}{2}\nabla_x\phi\cdot vf_\pm+\Gamma_\pm(f,f),\\
h_\pm &= -v\cdot\nabla_x(\II-\PP)f+L_\pm f. \notag
\end{align}
For high-order moments, we define for $1\le j,k\le 3$ that 
\begin{align*}
	\Theta_{jk}(f_\pm) = ((v_jv_k-1)\mu^{1/2},f_\pm)_{L^2_v},\ \ \Lambda_j(f_\pm) =\frac{1}{10}((|v|^2-5)v_j\mu^{1/2},f_\pm)_{L^2_v}. 
\end{align*}
Notice that $(\P_\pm f,v\mu^{1/2})_{L^2_v}$ and $(\P_\pm f,(|v|^2-3)\mu^{1/2})_{L^2_v}$ is not $0$ in general and similar for $\Gamma_\pm$. Also, we have used
\begin{align*}
(\pm\nabla_x\phi\cdot\nabla_vf_\pm\mp\frac{1}{2}\nabla_x\phi\cdot vf_\pm,\mu^{1/2})_{L^2_v}=0,
\end{align*}which is obtained by integration by parts on $v$.
By taking the mean value of every two equations with sign $\pm$ in \eqref{17}, we have 
\begin{equation}\label{19}\left\{
\begin{aligned}
&\partial_t\Big(\frac{a_++a_-}{2}\Big)+\nabla_x\cdot b = 0,\\
&\partial_tb_j+\partial_j\Big(\Big(\frac{a_++a_-}{2}\Big)+2c\Big)+\frac{1}{2}\sum_{k=1}^3\partial_k\Theta_{jk}((\I-\P)f\cdot(1,1))
= \frac{1}{2}(g_++g_-,v_j\mu^{1/2})_{L^2_v},\\
&\partial_tc+\frac{1}{3}\nabla_x\cdot b + \frac{5}{6}\sum^3_{j=1}\partial_j\Lambda_j((\I-\P)f\cdot(1,1)) = \frac{1}{12}(g_++g_-,(|v|^2-3)\mu^{1/2})_{L^2_v},\\
&\partial_t\Big(\frac{1}{2}\Theta_{jk}((\II-\PP)f\cdot(1,1))+2c\delta_{jk}\Big) + \partial_jb_k+\partial_kb_j = \frac{1}{2}\Theta_{jk}(g_++g_-+h_++h_-),\\
&\frac{1}{2}\partial_t\Lambda_j((\II-\PP)f\cdot(1,1))+\partial_jc = \frac{1}{2}\Lambda_j(g_++g_-+h_++h_-),
\end{aligned}\right.
\end{equation}for $1\le j\le3$. 
where $\delta_{jk}$ is the Kronecker delta. Moreover, for obtaining the dissipation of the electric field $E$, we take the difference with sign $\pm$ in the first two equations in \eqref{17}, we have 
\begin{equation}\label{21}\left\{
\begin{aligned}
&\partial_t(a_+-a_-)+\nabla_x\cdot G=0,\\
&\partial_tG + \nabla_x(a_+-a_-)-2E+\nabla_x\cdot\Theta((\I-\P)f\cdot(1,-1))=((g+Lf)\cdot(1,-1),v\mu^{1/2})_{L^2_v},
\end{aligned}\right.
\end{equation}
where 
\begin{align}\label{27aa}
G = (v\mu^{1/2},(\I-\P)f\cdot(1,-1))_{L^2_v}.
\end{align}
Recall that $E=-\nabla_x\phi$. Then by equation \eqref{8}, we have 
\begin{align}\label{16}
\nabla_x\cdot E = a_+-a_-. 
\end{align}

\begin{Lem}\label{Lemma31}
	Let $(f,E)$ be the solution to the Cauchy problem \eqref{7}-\eqref{9}. For any $K\ge 3$, there exists a functional $\E^{(1)}_{K}(t),\E^{(1)}_{K,h}(t)$ such that
	\begin{align}\label{27a}
	\E^{(1)}_{K}&\lesssim \sum_{|\alpha|\le K}\|\partial^\alpha f\|^2_{L^2_{x,v}}+\sum_{|\alpha|\le K-1}\|\partial^\alpha E\|^2_{L^2_x},\\
	\label{27b}\E^{(1)}_{K,h}&\lesssim \sum_{1\le|\alpha|\le K}\|\partial^\alpha\P f\|^2_{L^2_{x,v}}+\sum_{|\alpha|\le K}\|\partial^\alpha(\I-\P) f\|^2_{L^2_{x,v}}+\sum_{|\alpha|\le K-1}\|\partial^\alpha E\|^2_{L^2_x},
	\end{align}
	and for any $t\ge 0$, 
	\begin{equation}\label{24}\begin{aligned}
	\partial_t&\E^{(1)}_{K} + \lambda\sum_{|\alpha|\le K-1} \|\partial^\alpha\nabla_x(a_\pm,b,c)\|^2_{L^2_x}+\|a_+-a_-\|^2_{L^2_x}+\sum_{|\alpha|\le K-1}\|\partial^\alpha E\|^2_{L^2_x}\\
	&\qquad\lesssim\sum_{|\alpha|\le K} \|(\I-\P)\partial^{\alpha}{f}\|_{\sigma,0}^2+\E_{K,l}(t)\D_{K,l}(t).
	\end{aligned}
	\end{equation}
	\begin{equation}\label{24a}\begin{aligned}
	\partial_t&\E^{(1)}_{K,h} + \lambda\sum_{1\le|\alpha|\le K-1} \|\partial^\alpha\nabla_x(a_\pm,b,c)\|^2_{L^2_x}+\|\nabla_x(a_+-a_-)\|^2_{L^2_x}+\sum_{|\alpha|\le K-1}\|\partial^\alpha E\|^2_{L^2_x}\\
	&\qquad\lesssim\sum_{|\alpha|\le K} \|(\I-\P)\partial^{\alpha}{f}\|_{\sigma,0}^2+\E^h_{K,l}(t)\D_{K,l}(t).
	\end{aligned}
	\end{equation}
\end{Lem}
\begin{proof}
	The proof is the same as \cite[Lemma 3.1]{Deng2020c} and \cite[Lemma 5.1]{Duan2013}, so we only illustrate the difference; see also \cite{Duan2013,Strain2013}. 
	Let $\zeta(v)$ be a function satisfying $
	|\zeta(v)|\approx e^{-\lambda|v|^2},$ for some $\lambda>0$. Then we will use  
\begin{align*}
	(\zeta,(\I-\P)f)_{L^2_{v,x}}\lesssim \|(\I-\P)f\|_{\sigma,0},
\end{align*}
and 
\begin{align*}
	|(L f,f)_{L^2_{v,x}}|\lesssim\|(\I-\P)f\|^2_{\sigma,0},
\end{align*}instead of using the norm $\|(\tilde{a}^{1/2})^w(\cdot)\|_{L^2_{v,x}}$ in \cite[Lemma 3.1]{Deng2020c}. Using the same argument as \cite[Lemma 3.1]{Deng2020c} or \cite[Lemma 5.1]{Duan2013}, we complete the proof. 
\end{proof}

\begin{Lem}
	Let $f$ be the solution to \eqref{39}. Then the followings are valid. 
	
	(1) There exists a time-frequency interactive functional $\E^{(2)}$ such that 
	\begin{equation*}
	\E^{(2)} \approx |\widehat{f}|_{L^2_v}^2+|\widehat{E}|^2,
	\end{equation*}and for $t\ge 0$, $y\in\R^3$, 
	\begin{equation}\label{44a}
	\partial_t\E^{(2)}(t,y)+\frac{\lambda|y|^2}{1+|y|^2}(|\widehat{f}|^2_{\sigma,0}+|E|^2)\le 0. 
	\end{equation}
	
	(2)
	There exists a time-frequency interactive functional $\E^{(2)}_l$ such that 
	\begin{align}\label{444}
	\E^{(2)}_l \approx |w^l\widehat{f}|_{L^2_v}^2+|\widehat{E}|^2,
	\end{align}and for $t\ge 0$, $y\in\R^3$, 
	\begin{align}\label{45aa}
	\partial_t\E^{(2)}_l(t,y)+\frac{\lambda|y|^2}{1+|y|^2}(|\widehat{f}|^2_{\sigma,l}+|\widehat{E}|^2)\le 0. 
	\end{align}
\end{Lem}
\begin{proof}
	The proof is the same as \cite[Lemma 3.3]{Deng2020c}, where we use dissipation rate from Lemma \ref{lemmaL} instead. Then we can obtain the desired result. 
\end{proof}

Now we are in a position to prove the large time behavior of the homogeneous system \eqref{39}. 
\begin{proof}[Proof of Theorem \ref{homogen}]
	{\bf Step 1}. We firstly prove the hard potential case. 
	By noticing $\|w^l(\cdot)\|_{L^2_v}\lesssim \|\cdot\|_{\sigma,l}$ for hand potential, \eqref{45aa} gives that 
	\begin{align*}
	\partial_t\E^{(2)}_l(t,y)+\frac{\lambda|y|^2}{1+|y|^2}\E^{(2)}_l(t,y)\le 0. 
	\end{align*}
	Then by solving this ODE, we have 
	\begin{align*}
	\E^{(2)}_l(t,y)\lesssim e^{-\frac{\lambda|y|^2t}{1+|y|^2}}\E^{(2)}_l(0,y).
	\end{align*}
	By using \eqref{444}, 
	\begin{align}\label{51a}
	\|\nabla^m_xw^lf\|^2_{L^2_{v,x}}&+\|\nabla^m_xE\|^2_{L^2_x}\approx \int_{\R^3}|y|^{2m}\E^{(2)}_l(t,y)\,dy\\
	&\lesssim \int_{|y|\le 1}|y|^{2m}e^{-\lambda|y|^2t}\E^{(2)}_l(0,y)\,dy+e^{-\lambda t}\int_{|y|\ge 1}|y|^{2m}\E^{(2)}_l(0,y)\,dy.\notag
	\end{align}
	By H\"{o}lder's inequality and scaling on $y$, one has 
	\begin{align*}
	\int_{|y|\le 1}|y|^{2m}e^{-\lambda|y|^2t}\E^{(2)}_l(0,y)\,dy
	&\lesssim \min\{1,t^{-3/2-m}\}\|\E^{(2)}_l(0,y)\|_{L^\infty_y}.
	\end{align*}
	For the case $|y|\ge 1$, noticing \eqref{8}, we have 
	\begin{align*}
	\||y|^m\widehat{E_0}\|_{L^2_y(|y|\ge 1)}\lesssim \||y|^{m-1}\widehat{f_0}\|_{L^2_y(|y|\ge 1)}\lesssim \||y|^{m}\widehat{f_0}\|_{L^2_y(|y|\ge 1)},
	\end{align*}
	which yields that 
	\begin{align*}
	e^{-\lambda t}\int_{|y|\ge 1}|y|^{2m}\E^{(2)}_l(0,y)\,dy\lesssim e^{-\lambda t}\|w^l\nabla^m_xf_0\|^2_{L^2_{v,x}}.
	\end{align*}
	Thus, \eqref{51a} becomes 
	\begin{align*}
	\|\nabla^m_xw^lf\|^2_{L^2_{v,x}}+\|\nabla^m_xE\|^2_{L^2_x}\lesssim (1+t)^{-3/2-m}\big(\|w^lf_0\|^2_{Z_1}+\|E_0\|^2_{L_1}+\|w^l\nabla^m_xf_0\|^2_{L^2_{v,x}})
	\end{align*}
	This completes the case of hard potential. 

{\bf Step 2}. For soft potential $-1\le\gamma+2<0$, we denote $\E_l=\E^{(2)}_l$ for brevity. By using \eqref{45aa} and \eqref{10w}, for $l\ge 0$, 
	\begin{align}\label{46a}
		\partial_t\E_l(t,y)+\frac{\lambda|y|^2}{1+|y|^2}(\|w^{l-\frac{\gamma+2}{2\gamma}}\widehat{f}\|^2_{L^2_v}+|\widehat{E}|^2)\le 0. 
	\end{align}
	Since $\gamma+2< 0$, we need to use the trick in \cite{Strain2012}. In fact, for $j>0$, using H\"{o}lder's inequality, we have 
	\begin{align*}
		\E_l(t,y)\lesssim \E^{\frac{j}{j+1}}_{l-\frac{\gamma+2}{2\gamma}}(t,y)\E^{\frac{1}{j+1}}_{l+j\frac{\gamma+2}{2\gamma}}(t,y).
	\end{align*}
	Thus,
	\begin{align*}
		\E^{\frac{j+1}{j}}_l(t,y)\lesssim \E_{l-\frac{\gamma+2}{2\gamma}}(t,y)\E^{\frac{1}{j}}_{l+j\frac{\gamma+2}{2\gamma}}(t,y)\lesssim \E_{l-\frac{\gamma+2}{2\gamma}}(t,y)\E^{\frac{1}{j}}_{l+j\frac{\gamma+2}{2\gamma}}(0,y),
	\end{align*}where the second inequality follows from $\partial_t\E_{l+j\frac{\gamma+2}{2\gamma}}(t,y)\le 0$. 
	Now \eqref{46a} becomes 
	\begin{align*}
		\partial_t\E_l(t,y)+\frac{\lambda|y|^2}{1+|y|^2}\E^{\frac{j+1}{j}}_l(t,y)\E^{-\frac{1}{j}}_{l+j\frac{\gamma+2}{2\gamma}}(0,y)\le 0. 
	\end{align*}
	Then by solving this ODE, we have 
	\begin{align*}
		-j\E^{-\frac{1}{j}}_l(t,y)+j\E^{-\frac{1}{j}}_l(0,y)\le -\frac{t\lambda|y|^2}{1+|y|^2}\E^{-\frac{1}{j}}_{l+j\frac{\gamma+2}{2\gamma}}(0,y),\\
		\E_l(t,y)\le \Big(1+\frac{t\lambda|y|^2}{j(1+|y|^2)}\Big)^{-j}\E_{l+j\frac{\gamma+2}{2\gamma}}(0,y).
	\end{align*}
	Together with \eqref{444}, 
	\begin{align}\label{51}
		\|\nabla^m_x&w^lf\|^2_{L^2_{v,x}}+\|\nabla^m_xE\|^2_{L^2_x}\approx \int_{\R^3}|y|^{2m}\E_l(t,y)\,dy\\
		&\lesssim \int_{|y|\ge 1}|y|^{2m}(1+\frac{t\lambda}{2j})^{-j}\E_{l+j\frac{\gamma+2}{2\gamma}}(0,y)\,dy+e^{-\lambda t}\int_{|y|\le 1}|y|^{2m}(1+\frac{t\lambda|y|^2}{j(1+|y|^2)})^{-j}\E_{l+j\frac{\gamma+2}{2\gamma}}(0,y)\,dy.\notag
	\end{align}
	For the case $|y|\ge 1$, noticing \eqref{8}, we have 
	\begin{align*}
		\||y|^m\widehat{E_0}\|_{L^2_y(|y|\ge 1)}\lesssim \||y|^{m-1}\widehat{f_0}\|_{L^2_{v,y}(|y|\ge 1)}\lesssim \||y|^{m}\widehat{f_0}\|_{L^2_{v,y}(|y|\ge 1)},
	\end{align*}
	which yields that 
	\begin{align*}
		\int_{|y|\ge 1}|y|^{2m}(1+\frac{t\lambda}{2j})^{-j}\E_{l+j\frac{\gamma+2}{2\gamma}}(0,y)\,dy\lesssim (1+t)^{-j}\|w^{l+j\frac{\gamma+2}{2\gamma}}\nabla^m_xf_0\|^2_{L^2_{v,x}}.
	\end{align*}
	On the other hand, by H\"{o}lder's inequality and scaling on $y$, one has 
	\begin{align*}
		\int_{|y|\le 1}|y|^{2m}(1+\frac{t\lambda|y|^2}{j(1+|y|^2)})^{-j}\E_{l+j\frac{\gamma+2}{2\gamma}}(0,y)\,dy
		&\lesssim \int_{|y|\le 1}|y|^{2m}(1+\frac{t\lambda|y|^2}{j(1+|y|^2)})^{-j}\,dy\|\E_{l+j\frac{\gamma+2}{2\gamma}}(0,y)\|_{L^\infty_y}\\
		&\lesssim(1+t)^{-2\sigma_{m}}\|\E_{l+j\frac{\gamma+2}{2\gamma}}(0,y)\|_{L^\infty_y},
	\end{align*}for any $j>2\sigma_m$. 
	Thus, \eqref{51} becomes 
	\begin{align*}
		\|\nabla^m_xw^lf\|^2_{L^2_{v,x}}+\|\nabla^m_xE\|^2_{L^2_x}\lesssim (1+t)^{-2\sigma_m}\big(\|w^{l+j\frac{\gamma+2}{2\gamma}}f_0\|^2_{Z_1}+\|E_0\|^2_{L_1}+\|w^{l+j\frac{\gamma+2}{2\gamma}}\nabla^m_xf_0\|^2_{L^2_{v,x}}).
	\end{align*}
	This completes the proof. 
	\qe\end{proof}

\section{Global Existence}\label{sec4}
Assume $T>0$, $K\ge 3$ and $l\ge K$. In this section, we are going to prove the Main Theorem \ref{main1}, the global existence of the solution to the following system. 
\begin{equation}\label{16b}
\left\{\begin{aligned}
&\partial_tf_\pm + v_i\partial^{e_i}f_\pm \pm \frac{1}{2}\partial^{e_i}\phi v_if_\pm  \mp\partial^{e_i}\phi\partial_{e_i}f_\pm \pm \partial^{e_i}\phi  v_i\mu^{1/2} - L_\pm f = \Gamma_{\pm}(f,f),\\
&-\Delta_x \phi = \int_{\Rd}(f_+-f_-)\mu^{1/2}\,dv,\\ 
&f_\pm|_{t=0} = f_{0,\pm}. 
\end{aligned}\right.
\end{equation}The index appearing in both superscript and subscript means the summation. Our goal is to obtain the $a$ $priori$ from this equation. 
For this, we suppose that the Cauchy problem \eqref{16b} admits a smooth solution $f(t,x,v)$ over $0\le t\le T$ for $0<T\le\infty$, and the solution $f(t,x,v)$ satisfies 
\begin{align}\label{74a}
\sup_{0\le t\le T}X(t)\le \delta_0,
\end{align} where $X(t)$ is defined by \eqref{88} and $\delta_0$ is a suitably small constant. 
Under this assumption, we can derive a simple fact that 
\begin{align*}
\|\phi\|_{L^\infty}\lesssim\|\phi\|_{H^2_x}\le \delta_0, \quad \|e^{\pm\phi}\|_{L^\infty}\approx 1.
\end{align*}
Also, by equation \eqref{21}$_1$, we have 
\begin{equation}\label{34a}
\partial_t\phi = -\Delta_x^{-1}\partial_t(a_+-a_-)=\Delta_x^{-1}\nabla\cdot G,
\end{equation}
\begin{equation}\label{34}
\|\partial_t\phi\|_{L^\infty}\lesssim  \|\nabla_x\partial_t\phi\|^{1/2}_{L^2_x}\|\nabla^2_x\partial_t\phi\|^{1/2}_{L^2_x}\lesssim \|\nabla_x G\|_{H^1}\lesssim \|(\I-\P)f\|_{L^2_vH^2_x} \lesssim (\E^h_{K,l})^{1/2}(t). 
\end{equation}

\begin{Thm}\label{thm411}Assume $\psi=1$. For any $l\ge K\ge 3$, there is $\E_{K,l}$ satisfying \eqref{Defe} such that for $0\le t\le T$,
	\begin{align}\label{42a}
	\partial_t\E_{K,l}(t)+\lambda D_{K,l}(t)\lesssim \|\partial_t\phi\|_{L^\infty_x}\E_{K,l}(t), 
	\end{align}where $D_{K,l}$ is defined by \eqref{Defd}. 
\end{Thm}
\begin{proof}
	For any $K\ge 3$ being the total derivative of $v,x$, we let $|\alpha|+|\beta|\le K$.
	On one hand, we apply $\partial^\alpha$ to equation \eqref{16b}$_1$ to get 
	\begin{equation}\begin{aligned}\label{35a}
	&\quad\,\partial_t\partial^{\alpha} f_\pm + v_i\partial^{e_i+\alpha} f_\pm \pm \frac{1}{2}\sum_{\substack{\alpha_1\le\alpha}}\partial^{e_i+\alpha_1}\phi v_i\partial^{\alpha-\alpha_1}f_\pm \\ &\qquad\mp\sum_{\substack{\alpha_1\le\alpha}}\partial^{e_i+\alpha_1}\phi\partial^{\alpha-\alpha_1}_{e_i} f_\pm \pm \partial^{e_i+\alpha}\phi v_i\mu^{1/2} - \partial^{\alpha} L_\pm  f =
	\partial^{\alpha} \Gamma_{\pm}(f,f).\end{aligned}
	\end{equation}
	On the other hand, we apply $\partial^{\alpha}_\beta$ to equation \eqref{16b}$_1$ and decompose $f_\pm=\PP f+(\II-\PP)f$. Then, 
	\begin{align}\label{36a}
	&\quad\,\partial_t\partial^{\alpha}_\beta (\II-\PP)f + \sum_{\beta_1\le \beta}C^{\beta_1}_{\beta}\partial_{\beta_1}v_i\partial^{e_i+\alpha}_{\beta-\beta_1}(\II-\PP)f \notag\\
	&\notag\qquad\pm \frac{1}{2}\sum_{\substack{\alpha_1\le\alpha}}\sum_{\beta_1\le\beta}\partial^{e_i+\alpha_1}\phi \partial_{\beta_1}v_i\partial^{\alpha-\alpha_1}_{\beta-\beta_1}(\II-\PP)f \\ &\qquad\mp\sum_{\substack{\alpha_1\le\alpha}}\partial^{e_i+\alpha_1}\phi\partial^{\alpha-\alpha_1}_{\beta+e_i}(\II-\PP)f \pm \partial^{e_i+\alpha}\phi \partial_\beta(v_i\mu^{1/2}) - \partial^{\alpha}_\beta L_\pm (\I-\P)f \\
	&\notag= -\partial_t\partial^{\alpha}_\beta \PP f + \sum_{\beta_1\le \beta}C^{\beta_1}_{\beta}\partial_{\beta_1}v_i\partial^{e_i+\alpha}_{\beta-\beta_1}\PP f \mp \frac{1}{2}\sum_{\substack{\alpha_1\le\alpha}}\sum_{\beta_1\le\beta}\partial^{e_i+\alpha_1}\phi \partial_{\beta_1}v_i\partial^{\alpha-\alpha_1}_{\beta-\beta_1}\PP f \\ &\notag\qquad\mp\sum_{\substack{\alpha_1\le\alpha}}\partial^{e_i+\alpha_1}\phi\partial^{\alpha-\alpha_1}_{\beta+e_i}\PP f
	+
	\partial^{\alpha}_\beta \Gamma_{\pm}(f,f).
	\end{align}

	{\bf Step 1. Estimate without weight.}
	For the estimate without weight, we take the case $|\alpha|\le K$ and $\beta=0$. This case is for obtaining the term $\|\partial^\alpha\nabla_x\phi\|^2_{L^2_x}$ on the left hand side in the energy inequality. Taking inner product of equation \eqref{35a} with $e^{\pm\phi}\partial^{\alpha} f_\pm$ over $\R^3_v\times\R^3_x$, we have   
	\begin{align}\label{45a}
	&\notag\quad\,\Big(\partial_t\partial^{\alpha} f_\pm,e^{\pm\phi}\partial^{\alpha} f_\pm\Big)_{L^2_{v,x}}
	+ \Big(v_i\partial^{e_i+\alpha}f_\pm,e^{\pm\phi}\partial^{\alpha} f_\pm\Big)_{L^2_{v,x}}\\ 
	&\notag\pm \Big(\frac{1}{2}\sum_{\substack{\alpha_1\le\alpha}}C^{\alpha_1}_{\alpha}\partial^{e_i+\alpha_1}\phi v_i\partial^{\alpha-\alpha_1}f_\pm,e^{\pm\phi}\partial^{\alpha} f_\pm\Big)_{L^2_{v,x}} \\ 
	&\mp
	\Big(\sum_{\substack{\alpha_1\le\alpha}}C^{\alpha_1}_{\alpha}\partial^{e_i+\alpha_1}\phi\partial^{\alpha-\alpha_1}_{e_i}f_\pm,e^{\pm\phi}\partial^{\alpha} f_\pm\Big) _{L^2_{v,x}}\\
	&\notag\pm \Big(\partial^{e_i+\alpha}\phi v_i\mu^{1/2},e^{\pm\phi}\partial^{\alpha} f_\pm\Big)_{L^2_{v,x}} 
	- \Big(\partial^{\alpha} L_\pm f,e^{\pm\phi}\partial^{\alpha} f_\pm\Big)_{L^2_{v,x}}\\ 
	&\notag= \Big(\partial^{\alpha} \Gamma_{\pm}(f,f),e^{\pm\phi}\partial^{\alpha} f_\pm\Big)_{L^2_{v,x}}.
	\end{align}
	Now we take the real part and summation $\sum_{\pm}$. Denote these them by $I_1$ to $I_7$ and we will estimate them term by term. 
	
	For the first term $I_1$,
	\begin{align}\label{37aa}
	I_1 &= \frac{1}{2}\partial_t\sum_{\pm}\|e^{\frac{\pm\phi}{2}}\partial^{\alpha} f_\pm\|^2_{L^2_{v,x}} \mp \Re\sum_{\pm}\frac{1}{2}(\partial_t\phi e^{\pm\phi}\partial^{\alpha} f_\pm, \partial^{\alpha} f_\pm)_{L^2_{v,x}}. 
	\end{align}The second term on the right hand side of \eqref{37aa} is estimated as 
	\begin{align}\label{78}
	\Big|\frac{1}{2}(\partial_t\phi e^{\pm\phi}\partial^{\alpha} f_\pm, \partial^{\alpha} f_\pm)_{L^2_{v,x}}\Big|\lesssim \|\partial_t\phi\|_{L^\infty}\|\partial^\alpha f_\pm\|^2_{L^2_{v,x}}\lesssim \|\partial_t\phi\|_{L^\infty}\E_{K,l}(t)
	\end{align}

	For the second term $I_2$, we will combine it with $I_3$ with $\alpha_1=0$ in $I_3$. It turns out that the sum is zero. This is what $e^{\pm\phi}$ designed for as in \cite{Guo2012}. By taking integration by parts on $x$, one has  
	\begin{align}\label{48aaa}
	&\quad\,\Big(v_i\partial^{e_i+\alpha}f_\pm,e^{\pm\phi}\partial^{\alpha} f_\pm\Big)_{L^2_{v,x}}
	\pm \Big(\frac{1}{2}\partial^{e_i}\phi v_i\partial^{\alpha}f_\pm,e^{\pm\phi}\partial^{\alpha} f_\pm\Big)_{L^2_{v,x}}=0.
	\end{align}
	For the left terms in $I_3$, the weight will be used. In this case, $\alpha_1$ is not zero and by Lemma \ref{Lem26}, it's bounded above by $\E^{1/2}_{K,l}\D_{K,l}$. 
	Using Lemma \ref{Lem27}, the term $I_4$ is also bounded above by $\E^{1/2}_{K,l}\D_{K,l}$.

	For the term $I_5$, we will divide $e^{\pm\phi}$ into $(e^{\pm\phi}-1)$ and $1$. Recall equation \eqref{16} and \eqref{21}. For the part of $1$, 
	\begin{align}\label{43a}\notag
	\sum_{\pm}\pm\Big(\partial^{e_i+\alpha}\phi v_i\mu^{1/2},\partial^{\alpha} f_\pm\Big)_{L^2_{v,x}} 
	&=\notag -\Big(\partial^{\alpha}\phi,\partial^{\alpha} \nabla_x\cdot G\Big)_{L^2_{x}}\\
	&=\notag \Big(\partial^{\alpha}\phi,\partial^{\alpha} \partial_t(a_+-a_-)\Big)_{L^2_{x}}\\
	&= \frac{1}{2}\partial_t\|\partial^{\alpha}\nabla_x\phi\|_{L^2_x}^2.
	\end{align}
	For the part of $(e^{\pm\phi}-1)$, notice that 
	\begin{align*}
	|e^{\pm\phi}-1|\lesssim \|\phi\|_{L^\infty}\lesssim \|\nabla_x\phi\|_{H^1_x}.
	\end{align*}Then, 
	\begin{align}
	\Big|&\sum_{\pm}\pm\Big(\partial^{e_i+\alpha}\phi v_i\mu^{1/2},(e^{\pm\phi}-1)\partial^{\alpha} f_\pm\Big)_{L^2_{v,x}}\Big|\notag\\
	&\lesssim \|\nabla_x\phi\|_{H^1_x}\sum_{|\alpha|\le K}\|\partial^\alpha\nabla_x\phi\|_{L^2_{v,x}}\sum_{|\alpha|\le K}\|\partial^\alpha(\II-\PP)f\|_{L^2_{v,x}}\\
	&\lesssim \E^{1/2}_{K,l}(t)\D_{K,l}(t).\notag
	\end{align}
	
	For the term $I_6$, since $L_\pm$ commutes with $\partial^{\alpha}$ and $e^{\pm\phi}$, by Lemma \ref{lemmaL}, we have 
	\begin{align}
	I_6 = - \sum_{\pm}\Big(\partial^{\alpha} L_\pm f,e^{\pm\phi}\partial^{\alpha} f_\pm\Big)_{L^2_{v,x}}\ge \lambda \sum_{\pm}\|e^{\frac{\pm\phi}{2}}\partial^{\alpha}(\II-\PP) f\|_{\sigma,0}^2. 
	\end{align}
	
	For the term $I_7$, by Lemma \ref{lemmag}, we have 
	\begin{align}
	|I_7|&= \Big|\sum_{\pm}\Big(\partial^{\alpha} \Gamma_{\pm}(f,f),e^{\pm\phi}\partial^{\alpha} f_\pm\Big)_{L^2_{v,x}}\Big|\lesssim\E^{1/2}_{K,l}(t)\D_{K,l}(t).
	\end{align}
	
	Therefore, combining all the estimate above and take the summation on $\pm$, $|\alpha|\le K$, noticing that $|e^{\frac{\pm\phi}{2}}|\approx 1$, we conclude that, 
	\begin{equation}\label{47b}
	\begin{aligned}
	&\quad\,\frac{1}{2}\partial_t\sum_{\pm}\sum_{|\alpha|\le K}\Big(\|e^{\frac{\pm\phi}{2}}\partial^{\alpha} f_\pm\|_{L^2_{v,x}} +
	\|\partial^{\alpha}\nabla_x\phi\|_{L^2_x}^2\Big) + \lambda \sum_{\pm}\sum_{|\alpha|\le K}\|e^{\frac{\pm\phi}{2}}\partial^{\alpha} (\II-\PP)f\|_{\sigma,0}^2\\
	&\lesssim \|\partial_t\phi\|_{L^\infty}\E_{K,l}(t)+\E^{1/2}_{K,l}(t)\D_{K,l}(t).
	\end{aligned}
	\end{equation}
	Taking the combination $\eqref{47b}+\kappa\times\eqref{24}$ with $0<\kappa<<1$, we have that when $\psi=1$, 
	\begin{align}\label{48a}\notag
	&\quad\,\frac{1}{2}\partial_t\sum_{\pm}\sum_{|\alpha|\le K}\Big(\|e^{\frac{\pm\phi}{2}}\partial^{\alpha} f_\pm\|_{L^2_{v,x}} +
	\|\partial^{\alpha}\nabla_x\phi\|_{L^2_x}^2+\kappa\E^{(1)}_{K}\Big) + \lambda \sum_{\pm}\sum_{|\alpha|\le K}\|e^{\frac{\pm\phi}{2}}\partial^{\alpha} (\II-\PP)f\|_{\sigma,0}^2\\
	&\notag\qquad + \lambda\sum_{|\alpha|\le K-1} \|\partial^\alpha\nabla_x(a_\pm,b,c)\|^2_{L^2_x}+\lambda\|a_+-a_-\|^2_{L^2_x}+\lambda\sum_{|\alpha|\le K-1}\|\partial^\alpha E\|^2_{L^2_x}\\
	&\lesssim \|\partial_t\phi\|_{L^\infty}\E_{K,l}(t)+ (\E^{1/2}_{K,l}(t)+\E_{K,l}(t))\D_{K,l}(t)
	\end{align}
	The term $\|(\I-\P)\partial^{\alpha}\widehat{f}\|_{\sigma,0}^2$ in \eqref{24} is eliminated. 
	
	{\bf Step 2. Estimate with weight on $x$ derivatives}
	This case is particularly for $|\alpha|=K$. Let $1\le |\alpha|\le K$ and take inner product of \eqref{35a} with $e^{\pm\phi}w^{2l-2|\alpha|}\partial^\alpha f_\pm$ over $\R^3_v\times\R^3_x$.
	\begin{align}\label{77}
	&\notag\quad\,\Big(\partial_t\partial^{\alpha} f_\pm,w^{2l-2|\alpha|}e^{\pm\phi}\partial^{\alpha} f_\pm\Big)_{L^2_{v,x}}
	+ \Big(v_i\partial^{e_i+\alpha}f_\pm,w^{2l-2|\alpha|}e^{\pm\phi}\partial^{\alpha} f_\pm\Big)_{L^2_{v,x}}\\ 
	&\notag\pm \Big(\frac{1}{2}\sum_{\substack{\alpha_1\le\alpha}}C^{\alpha_1}_{\alpha}\partial^{e_i+\alpha_1}\phi v_i\partial^{\alpha-\alpha_1}f_\pm,w^{2l-2|\alpha|}e^{\pm\phi}\partial^{\alpha} f_\pm\Big)_{L^2_{v,x}} \\ 
	&\mp
	\Big(\sum_{\substack{\alpha_1\le\alpha}}C^{\alpha_1}_{\alpha}\partial^{e_i+\alpha_1}\phi\partial^{\alpha-\alpha_1}_{e_i}f_\pm,w^{2l-2|\alpha|}e^{\pm\phi}\partial^{\alpha} f_\pm\Big) _{L^2_{v,x}}\\
	&\notag\pm \Big(\partial^{e_i+\alpha}\phi v_i\mu^{1/2},w^{2l-2|\alpha|}e^{\pm\phi}\partial^{\alpha} f_\pm\Big)_{L^2_{v,x}} 
	- \Big(\partial^{\alpha} L_\pm f,w^{2l-2|\alpha|}e^{\pm\phi}\partial^{\alpha} f_\pm\Big)_{L^2_{v,x}}\\ 
	&\notag= \Big(\partial^{\alpha} \Gamma_{\pm}(f,f),w^{2l-2|\alpha|}e^{\pm\phi}\partial^{\alpha} f_\pm\Big)_{L^2_{v,x}}.
	\end{align}
	As in the Step 1, taking summation on $\pm$ and real part, we estimate it term by term. The proof is similar to $I_1$ to $I_7$. The first term on the left hand bounded below by 
	\begin{align*} &\quad\,\frac{1}{2}\partial_t\sum_{\pm}\|e^{\frac{\pm\phi}{2}}w^{l-|\alpha|}\partial^{\alpha} f_\pm\|_{L^2_{v,x}} - C\|\partial_t\phi\|_{L^2_{v,x}}\E_{K,l}(t). 
	\end{align*}
	The second term and the third term with $\alpha_1=0$ are canceled by using integration by parts. The left case $\alpha_1\neq 0$ in the third term and the fourth term are bounded above by $\E^{1/2}_{K,l}(t)\D_{K,l}(t)$ by using Lemma \ref{Lem26} and Lemma \ref{Lem27}.
	For the fifth term, we write a upper bound: for any $\eta>0$,
	\begin{align*}
	\Big|\Big(\partial^{e_i+\alpha}\phi v_i\mu^{1/2},w^{2l-2|\alpha|}e^{\pm\phi}\partial^{\alpha} f_\pm\Big)_{L^2_{v,x}}\Big| &\lesssim \eta\|\partial^{\alpha} f_\pm\|^2_{\sigma,l-|\alpha|}+C_\eta\|\partial^\alpha\nabla_x\phi\|^2_{L^2_x}.
	\end{align*}
	For the sixth term, since $L_\pm$ commutes with $\partial^{\alpha}$ and $e^{\pm\phi}$, by Lemma \ref{lemmaL}, we have 
	\begin{align*}
	&- \sum_{\pm}\Big(\partial^{\alpha}w^{2l-2|\alpha|} L_\pm f,e^{\pm\phi}\partial^{\alpha} f_\pm\Big)_{L^2_{v,x}}\ge \lambda \|e^{\frac{\pm\phi}{2}}\partial^{\alpha} f\|_{\sigma,l-|\alpha|}^2-C\|\partial^\alpha f\|^2_{\sigma,0}. 
	\end{align*}
	By using Lemma \ref{lemmag}, the first term on the right hand of \eqref{77} is bounded above by $\E^{1/2}_{K,l}(t)\D_{K,l}(t)$.
	Taking $\psi=1$, combining the above estimate, taking summation on $1\le |\alpha|\le K$ and letting $\eta$ suitably small, we have 
	\begin{align}\label{74}
	&\quad\,\notag\frac{1}{2}\partial_t\sum_{\pm}\sum_{1\le|\alpha|\le K}\|e^{\frac{\pm\phi}{2}}w^{l-|\alpha|}\partial^{\alpha} f_\pm\|_{L^2_{v,x}} + \lambda\sum_{\pm}\sum_{1\le|\alpha|\le K}\|e^{\frac{\pm\phi}{2}}\partial^{\alpha} f_\pm\|_{\sigma,l-|\alpha|}^2\\
	&\lesssim \|\partial_t\phi\|_{L^2_{v,x}}\E_{K,l}(t)+\sum_{|\alpha|\le K}\|\partial^\alpha\nabla_x\phi\|^2_{L^2_x}+\sum_{1\le|\alpha|\le K}\|\partial^\alpha f_\pm\|^2_{\sigma,0} + \E^{1/2}_{K,l}(t)\D_{K,l}(t).
	\end{align}

	{\bf Step 3. Estimate with weight on the mixed derivatives.}
	Let $K\ge 3$, $|\alpha|\le K-1$ and $|\alpha|+|\beta|\le K$. Taking inner product of equation \eqref{36a} with  $e^{\pm\phi}w^{2l-2|\alpha|-2|\beta|}\partial^{\alpha}_\beta (\II-\PP)f$ over $\R^3_v\times\R^3_x$, one has 
	\begin{align*}
	&\quad\,\Big(\partial_t\partial^{\alpha}_\beta (\II-\PP)f,e^{\pm\phi}w^{2l-2|\alpha|-2|\beta|}\partial^{\alpha}_\beta (\II-\PP)f\Big)_{L^2_{v,x}}\\
	&\qquad + \Big(\sum_{\beta_1\le \beta}C^{\beta_1}_{\beta}\partial_{\beta_1}v_i\partial^{e_i+\alpha}_{\beta-\beta_1}(\II-\PP)f,e^{\pm\phi}w^{2l-2|\alpha|-2|\beta|}\partial^{\alpha}_\beta (\II-\PP)f\Big)_{L^2_{v,x}} \\
	&\qquad\pm \Big(\frac{1}{2}\sum_{\substack{\alpha_1\le\alpha\\\beta_1\le\beta}}\partial^{e_i+\alpha_1}\phi \partial_{\beta_1}v_i\partial^{\alpha-\alpha_1}_{\beta-\beta_1}(\II-\PP)f,e^{\pm\phi}w^{2l-2|\alpha|-2|\beta|}\partial^{\alpha}_\beta (\II-\PP)f\Big)_{L^2_{v,x}} \\ &\qquad\mp\Big(\sum_{\substack{\alpha_1\le\alpha}}\partial^{e_i+\alpha_1}\phi\partial^{\alpha-\alpha_1}_{\beta+e_i}(\II-\PP)f,e^{\pm\phi}w^{2l-2|\alpha|-2|\beta|}\partial^{\alpha}_\beta (\II-\PP)f\Big)_{L^2_{v,x}}\\
	&\qquad \pm \Big(\partial^{e_i+\alpha}\phi \partial_\beta(v_i\mu^{1/2}),e^{\pm\phi}w^{2l-2|\alpha|-2|\beta|}\partial^{\alpha}_\beta (\II-\PP)f\Big)_{L^2_{v,x}}\\
	&\qquad - \Big(\partial^{\alpha}_\beta L_\pm (\I-\P)f,e^{\pm\phi}w^{2l-2|\alpha|-2|\beta|}\partial^{\alpha}_\beta (\II-\PP)f\Big)_{L^2_{v,x}} \\
	&= -\Big(\partial_t\partial^{\alpha}_\beta \PP f,e^{\pm\phi}w^{2l-2|\alpha|-2|\beta|}\partial^{\alpha}_\beta (\II-\PP)f\Big)_{L^2_{v,x}}\\
	&\qquad + \Big(\sum_{\beta_1\le \beta}C^{\beta_1}_{\beta}\partial_{\beta_1}v_i\partial^{e_i+\alpha}_{\beta-\beta_1}\PP f,e^{\pm\phi}w^{2l-2|\alpha|-2|\beta|}\partial^{\alpha}_\beta (\II-\PP)f\Big)_{L^2_{v,x}} \\
	&\qquad\mp \Big(\frac{1}{2}\sum_{\substack{\alpha_1\le\alpha}}\sum_{\beta_1\le\beta}\partial^{e_i+\alpha_1}\phi \partial_{\beta_1}v_i\partial^{\alpha-\alpha_1}_{\beta-\beta_1}\PP f,e^{\pm\phi}w^{2l-2|\alpha|-2|\beta|}\partial^{\alpha}_\beta (\II-\PP)f\Big)_{L^2_{v,x}} \\ &\qquad\mp\Big(\sum_{\substack{\alpha_1\le\alpha}}\partial^{e_i+\alpha_1}\phi\partial^{\alpha-\alpha_1}_{\beta+e_i}\PP f,e^{\pm\phi}w^{2l-2|\alpha|-2|\beta|}\partial^{\alpha}_\beta (\II-\PP)f\Big)_{L^2_{v,x}}\\&\qquad
	+
	\Big(\partial^{\alpha}_\beta \Gamma_{\pm}(f,f),e^{\pm\phi}w^{2l-2|\alpha|-2|\beta|}\partial^{\alpha}_\beta (\II-\PP)f\Big)_{L^2_{v,x}}.
	\end{align*}
	Now we denote these terms with summation $\sum_{\pm}$ and real part by $J_1$ to $J_{11}$ and estimate them term by term. 
	The estimate of $J_1$ to $J_3$ are similar to $I_1$ to $I_3$. 
	That is 
	\begin{align*}
	J_1 &\ge \partial_t\sum_{\pm}\|e^{\frac{\pm\phi}{2}}w^{l-|\alpha|-|\beta|}\partial^{\alpha}_\beta (\II-\PP)f\|_{L^2_{v,x}} - C\|\partial_t\phi\|_{L^\infty}\E^h_{K,l}(t).
	\end{align*}
By using Lemma \ref{Lem26} and Lemma \ref{Lem27}, we have 
	\begin{equation*}
	|J_2 + J_3+J_4|\lesssim  \E^{1/2}_{K,l}(t)\D_{K,l}(t),
	\end{equation*}where $J_2$ and $J_3$ with $\alpha_1=0$ are canceled by integration by parts on $\partial^{e_i}$.

	For the term $J_5$, we only need to have a upper bound.
	\begin{align*}\notag
	|J_5| &= 
	\Big|\sum_{\pm}\pm \Big(\partial^{e_i+\alpha}\phi \partial_\beta(v_i\mu^{1/2}),e^{\pm\phi}w^{2l-2|\alpha|-2|\beta|}\partial^{\alpha}_\beta (\II-\PP)f\Big)_{L^2_{v,x}}\Big|\\
	&\lesssim\sum_{|\alpha|\le K} \|\partial^\alpha\nabla_x\phi\|_{L^2_{v,x}}\sum_{\substack{|\alpha|+|\beta|\le K}}\|\partial^{\alpha}_{\beta}(\I-\P)f\|_{\sigma,l-|\alpha|-|\beta|}\\
	&\lesssim \eta\sum_{\substack{|\alpha|+|\beta|\le K}}\|\partial^{\alpha}_{\beta}(\I-\P)f\|^2_{\sigma,l-|\alpha|-|\beta|}+C_\eta\sum_{|\alpha|\le K} \|\partial^\alpha\nabla_x\phi\|_{L^2_{v,x}}^2.
	\end{align*}

	For the term $J_6$, since $L_\pm$ commutes with $e^{\pm\phi}$, by Lemma \ref{lemmaL}, we have 
	\begin{align*}
	J_6 &= - \sum_{\pm}\Big(\partial^{\alpha}_\beta L_\pm(\I-\P) f,e^{\pm\phi}w^{2l-2|\alpha|-2|\beta|}\partial^{\alpha}_\beta (\II-\PP)f\Big)_{L^2_{v,x}}\\&\ge \lambda\sum_{\pm} \|e^{\frac{\pm\phi}{2}}\partial^{\alpha}_\beta (\II-\PP)f\|_{\sigma,l-|\alpha|-|\beta|}^2-C_\eta\sum_{\pm}\|\partial^\alpha(\II-\PP)f\|^2_{\sigma,0}\\
	&\qquad -\eta\sum_{\pm}\sum_{|\beta_1|\le|\beta|}\|e^{\frac{\pm\phi}{2}}\partial^{\alpha}_{\beta_1}(\II-\PP)f\|^2_{\sigma,l-|\alpha|-|\beta_1|},
	\end{align*}for any $\eta>0$. 
	Here we use the fact that $\|w^{l-|\alpha|-|\beta|}(\cdot)\|_{L^2(B_{C_\eta})}\lesssim \|\cdot\|_{\sigma,0}$. 
	For $J_7$, $J_8$, using the exponential decay in $v$ and the conservation laws \eqref{17}, we have 
	\begin{align*}
	J_7 +J_8&\lesssim \eta\sum_{\pm}\|\partial^\alpha_\beta(\II-\PP)f\|_{\sigma,l-|\alpha|-|\beta|}+C_\eta\Big(\sum_{|\alpha|\le K}\|\partial^\alpha\nabla_x\phi\|^2_{L^2_x}\\&\quad+\sum_{|\alpha|\le K-1}\|\partial^\alpha\nabla_x(a_\pm,b,c)\|^2_{L^2_x}+\sum_{|\alpha|\le K}\|\partial^{\alpha}(\I-\P)f\|^2_{\sigma,0}+\E_{K,l}(t)\D_{K,l}(t)\Big).
	\end{align*}
	Notice that here we used $|\alpha|\le K-1$. 
	$J_9$ and $J_{10}$ can be controlled by using a similar argument to Lemma \ref{Lem26} and Lemma \ref{Lem27}, since there's exponential decay in $v$. Then one can derive 
	\begin{align*}
	|J_9+J_{10}|\lesssim \E^{1/2}_{K,l}(t)\D_{K,l}(t)
	\end{align*} 
	For the term $J_{11}$, by Lemma \ref{lemmag}, we have 
	\begin{align*}
	|J_{11}|&= \Big|\sum_{\pm}\Big(\partial^{\alpha}_\beta \Gamma_{\pm}(f,f),w^{2l-2|\alpha|-2|\beta|}e^{\pm\phi}\partial^{\alpha}_\beta (\II-\PP)f\Big)_{L^2_{v,x}}\Big|\lesssim\E^{1/2}_{K,l}\D_{K,l}.
	\end{align*}

	Therefore, combining all the estimate above and take the summation on $|\alpha|\le K-1$, $|\alpha|+|\beta|\le K$, noticing that $|e^{\frac{\pm\phi}{2}}|\approx 1$, and letting $\eta$ sufficiently small, we conclude that, when $\psi=1$, 
	\begin{align}\label{48}\notag
	&\quad\,\partial_t\sum_{\pm}\sum_{\substack{|\alpha|\le K-1\\|\alpha|+|\beta|\le K}}\|e^{\frac{\pm\phi}{2}}w^{l-|\alpha|-|\beta|}\partial^{\alpha}_\beta(\II-\PP)f\|^2_{L^2_{v,x}}\\
	&\qquad
	+\lambda \sum_{\pm}\sum_{\substack{|\alpha|\le K-1\\|\alpha|+|\beta|\le K}}\|e^{\frac{\pm\phi}{2}}\partial^{\alpha}_\beta (\I-\P)f\|_{\sigma,l-|\alpha|-|\beta|}^2\notag\\
	&\lesssim \sum_{|\alpha|\le K-1} \|\partial^\alpha\nabla_x\phi\|_{L^2_{x}}^2 +(\E^{1/2}_{K,l}(t)+\E_{K,l}(t))\D_{K,l}(t)+\|\partial_t\phi\|_{L^\infty}\E^h_{K,l}(t)\\
	&\qquad+\sum_{|\alpha|\le K-1}\|\partial^\alpha\nabla_x(a_\pm,b,c)\|^2_{L^2_x}+\sum_{|\alpha|\le K}\|\partial^{\alpha}(\I-\P)f\|^2_{\sigma,0}.\notag
	\end{align}
	The redundant terms on the right hand side will be eliminated by using \eqref{48a}.

	\paragraph{Step 4.}
	
	We are able to prove this theorem by taking the proper linear combination of those estimates obtained in the above steps. Taking combination $C_1\times\eqref{48a}+\eqref{74}+\eqref{48}$ with sufficiently large $C_1>0$, we have 
	\begin{align}\label{58}
	&\notag\quad\,\partial_t\E_{K,l}(t)
	\notag + C_1\lambda \sum_{\pm}\sum_{|\alpha|\le K}\|e^{\frac{\pm\phi}{2}}\partial^{\alpha} f_\pm\|_{\sigma,0}^2 + C_1\lambda\sum_{|\alpha|\le K-1} \|\partial^\alpha\nabla_x(a_\pm,b,c)\|^2_{L^2_x}\\
	&\notag\qquad+C_1\lambda\|a_+-a_-\|^2_{L^2_x}+C_1\lambda\sum_{|\alpha|\le K-1}\|\partial^\alpha E\|^2_{L^2_x}
	+\lambda \sum_{\pm}\sum_{\substack{|\alpha|+|\beta|\le K}}\|e^{\frac{\pm\phi}{2}}\partial^{\alpha}_\beta (\II-\PP)f\|_{\sigma,l-|\alpha|-|\beta|}^2\notag\\
	&\lesssim \|\partial_t\phi\|_{L^\infty}\E_{K,l}(t)+ (\E^{1/2}_{K,l}(t)+\E_{K,l}(t))\D_{K,l}(t). 
	\end{align}
	where 
	\begin{align}
	\E_{K,l}(t)&=\notag \frac{C_1}{2}\sum_{\pm}\sum_{|\alpha|\le K}\|e^{\frac{\pm\phi}{2}}\partial^{\alpha} f_\pm\|_{L^2_{v,x}} +\sum_{\pm}\sum_{\substack{|\alpha|+|\beta|\le K}}\|e^{\frac{\pm\phi}{2}}w^{l-|\alpha|-|\beta|}\partial^{\alpha}_\beta(\II-\PP)f\|^2_{L^2_{v,x}}\\&\qquad+
	C_1\sum_{|\alpha|\le K}\|\partial^{\alpha}\nabla_x\phi\|_{L^2_x}^2+\kappa\E^{(1)}_{K}.
	\end{align}The second to fourth terms on the left hand side of \eqref{58} is larger than $D_{K,l}$.
	Notice that here for the term $\sum_{|\alpha|= K}\|\partial^\alpha\nabla_x\phi\|$, we use the fact that by \eqref{8},  
		\begin{align}\label{47a}
		\|\nabla^{K}_x\nabla_x\phi\|_{L^2_x}\lesssim \|\nabla^{K}_x\nabla_x\Delta^{-1}_x(a_+-a_-)\|_{L^2_x}\lesssim \sum_{\pm}\|\nabla^{K-1}_xa_\pm\|_{L^2_{x}}\lesssim \D^{1/2}_{K,l},
		\end{align}
	 and, hence $\sum_{|\alpha|= K}\|\partial^\alpha\nabla_x\phi\|$ can be eliminated by using $C_1\times\eqref{48a}$. 
	Noticing \eqref{27a} and $\kappa<<1$, it's direct to see that 
	\begin{align*}
	\E_{K,l}(t)\notag &\approx \sum_{|\alpha|\le K}\|\partial^\alpha E(t)\|^2_{L^2_x}+\sum_{|\alpha|\le K}\|\partial^\alpha\P f\|^2_{L^2_{v,x}}+\sum_{\substack{|\alpha|+|\beta|\le K}}\|w^{l-|\alpha|-|\beta|}\partial^\alpha_\beta(\I-\P) f\|^2_{L^2_{v,x}}
	\end{align*}
	Recalling the $a$ $priori$ assumption \eqref{74a}, the desired estimate \eqref{42a} follows directly from \eqref{58}. 
\end{proof}

For the higher order instant energy, we have the following theorem. 
\begin{Thm}\label{thm42}For any $l\ge K$, there is $\E^h_{K,l}(t)$ satisfying \eqref{Defeh} such that for any $0\le t\le T$,
	\begin{align}\label{91}
	\partial_t\E^h_{K,l}+\lambda \D_{K,l}(t)\lesssim \|\partial_t\phi\|_{L^\infty}\E^h_{K,l}(t)+\|\nabla_x(a_\pm,b,c)\|^2_{L^2_x}, 
	\end{align}where $\D_{K,l}$ is defined by \eqref{Defd}. 
\end{Thm}
\begin{proof}
	By letting $|\alpha|\ge1$ in \eqref{45a}, repeating the calculations from \eqref{45a} to \eqref{47b}, we can instead obtain 
	\begin{align}\label{63}
	&\notag\quad\,\frac{1}{2}\partial_t\sum_{\pm}\sum_{1\le|\alpha|\le K}\Big(\|e^{\frac{\pm\phi}{2}}\partial^{\alpha} f_\pm\|_{L^2_{v,x}} +
	\|\partial^{\alpha}\nabla_x\phi\|_{L^2_x}^2\Big) + \lambda \sum_{\pm}\sum_{1\le|\alpha|\le K}\|e^{\frac{\pm\phi}{2}}\partial^{\alpha} (\II-\PP)f\|_{\sigma,0}^2
	\\&\lesssim \|\partial_t\phi\|_{L^\infty}\E^h_{K,l}(t)+\E^{1/2}_{K,l}(t)\D_{K,l}(t),
	\end{align}
	Notice that here the first right-hand term contains $\E^h_{K,l}$ since there's at least one derivative on $x$ in the estimate of \eqref{78}. 
	We will combine this with \eqref{74} and \eqref{48}. In order to eliminate the term $\|(\I-\P)f\|^2_{\sigma,0}$ in \eqref{74} and \eqref{48}, we shall take the inner product of \eqref{36a} with $e^{\pm\phi}(\II-\PP)f$ over $\R^3_v\times\R^3_x$ and $\alpha=\beta=0$. 
	\begin{align*}
	&\quad\,\big(\partial_t(\II-\PP)f,e^{\pm\phi}(\II-\PP)f\big)_{L^2_{v,x}} + \big(v_i\partial^{e_i}(\II-\PP)f,e^{\pm\phi}(\II-\PP)f\big)_{L^2_{v,x}} \\
	&\qquad\pm \big(\frac{1}{2}\partial^{e_i}\phi v_i(\II-\PP)f,e^{\pm\phi}(\II-\PP)f\big)_{L^2_{v,x}} 
	\mp\big(\partial^{e_i}\phi\partial_{e_i}(\II-\PP)f,e^{\pm\phi}(\II-\PP)f\big)_{L^2_{v,x}}\\&\qquad \pm \big(\partial^{e_i}\phi v_i\mu^{1/2},e^{\pm\phi}(\II-\PP)f\big)_{L^2_{v,x}} -  \big(L_\pm (\I-\P)f,e^{\pm\phi}(\II-\PP)f\big)_{L^2_{v,x}} \\
	&= -\big(\partial_t\PP f,e^{\pm\phi}(\II-\PP)f\big)_{L^2_{v,x}} + \big(v_i\partial^{e_i}\PP f,e^{\pm\phi}(\II-\PP)f\big)_{L^2_{v,x}} \\&\qquad\mp \big(\frac{1}{2}\partial^{e_i}\phi v_i\PP f,e^{\pm\phi}(\II-\PP)f\big)_{L^2_{v,x}} \mp\big(\partial^{e_i}\phi\partial^{}_{e_i}\PP f,e^{\pm\phi}(\II-\PP)f\big)_{L^2_{v,x}}\\ &\qquad
	+
	\big(\partial^{\alpha}_\beta \Gamma_{\pm}(f,f),e^{\pm\phi}(\II-\PP)f\big)_{L^2_{v,x}}.
	\end{align*}
	As before, we denote these terms with taking summation over $\pm$ and real part by $L_1,\dots,L_{11}$, and estimate them term by term. 
	\begin{equation*}
	L_1\ge \frac{1}{2}\partial_t\sum_{\pm}\|e^{\frac{\pm\phi}{2}}(\II-\PP)f\|^2_{L^2_{v,x}}-C\|\partial_t\phi\|_{L^\infty_x}\E^h_{K,l}(t). 
	\end{equation*}
	The same as \eqref{48aaa}, by integration by parts on $x_i$, $L_2+L_3=0$. By integration by parts on $v_i$, $L_4=0$. 
	Similar to \eqref{43a}, 
	\begin{align*}
	L_5 = \frac{1}{2}\partial_t\|\nabla_x\phi\|^2_{L^2_x}. 
	\end{align*}
	By Lemma \ref{lemmaL}, 
	\begin{align*}
	L_6 \ge \lambda\|(\I-\P)f\|^2_{\sigma,0}. 
	\end{align*}
	Recalling the conservation laws \eqref{19}, one has 
	\begin{align*}
	L_7 \le \frac{\lambda}{4}\|(\I-\P)f\|^2_{\sigma,0}+C\big(\|\nabla_x(a_\pm,b,c)\|^2_{L^2_x}+\|\nabla_x\phi\|^2_{L^2_x}+\|\nabla_x(\I-\P)f\|^2_{\sigma,0}+\E_{K,l}\D_{K,l}\big).
	\end{align*}
	By Cauchy-Schwarz inequality, 
	\begin{align*}
	L_8\le \frac{\lambda}{4}\|(\I-\P)f\|^2_{\sigma,0}+C\|\nabla_x(a_\pm,b,c)\|^2_{L^2_{v,x}}. 
	\end{align*}
	Similar to the calculation on $J_9,J_{10},J_{11}$, we have that $L_9,L_{10},L_{11}$ are bounded above by $\E^{1/2}_{K,l}\D_{K,l}$. 
	Combining the above estimate, we have 
	\begin{align}\label{64}
	&\quad\,\notag\frac{1}{2}\partial_t\sum_{\pm}\|e^{\frac{\pm\phi}{2}}(\II-\PP)f\|^2_{L^2_{v,x}}+\frac{1}{2}\partial_t\|\nabla_x\phi\|^2_{L^2_x}+\lambda\|(\I-\P)f\|^2_{\sigma,0}\\
	&\lesssim \|\partial_t\phi\|_{L^\infty_x}\E^h_{K,l}(t)+\|\nabla_x(a_\pm,b,c)\|^2_{L^2_{v,x}}+\|\nabla_x(\I-\P)f\|^2_{\sigma,0}+(\E^{1/2}_{K,l}+\E_{K,l})\D_{K,l}
	\end{align}
	
	Now we use combination $C_2\times(C_1\times(\eqref{63}+\kappa\times\eqref{24a})+\eqref{64})+\eqref{74}+\eqref{48}$ with $\kappa<<1$. Taking $C_1>>\frac{1}{\kappa}$ sufficiently large then taking $C_2$ sufficiently large, we obtain that when $\psi=1$, 
	\begin{align}
	\notag&\partial_t\E^h_{K,l}(t) + C_1C_2\kappa\lambda\sum_{1\le|\alpha|\le K-1} \|\partial^\alpha\nabla_x(a_\pm,b,c)\|^2_{L^2_x}+C_1C_2\kappa\|\nabla_x(a_+-a_-)\|^2_{L^2_x}\\
	\notag&\qquad +C_1C_2\sum_{|\alpha|\le K-1}\|\partial^\alpha E\|^2_{L^2_x}+ C_1C_2\lambda \sum_{\pm}\sum_{1\le|\alpha|\le K}\|e^{\frac{\pm\phi}{2}}\partial^{\alpha} f_\pm\|_{\sigma,0}^2\\
	\notag&\qquad+C_2\lambda\|(\I-\P)f\|^2_{\sigma,0}
	+\lambda \sum_{\pm}\sum_{\substack{|\alpha|+|\beta|\le K}}\|e^{\frac{\pm\phi}{2}}\partial^{\alpha}_\beta (\I-\P)f\|_{\sigma,l-|\alpha|-|\beta|}^2\notag\\
	&\lesssim\|\partial_t\phi\|_{L^\infty}\E^h_{K,l}(t)+(\E^{1/2}_{K,l}(t)+\E_{K,l}(t))\D_{K,l}(t)+\|\nabla_x(a_\pm,b,c)\|^2_{L^2_x},\label{94}
	\end{align}
	where left-hand terms except the first one adding $\|\nabla_x(a_\pm,b,c)\|^2_{L^2_x}$ is larger than $\D_{K,l}$ and 
	\begin{align*}
	\E^h_{K,l}(t) &= C_1C_2\kappa\E^{(1)}_{K,h}+\frac{C_1C_2}{2}\sum_{\pm}\sum_{1\le|\alpha|\le K}\Big(\|e^{\frac{\pm\phi}{2}}\partial^{\alpha} f_\pm\|_{L^2_{v,x}} +
	C_1C_2\|\partial^{\alpha}\nabla_x\phi\|_{L^2_x}^2\Big)\\
	&\qquad+\frac{C_2}{2}\sum_{\pm}\|e^{\frac{\pm\phi}{2}}(\II-\PP)f\|^2_{L^2_{v,x}}+\frac{C_2}{2}\|\nabla_x\phi\|^2_{L^2_x}\\
	&\qquad+\sum_{\pm}\sum_{\substack{|\alpha|+|\beta|\le K}}\|e^{\frac{\pm\phi}{2}}w^{l-|\alpha|-|\beta|}\partial^{\alpha}_\beta(\II-\PP)f\|^2_{L^2_{v,x}}.
	\end{align*}
	Noticing $\kappa<<1$ is sufficiently small, it's direct to verify \eqref{Defeh}. 
	At last, by using the $a$ $priori$ assumption \eqref{74a}, we obtain the desired estimate \eqref{91} from \eqref{94}. 
\end{proof}

\begin{Thm}Let $T\in(0,\infty]$, $p\in(1/2,1)$ and $K\ge 3$. 
	Consider the solution $f$ to Cauchy problem \eqref{7}-\eqref{9}.
Assume $l_0\ge K$ satisfies that 
	\begin{equation}l_0\ge\left\{\begin{aligned}
			& \frac{3(\gamma+2)}{-4\gamma}+3,\  \text{ if }-1\le\gamma+2<0,\\
			& 2(\gamma+2)+3,\quad\text{ if }\gamma+2\ge0.
		\end{aligned}\right.
	\end{equation} Define 
	\begin{equation*}
	X(t) = \left\{\begin{aligned}
		&\sup_{0\le\tau\le t}(1+\tau)^{3/2}\E_{K,l_0}(\tau)+\sup_{0\le\tau\le t}(1+\tau)^{5/2}\E^h_{K,l_0}(\tau),\text{ if }\gamma+2\ge 0,\\
		\sup_{0\le\tau\le t}\E_{K,l_0+l_1}&(\tau)+\sup_{0\le\tau\le t}(1+\tau)^{3/2}\E_{K,l_0}(\tau)+\sup_{0\le\tau\le t}(1+\tau)^{3/2+p}\E^h_{K,l_0}(\tau),\text{ if }-1\le\gamma+2< 0,
	\end{aligned}\right.
	\end{equation*}and
	\begin{equation*}
	\epsilon_0 = (\E_{K,l_0+l_1}(0))^{1/2}+\|w^{l_2}f_0\|_{Z_1}+\|E_0\|_{L^1},
	\end{equation*}
where $l_1=\frac{5(\gamma+2)}{4(1-p)\gamma}$, $l_2=\frac{5(\gamma+2)}{4\gamma}$ for soft potential $-1\le\gamma+2< 0$ and $l_1=l_2=0$ for hard potential $\gamma+2\ge 0$. 
	Then under the $a$ $priori$ assumption \eqref{74a} for $\delta_0>0$ sufficiently small, we have 
	\begin{equation}\label{80}
	X(t)\lesssim \epsilon^2_0+X^{3/2}(t)+X^2(t),
	\end{equation}
	for any $0\le t\le T$. 
\end{Thm}
\begin{proof}
	{\bf Step 1. } From \eqref{34} and \eqref{42a}, we have that 
	\begin{align}\label{81}
		\|\partial_t\phi\|_{L^\infty_x}&\lesssim (\E^h_{K,l_0}(t))^{1/2}\lesssim(1+t)^{-3/4-p/2}X^{1/2}(t)\lesssim\delta^{1/2}_0(1+t)^{-3/4-p/2},\\
		&\partial_t\E_{K,l_0+l_1}(t)+D_{K,l_0+l_1}(t)\lesssim \|\partial_t\phi\|_{L^\infty_x}\E_{K,l_0+l_1}(t).\notag
	\end{align}
	By using Gronwall's inequality, for $0\le t\le T$, 
	\begin{align}\label{81a}
		\E_{K,l_0+l_1}(t)\lesssim \E_{K,l_0+l_1}(0)e^{C\int^t_0\|\partial_t\phi(\tau)\|_{L^\infty_x}}\,d\tau\lesssim \E_{K,l_0+l_1}(0)\lesssim \epsilon_0^2.
	\end{align}
	
	{\bf Step 2}. To prove the decay of $\E^h_{K,l_0}$, we use Theorem \ref{thm42} to get 
	\begin{align}\label{99a}
	\partial_t\E^h_{K,l_0}+\lambda \D_{K,l_0}(t)\lesssim \|\partial_t\phi\|_{L^\infty}\E^h_{K,l_0}(t)+\|\nabla_x(a_\pm,b,c)\|^2, 
	\end{align}
	We will use the trick in \cite{Strain2012}.
	Firstly, we analyze the hard potential case $\gamma+2\ge0$. Noticing $\|\cdot\|_{L^2_v}\lesssim\|\cdot\|_{\sigma,0}$ for hard potential, we have 
	\begin{align*}
	\partial_t\E^h_{K,l_0}+\lambda \E^h_{K,l_0}(t)\lesssim \|\partial_t\phi\|_{L^\infty}\E^h_{K,l_0}(t)+\|\nabla_x(a_\pm,b,c)\|^2.
	\end{align*}
	By Gronwall's inequality, 
	\begin{align}\label{82a}
	\E^h_{K,l_0}(t)\lesssim e^{-\lambda t}\E^h_{K,l_0}(0)+\int^t_0\,d\tau e^{-\lambda (t-\tau)}\big(\|\partial_t\phi(\tau)\|_{L^\infty}\E^h_{K,l_0}(\tau)+\|\nabla_x(a_\pm,b,c)(\tau)\|^2\big).
	\end{align}
	We will need to deal with the terms inside the time integral. By \eqref{81}, for $0\le \tau\le t$, 
	\begin{align*}
	\|\partial_t\phi(\tau)\|_{L^\infty}\E^h_{K,l_0}(\tau)\lesssim (\E^h_{K,l_0}(t))^{3/2}\lesssim (1+t)^{-15/4}X^{3/2}(t).
	\end{align*}
	We claim that for $0\le t\le T$, 
	\begin{align*}
	\|\nabla_x(a_\pm,b,c)(t)\|^2\lesssim (\epsilon_0+X(t))(1+t)^{-5/4}
	\end{align*}
	Recalling \eqref{42b}, by Duhamel's principle, we can write the solution to \eqref{7} as 
	\begin{align*}
	f(t) = e^{tB}f_0+\int^t_0e^{(t-\tau)B}g(\tau)\,d\tau,
	\end{align*}where $g=(g_+,g_-)$ is defined by \eqref{22a}. Applying Theorem \ref{homogen} with $m=1$, $l=0$, $\sigma_1= \frac{5}{4}$ therein, we have 
	\begin{align*}
	\|\nabla_x\P f(t)\|_{L^2_{v,x}} \lesssim\|\nabla_xf(t)\|_{L^2_{v,x}} \lesssim \epsilon_0(1+t)^{-5/4}+\int^t_0(1+t-\tau)^{-5/4}\big(\|g(\tau)\|_{Z_1}+\|\nabla_xg(\tau)\|_{L^2_{v,x}}\big)\,d\tau,
	\end{align*}where we use the fact that $(g_\pm,\mu^{1/2})=0$.
	By using Corollary \ref{Coro1}, 
	and $\E_{K,l_0}(t)\le (1+t)^{-3/2}X(t)$, we have 
	\begin{align}\label{84a}
	\|\nabla_xf(t)\|_{L^2_{v,x}} &\lesssim\notag \epsilon_0(1+t)^{-5/4}+X(t)\int^t_0(1+t-\tau)^{-5/4}(1+\tau)^{-3/2}\,d\tau\\
	&\lesssim (\epsilon_0+X(t))(1+t)^{-5/4}.
	\end{align}This proves the claim. 
	Now \eqref{82a} gives that for $0\le t\le T$, 
	\begin{align}\label{85a}
	\E^h_{K,l_0}(t)\notag&\lesssim e^{-\lambda t}\epsilon_0+\int^t_0\,d\tau e^{-\lambda (t-\tau)}\big((1+\tau)^{-15/4}X^{3/2}(t)+(\epsilon_0^2+X^2(t))(1+\tau)^{-5/2}\big)\\
	&\lesssim (\epsilon_0^2+X^{3/2}(t)+X^2(t))(1+t)^{-5/2}. 
	\end{align}
By using the same way proving \eqref{84a} and applying $m=0$, $\sigma_0=\frac{3}{4}$ instead, we can obtain 
	\begin{align}\label{86a}
	\|f\|_{L^2_{v,x}}\lesssim (\epsilon_0+X(t))(1+t)^{-3/4}.
	\end{align}
	Since $\E_{K,l_0}(t)\approx\|\P f\|^2_{L^2_x}+\E^h_{K,l_0}(t)$, we have from \eqref{85a} and \eqref{86a} that 
	\begin{align}\label{87a}
	\E_{K,l_0}(t)\approx (\epsilon_0^2+X^{3/2}(t)+X^2(t))(1+t)^{-3/2}. 
	\end{align}
	Now the desired estimate \eqref{80} follows from \eqref{81a}, \eqref{85a} and \eqref{87a}.

	{\bf Step 3.} In this step, we analyze the soft potential case $-1\le\gamma+2<0$.  
To prove the decay of $\E^h_{K,l_0}$, 
We will use the trick in \cite{Strain2012} to split the velocity space for any time $t>0$:
\begin{align*}
	\mathbf{E}(t)= \{\<v\>^{-\gamma-2}\le t^{1-p}\},\quad \mathbf{E}^c(t)= \{\<v\>^{-\gamma-2}> t^{1-p}\}.
\end{align*}
Then we define $\E^{h,low}_{K,l_0}$ and $\E^{h,high}_{K,l_0}$ to be the restriction of $\E^h_{K,l_0}$ on $\mathbf{E}(t)$ and $\mathbf{E}^c(t)$ respectively. For the term involving only $x$ like $\|\partial^\alpha E\|_{L^2_x}^2$ in the definition of $\E^{h}_{K,l}$, we can add a term of $v$ as $C\|\mu\<v\>^{-\gamma-2}\|^2_{L^2_v}$ to generate the time-velocity interactive control. 
Notice that the term $\|\partial^K E\|_{L^2_x}^2$ inside $\E^{h}_{K,l}$ is bounded above by 
\begin{align*}
	\|\nabla^{K+1}_x\phi\|^2_{L^2_x}=\|\nabla_x^{K+1}\Delta^{-1}_x(a_+-a_-)\|^2_{L^2_x}
	\le \|\nabla_x^{K-1}(a_\pm,b,c)\|^2_{L^2_x},
\end{align*}
which is a term in $\D_{K,l_0}$. 
Then we have 
\begin{align*}
	t^{p-1}\E^h_{K,l_0} &= t^{p-1}\E^{h,low}_{K,l_0}+t^{p-1}\E^{h,high}_{K,l_0}\\
	&\le \D_{K,l_0} + t^{p-1}\E^{h,high}_{K,l_0}.
\end{align*}
Thus by \eqref{99a}, 
\begin{align*}
	\partial_t\E^h_{K,l_0}+\lambda t^{p-1}\E^h_{K,l_0}\lesssim \|\partial_t\phi\|_{L^\infty}\E^h_{K,l_0}(t)+\|\nabla_x\P f\|^2_{L^2_{v,x}} + t^{p-1}\E^{h,high}_{K,l_0}.
\end{align*}
Now we estimate the right-hand terms one by one. 
Firstly, by \eqref{81},
\begin{align*}
	\|\partial_t\phi\|_{L^\infty}\E^h_{K,l_0}(t)\lesssim (\E^h_{K,l_0}(t))^{3/2}\lesssim (1+t)^{-(\frac{3}{2}+p)\frac{3}{2}}X^{3/2}(t).
\end{align*}
Secondly, we claim that  
\begin{align*}
	\E^{h,high}_{K,l_0}\lesssim \epsilon_0^2(1+t)^{-\frac{5}{2}}.
\end{align*}
Indeed, one one hand, 
\begin{align*}
	\E^{h,high}_{K,l_0}\lesssim \E^{h}_{K,l}\lesssim \E_{K,l_0+l_1}\lesssim \epsilon_0^2.
\end{align*}
On the other hand, since $w^{-l_1}=\<v\>^{(\gamma+2)\frac{2\gamma l_1}{\gamma+2}}<t^{-\frac{5}{4}}$ on $\mathbf{E}^c(t)$, by \eqref{81a}, 
\begin{align*}
	\E^{h,high}_{K,l_0}\lesssim \epsilon^2_0t^{-\frac{5}{2}}.
\end{align*}This completes the claim. 
Thirdly, we claim that for $0\le t\le T$, 
\begin{align*}
	\|\nabla_x \P f\|\lesssim (\epsilon_0+X(t))(1+t)^{-5/4}.
\end{align*}
The same as in Step 2, applying Theorem \ref{homogen} with $m=1$, $l=0$, $\sigma_1= \frac{5}{4}$ therein, we have that for $l_*>\frac{5(\gamma+2)}{4\gamma}$, 
\begin{align}\label{83}
	\|\nabla_x\P f(t)\|_{L^2_{v,x}}  \lesssim \epsilon_0(1+t)^{-5/4}+\int^t_0(1+t-\tau)^{-5/4}\big(\|w^{l_*}g(\tau)\|_{Z_1}+\|w^{l_*}\nabla_xg(\tau)\|_{L^2_{v,x}}\big)\,d\tau,
\end{align}where we use the fact that $(g_\pm,\mu^{1/2})=0$.
By using Corollary \ref{Coro1} 
and $\E_{K,l_0}(t)\le (1+t)^{-3/2}X(t)$, we have 
\begin{align}\label{84}
	\|\nabla_xf(t)\|_{L^2_{v,x}} &\lesssim\notag \epsilon_0(1+t)^{-5/4}+X(t)\int^t_0(1+t-\tau)^{-5/4}(1+\tau)^{-3/2}\,d\tau\\
	&\lesssim (\epsilon_0+X(t))(1+t)^{-5/4}.
\end{align}This proves the claim. 
Combining the above estimate, we have 
\begin{align}\label{73}
	\partial_t\E^h_{K,l_0}+\lambda t^{p-1}\E^h_{K,l_0}\lesssim
	(1+t)^{-(\frac{3}{2}+p)\frac{3}{2}}X^{3/2}(t)+\epsilon_0^2t^{p-1}(1+t)^{-\frac{5}{2}}+(\epsilon^2_0+X^2(t))(1+t)^{-5/2}.
\end{align}
Notice that one has the following inequalities.
\begin{align*}
	\int^t_0e^{\lambda(t^p-\tau^p)}(1+\tau)^{-(\frac{3}{2}+p)\frac{3}{2}}\,d\tau\lesssim (1+t)^{-(\frac{5}{4}+\frac{5p}{2})},\\
	\int^t_0e^{\lambda(t^p-\tau^p)}\tau^{p-1}(1+\tau)^{-\frac{5}{2}}\,d\tau\lesssim (1+t)^{-\frac{5}{2}},\\
	\int^t_0e^{\lambda(t^p-\tau^p)}(1+\tau)^{-\frac{5}{2}}\,d\tau\lesssim (1+t)^{-(\frac{3}{2}+p)}.
\end{align*}
Applying Gronwall's inequality to \eqref{73}, we have  
\begin{align}\label{85}
	\E^h_{K,l_0}(t)&\lesssim e^{-\lambda t^p}\E^h_{K,l_0}(0)+\int^t_0\,d\tau e^{-\lambda (t^p-\tau^p)}\big((1+\tau)^{-(\frac{3}{2}+p)\frac{3}{2}}X^{3/2}(\tau)+\epsilon_0^2t^{p-1}(1+\tau)^{-\frac{5}{2}}\notag\\&\qquad\qquad\qquad\qquad\qquad\qquad\qquad\qquad\qquad+(\epsilon_0+X(\tau))(1+\tau)^{-5/4}\big)\notag\\
	&\lesssim (\epsilon^2_0+X^{3/2}(t)+X^2(t))(1+t)^{-\frac{3}{2}-p},
\end{align}
for $0\le t\le T$. 
By using the same way proving \eqref{84} and applying $m=0$, $\sigma_0=\frac{3}{4}$ instead, we can obtain 
\begin{align}\label{86}
	\|f\|_{L^2_{v,x}}\lesssim (\epsilon_0+X(t))(1+t)^{-3/4}.
\end{align}
Since $\E_{K,l_0}(t)\approx\|\P f\|^2_{L^2_x}+\E^h_{K,l_0}(t)$, we have from \eqref{85} and \eqref{86} that 
\begin{align}\label{87}
	\E_{K,l_0}(t)\approx (\epsilon_0^2+X^{3/2}(t)+X^2(t))(1+t)^{-3/2}. 
\end{align}
Now the desired estimate \eqref{80} follows from \eqref{81a}, \eqref{85} and \eqref{87}. This completes the proof. 

\qe\end{proof}

\begin{proof}[Proof of Theorem \ref{main1}]
	It follows immediate from the $a$ $priori$ estimate \eqref{80} that $X(t)\lesssim \epsilon^2_0$ holds true for any $t\ge 0$, whenever $\epsilon_0$ is sufficiently small. The rest is to prove the local existence and uniqueness of solutions in terms of the energy norm $\E_{K,l_0}$ and the non-negativity of $\F_\pm=\mu+\mu^{1/2}f$. One can use the iteration on system 
	\begin{equation}
	\left\{\begin{aligned}
	&\partial_tf^{n+1}_\pm+v\cdot\nabla_xf^{n+1}_\pm\mp\nabla_x\phi^n\cdot\nabla_xf^{n+1}_\pm\pm\frac{1}{2}\nabla_x\phi^n\cdot vf^{n+1}_\pm\pm v\mu^{1/2}\cdot\nabla_x\phi^n-L_\pm f=\Gamma_\pm(f^n,f^{n+1}),\\
	&-\Delta_x\phi^{n+1}=\int_{\R^3}\mu^{1/2}(f^{n+1}_+-f^{n+1}_-)\,dv,\\
	&f^{n+1}|_{t=0}=f_0,
	\end{aligned}\right.
	\end{equation}
	and the details of proof are omitted for brevity; see \cite{Guo2012, Strain2013} and \cite{Gressman2011}.
	Therefore, the unique global-in-time solution to \eqref{7}-\eqref{9} exists by using continuity argument. The estimate \eqref{15a} follows directly. 
\end{proof}

\section{Regularity}\label{sec5}
In this section, we will prove the smoothing effect for solutions to Vlasov-Poisson-Landau system with lower order initial data. 
The Vlasov-Poisson-Landau system is 
\begin{equation}\label{16c}
\left\{\begin{aligned}
	&\partial_tf_\pm + v_i\partial^{e_i}f_\pm \pm \frac{1}{2}\partial^{e_i}\phi v_if_\pm  \mp\partial^{e_i}\phi\partial_{e_i}f_\pm \pm \partial^{e_i}\phi  v_i\mu^{1/2} - L_\pm f = \Gamma_{\pm}(f,f),\\
	&-\Delta_x \phi = \int_{\Rd}(f_+-f_-)\mu^{1/2}\,dv,\\ 
	&f_\pm|_{t=0} = f_{0,\pm}. 
\end{aligned}\right.
\end{equation}The index appearing in both superscript and subscript means the summation. Our goal is to obtain the $a$ $priori$ estimate from these equations. 
In order to extract the smoothing estimate, we let $N=N(\alpha,\beta)>0$ be a large number chosen later. Assume $T\in(0,1]$, $t\in[0,T]$ and 
\begin{equation}\label{93}
	\psi=t^{N},\quad\psi_k=\left\{\begin{aligned}
		1, \text{  if $k\le 0$},\\
		\psi^k, \text{ if $k>0$}. 
	\end{aligned}\right.
\end{equation}
is this section. Then $|\partial_t\psi_k|\lesssim \psi_{k-1/N}$. Let $f$ be the smooth solution to \eqref{7}-\eqref{9} over $0\le t\le T$ and assume the $a$ $priori$ assumption 
\begin{align}\label{priori1}
	\sup_{0\le t\le T}\E_{K,l}(t)\le \delta_0,
\end{align}where $\delta_0\in(0,1)$ is a suitably small constant.
Under this assumption, we can derive a simple fact that 
\begin{align*}
	\|\phi\|_{L^\infty}\lesssim\|\phi\|_{H^2_x}\le \delta_0, \quad \|e^{\pm\phi}\|_{L^\infty}\approx 1.
\end{align*}

\begin{Thm}\label{lem51}Assume $\gamma+2\ge -1$, $l\ge K\ge 3$.
	Let $f$ be the solution to \eqref{7}-\eqref{9} satisfying that
	\begin{align}\label{100}
		\epsilon^2_1 = \E_{3,l}(0)<\infty, 
		\end{align}
	If $\gamma+2>0$, then there exists $t_0\in(0,1)$ such that  
	\begin{align}
		\sup_{0\le t\le t_0}\E_{K,l}(t)\lesssim \epsilon^2_1.
	\end{align}
If $-1\le\gamma+2\le 0$, we assume additionally 
	\begin{align*}
		\sup_{0\le t\le T}\|\<v\>^{C_{K,l}}f(t)\|^2_{L^2_{v,x}}<\infty.
	\end{align*}for some large constant $C_{K,l}>0$ depending on $K,l$. Then there exists $t_0\in(0,1)$ such that  
\begin{align}
\sup_{0\le t\le t_0}\E_{K,l}(t)\lesssim \epsilon^2_1.
\end{align}
	
\end{Thm}
	The reason of choosing $\psi_{|\alpha|+|\beta|-3}$ in $\eqref{Defe}$ is that whenever $l\ge K\ge 3$, the initial value $\E_{K,l}(0)=\E_{3,l}(0)$, since $\psi_{|\alpha|+|\beta|-3}|_{t=0}=0$ whenever $|\alpha|+|\beta|\ge 4$. 

\begin{Lem}\label{thm41}For any $l\ge 0$, there is $\E_{K,l}$ satisfying \eqref{Defe} such that for $0\le t\le T$,
\begin{align}\label{72}
	\partial_t\E_{K,l}(t)+\lambda\D_{K,l}(t) \lesssim \|\partial_t\phi\|_{L^\infty_x}\E_{K,l}(t)+\E_{K,l} + \sum_{|\alpha|+|\beta|\le K}\|\psi_{|\alpha|+|\beta|-3-\frac{1}{2N}}w^{l-|\alpha|-|\beta|}\partial^\alpha_\beta f\|^2_{L^2_{v,x}}.
\end{align}where $D_{K,l}$ is defined by \eqref{Defd}. 
\end{Lem}
\begin{proof}
	For any $K\ge 3$ being the total derivative of $v,x$, we let $|\alpha|+|\beta|\le K$.
	On one hand, we apply $\partial^\alpha$ to equation \eqref{16c} to get 
	\begin{equation}\begin{aligned}\label{35}
		&\quad\,\partial_t\partial^{\alpha} f_\pm + v_i\partial^{e_i+\alpha} f_\pm \pm \frac{1}{2}\sum_{\substack{\alpha_1\le\alpha}}C^{\alpha_1}_{\alpha}\partial^{e_i+\alpha_1}\phi v_i\partial^{\alpha-\alpha_1}f_\pm \\ &\qquad\mp\sum_{\substack{\alpha_1\le\alpha}}C^{\alpha_1}_{\alpha}\partial^{e_i+\alpha_1}\phi\partial^{\alpha-\alpha_1}_{e_i} f_\pm \pm \partial^{e_i+\alpha}\phi v_i\mu^{1/2} - \partial^{\alpha} L_\pm  f =
		\partial^{\alpha} \Gamma_{\pm}(f,f).\end{aligned}
\end{equation}
On the other hand, we apply $\partial^{\alpha}_\beta$ to equation \eqref{16c}. Then,  
	\begin{align}\label{36}
		&\notag\quad\,\partial_t\partial^{\alpha}_\beta f_\pm + \sum_{\beta_1\le \beta}C^{\beta_1}_{\beta}\partial_{\beta_1}v_i\partial^{e_i+\alpha}_{\beta-\beta_1}f_\pm \pm \frac{1}{2}\sum_{\substack{\alpha_1\le\alpha}}\sum_{\beta_1\le\beta}C^{\alpha_1,\beta_1}_{\alpha,\beta}\partial^{e_i+\alpha_1}\phi \partial_{\beta_1}v_i\partial^{\alpha-\alpha_1}_{\beta-\beta_1}f_\pm \\ &\qquad\mp\sum_{\substack{\alpha_1\le\alpha}}C^{\alpha_1}_{\alpha}\partial^{e_i+\alpha_1}\phi\partial^{\alpha-\alpha_1}_{\beta+e_i}f_\pm \pm \partial^{e_i+\alpha}\phi \partial_\beta(v_i\mu^{1/2}) - \partial^{\alpha}_\beta L_\pm f =
		\partial^{\alpha}_\beta \Gamma_{\pm}(f,f).
	\end{align}	

{\bf Step 1. Estimate without weight.}
For the estimate without weight, we take the case $|\alpha|\le K$ and $\beta=0$. This case is for obtaining the term $\|\partial^\alpha\nabla_x\phi\|^2_{L^2_x}$ on the left hand side of the energy inequality. Taking inner product of equation \eqref{35} with $\psi_{2|\alpha|-6}e^{\pm\phi}\partial^{\alpha} f_\pm$ over $\R^3_v\times\R^3_x$, we have   
\begin{align}\label{45}
	&\notag\quad\,\Big(\partial_t\partial^{\alpha} f_\pm,\psi_{2|\alpha|-6}e^{\pm\phi}\partial^{\alpha} f_\pm\Big)_{L^2_{v,x}}
	+ \Big(v_i\partial^{e_i+\alpha}f_\pm,\psi_{2|\alpha|-6}e^{\pm\phi}\partial^{\alpha} f_\pm\Big)_{L^2_{v,x}}\\ 
	&\notag\pm \Big(\frac{1}{2}\sum_{\substack{\alpha_1\le\alpha}}C^{\alpha_1}_{\alpha}\partial^{e_i+\alpha_1}\phi v_i\partial^{\alpha-\alpha_1}f_\pm,\psi_{2|\alpha|-6}e^{\pm\phi}\partial^{\alpha} f_\pm\Big)_{L^2_{v,x}} \\ 
	&\mp
	\Big(\sum_{\substack{\alpha_1\le\alpha}}C^{\alpha_1}_{\alpha}\partial^{e_i+\alpha_1}\phi\partial^{\alpha-\alpha_1}_{e_i}f_\pm,\psi_{2|\alpha|-6}e^{\pm\phi}\partial^{\alpha} f_\pm\Big) _{L^2_{v,x}}\\
	&\notag\pm \Big(\partial^{e_i+\alpha}\phi v_i\mu^{1/2},\psi_{2|\alpha|-6}e^{\pm\phi}\partial^{\alpha} f_\pm\Big)_{L^2_{v,x}} 
	- \Big(\partial^{\alpha} L_\pm f,\psi_{2|\alpha|-6}e^{\pm\phi}\partial^{\alpha} f_\pm\Big)_{L^2_{v,x}}\\ 
	&\notag= \Big(\partial^{\alpha} \Gamma_{\pm}(f,f),\psi_{2|\alpha|-6}e^{\pm\phi}\partial^{\alpha} f_\pm\Big)_{L^2_{v,x}}.
\end{align}
Now we take the summation on $\pm$ and real part, and denote these resulting terms by $I_1$ to $I_7$. The estimate here is similar to Theorem \ref{thm411}. In the following we estimate them term by term. 
For the term $I_1$, 
\begin{align}\label{37}
	I_1 &=\notag \frac{1}{2}\partial_t\sum_{\pm}\|e^{\frac{\pm\phi}{2}}\psi_{|\alpha|-3}\partial^{\alpha} f_\pm\|^2_{L^2_{v,x}} \mp \Re\sum_{\pm}\frac{1}{2}(\partial_t\phi e^{\pm\phi}\partial^{\alpha} f_\pm, \psi_{2|\alpha|-6}\partial^{\alpha} f_\pm)_{L^2_{v,x}}\\&\qquad-\Re\sum_{\pm}(\partial_t(\psi_{|\alpha|-3})\partial^{\alpha} f_\pm,\psi_{|\alpha|-3} e^{\pm\phi}\partial^{\alpha} f_\pm)_{L^2_{v,x}}. 
\end{align}
The second term on the right hand side of \eqref{37} is estimated as 
\begin{align}\label{46}
	\Big|\frac{1}{2}(\partial_t\phi \psi_{2|\alpha|-6}e^{\pm\phi}\partial^{\alpha} f_\pm, \partial^{\alpha} f_\pm)_{L^2_{v,x}}\Big|\lesssim \|\partial_t\phi\|_{L^\infty}\|\psi_{|\alpha|-3}\partial^\alpha f_\pm\|^2_{L^2_{v,x}}\lesssim \|\partial_t\phi\|_{L^\infty}\E_{K,l}(t)
\end{align}
The third right-hand term of \eqref{37} is estimated as 
	\begin{align}
	|(\partial_t(\psi_{|\alpha|-3})\partial^{\alpha} f_\pm,\psi_{|\alpha|-3} e^{\pm\phi}\partial^{\alpha} f_\pm)_{L^2_{v,x}}|&\lesssim \|\psi_{|\alpha|-3-\frac{1}{2N}}\partial^\alpha f\|^2_{L^2_{v,x}}.\notag
\end{align}

For the second term $I_2$, we will combine it with $I_3$ and $\alpha_1=0$. Taking integration by parts on $x$, one has  
\begin{align}\label{48aa}
	&\quad\,\Big(v_i\partial^{e_i+\alpha}f_\pm,\psi_{2|\alpha|-6}e^{\pm\phi}\partial^{\alpha} f_\pm\Big)_{L^2_{v,x}}
	\pm \Big(\frac{1}{2}\partial^{e_i}\phi v_i\partial^{\alpha}f_\pm,\psi_{2|\alpha|-6}e^{\pm\phi}\partial^{\alpha} f_\pm\Big)_{L^2_{v,x}}=0.
\end{align}

For the left terms in $I_3$, the weight will be used. In this case, $\alpha_1$ is not zero and by Lemma \ref{Lem26}, it's bounded above by $\E^{1/2}_{K,l}\D_{K,l}$. 
Using Lemma \ref{Lem27}, the term $I_4$ is also bounded above by $\E^{1/2}_{K,l}\D_{K,l}$.

For the term $I_5$, we will divide $e^{\pm\phi}$ into $(e^{\pm\phi}-1)$ and $1$. Recall equation \eqref{16} and \eqref{21}. For the part of $1$, 
\begin{align}\label{43}\notag
	\sum_{\pm}\pm\Re\big(\partial^{e_i+\alpha}\phi v_i\mu^{1/2},\psi_{2|\alpha|-6}\partial^{\alpha} f_\pm\big)_{L^2_{v,x}} 
    &= \frac{1}{2}\partial_t\|\psi_{|\alpha|-3}\partial^{\alpha}\nabla_x\phi\|_{L^2_x}^2.
\end{align}
For the part of $(e^{\pm\phi}-1)$, noticing $|e^{\pm\phi}-1|\lesssim \|\phi\|_{L^\infty}\lesssim \|\nabla_x\phi\|_{H^1_x}$, we have 
\begin{align}
	&\quad\,\Big|\sum_{\pm}\pm\Re\Big(\partial^{e_i+\alpha}\phi v_i\mu^{1/2},(e^{\pm\phi}-1)\psi_{2|\alpha|-6}\partial^{\alpha} f_\pm\Big)_{L^2_{v,x}}\Big|\notag\\
	&\lesssim \|\nabla_x\phi\|_{H^1_x}\sum_{|\alpha|\le K}\|\partial^\alpha\nabla_x\phi\|_{L^2_{v,x}}\sum_{|\alpha|\le K}\|\partial^\alpha(\I-\P)f\|_{L^2_{v,x}}\\
	&\lesssim \E^{1/2}_{K,l}(t)\D_{K,l}(t).\notag
\end{align}

For the term $I_6$, since $L_\pm$ commutes with $\partial^{\alpha}$ and $e^{\pm\phi}$, by Lemma \ref{lemmaL}, we have 
\begin{align}
	I_6 = - \sum_{\pm}\Big(\partial^{\alpha} L_\pm f,\psi_{2|\alpha|-6}e^{\pm\phi}\partial^{\alpha} f_\pm\Big)_{L^2_{v,x}}\ge \lambda \sum_{\pm}\|\psi_{|\alpha|-3}\partial^{\alpha}(\II-\PP) f\|_{\sigma,0}^2. 
\end{align}

For the term $I_7$, by Lemma \ref{lemmag}, we have 
\begin{align}
	|I_7|&= \Big|\sum_{\pm}\Big(\partial^{\alpha} \Gamma_{\pm}(f,f),\psi_{2|\alpha|-6}e^{\pm\phi}\partial^{\alpha} f_\pm\Big)_{L^2_{v,x}}\Big|\lesssim\E^{1/2}_{K,l}(t)\D_{K,l}(t).
\end{align}

Therefore, combining all the estimate above and take the summation on $|\alpha|\le K$, we conclude that, 
\begin{equation}\label{47}
	\begin{aligned}
		&\quad\,\frac{1}{2}\partial_t\sum_{\pm}\sum_{|\alpha|\le K}\Big(\|\psi_{|\alpha|-3}e^{\frac{\pm\phi}{2}}\partial^{\alpha} f_\pm\|_{L^2_{v,x}} +
		\|\psi_{|\alpha|-3}\partial^{\alpha}\nabla_x\phi\|_{L^2_x}^2\Big)\\&\qquad + \lambda \sum_{\pm}\sum_{|\alpha|\le K}\|\psi_{|\alpha|-3}\partial^{\alpha} (\II-\PP)f\|_{\sigma,0}^2\\
 &\lesssim \|\partial_t\phi\|_{L^\infty}\E_{K,l}(t)+\E^{1/2}_{K,l}(t)\D_{K,l}(t)+\sum_{|\alpha|\le K}\|\psi_{|\alpha|-3-\frac{1}{2N}}w^{l-|\alpha|}\partial^\alpha f\|^2_{L^2_{v,x}}.
	\end{aligned}
\end{equation}

{\bf Step 2. Estimate with weight on the mixed derivatives.}
Let $K\ge 3$, $|\alpha|+|\beta|\le K$. Taking inner product of equation \eqref{36} with  $\psi_{2|\alpha|+2|\beta|-6}e^{\pm\phi}w^{2l-2|\alpha|-2|\beta|}\partial^{\alpha}_\beta f_\pm$ over $\R^3_v\times\R^3_x$, one has
\begin{align*}
	&\quad\,\Big(\partial_t\partial^{\alpha}_\beta  f,e^{\pm\phi}\psi_{2|\alpha|+2|\beta|-6}w^{2l-2|\alpha|-2|\beta|}\partial^{\alpha}_\beta  f\Big)_{L^2_{v,x}}\\
	&\qquad + \Big(\sum_{\beta_1\le \beta}C^{\beta_1}_{\beta}\partial_{\beta_1}v_i\partial^{e_i+\alpha}_{\beta-\beta_1} f,e^{\pm\phi}\psi_{2|\alpha|+2|\beta|-6}w^{2l-2|\alpha|-2|\beta|}\partial^{\alpha}_\beta  f\Big)_{L^2_{v,x}} \\
	&\qquad\pm \Big(\frac{1}{2}\sum_{\substack{\alpha_1\le\alpha\\\beta_1\le\beta}}C^{\alpha_1,\beta_1}_{\alpha,\beta}\partial^{e_i+\alpha_1}\phi \partial_{\beta_1}v_i\partial^{\alpha-\alpha_1}_{\beta-\beta_1} f,e^{\pm\phi}\psi_{2|\alpha|+2|\beta|-6}w^{2l-2|\alpha|-2|\beta|}\partial^{\alpha}_\beta  f\Big)_{L^2_{v,x}} \\ &\qquad\mp\Big(\sum_{\substack{\alpha_1\le\alpha}}C^{\alpha_1}_{\alpha}\partial^{e_i+\alpha_1}\phi\partial^{\alpha-\alpha_1}_{\beta+e_i} f,e^{\pm\phi}\psi_{2|\alpha|+2|\beta|-6}w^{2l-2|\alpha|-2|\beta|}\partial^{\alpha}_\beta  f\Big)_{L^2_{v,x}}\\
	&\qquad \pm \Big(\partial^{e_i+\alpha}\phi \partial_\beta(v_i\mu^{1/2}),e^{\pm\phi}\psi_{2|\alpha|+2|\beta|-6}w^{2l-2|\alpha|-2|\beta|}\partial^{\alpha}_\beta  f\Big)_{L^2_{v,x}}\\
	&\qquad - \Big(\partial^{\alpha}_\beta L_\pm f,e^{\pm\phi}\psi_{2|\alpha|+2|\beta|-6}w^{2l-2|\alpha|-2|\beta|}\partial^{\alpha}_\beta  f\Big)_{L^2_{v,x}} \\
	&= \Big(\partial^{\alpha}_\beta \Gamma_{\pm}(f,f),e^{\pm\phi}\psi_{2|\alpha|+2|\beta|-6}w^{2l-2|\alpha|-2|\beta|}\partial^{\alpha}_\beta  f\Big)_{L^2_{v,x}}.
\end{align*}
Now we denote these terms with summation $\sum_{\pm}$ by $J_1$ to $J_{7}$ and estimate them term by term. 
The estimate of $J_1$ to $J_4$ are similar to $I_1$ to $I_4$. 
For $J_1$, 
\begin{align*}
	J_1 &\ge \partial_t\sum_{\pm}\|e^{\frac{\pm\phi}{2}}\psi_{|\alpha|+|\beta|-3}w^{l-|\alpha|-|\beta|}\partial^{\alpha}_\beta f_\pm\|_{L^2_{v,x}} - C\|\partial_t\phi\|_{L^\infty}\E_{K,l}(t)\\&\qquad -\sum_{\pm}\|\psi_{|\alpha|+|\beta|-3-\frac{1}{2N}}w^{l-|\alpha|-|\beta|}\partial^\alpha_\beta f_\pm\|^2_{L^2_{v,x}} ,
\end{align*}
Similar to \eqref{48aa}, $J_2$ and $J_3$ with $\alpha_1=0$ are canceled by using integration by parts. Using Lemma \ref{Lem26} and Lemma \ref{Lem27}, the left case $\alpha_1\neq 0$ in $J_3$ and $J_4$ are bounded above by $\E^{1/2}_{K,l}(t)\D_{K,l}(t)$.
For the term $J_5$, we only need a upper bound: for any $\eta>0$, 
\begin{align*}\notag
	|J_5| &= 
	\Big|\sum_{\pm}\pm \Big(\partial^{e_i+\alpha}\phi \partial_\beta(v_i\mu^{1/2}),\psi_{2|\alpha|+2|\beta|-6}e^{\pm\phi}w^{2l-2|\alpha|-2|\beta|}\partial^{\alpha}_\beta f_\pm\Big)_{L^2_{v,x}}\Big|\\
	&\lesssim \eta\sum_\pm\|\psi_{|\alpha|+|\beta|-3}\partial^{\alpha}_{\beta}f_\pm\|^2_{\sigma,l-|\alpha|-|\beta|}+C_\eta\|\psi_{|\alpha|-3}\partial^\alpha\nabla_x\phi\|_{L^2_{v,x}}^2.
\end{align*}
Notice that $\|\psi_{|\alpha|-3}\partial^\alpha\nabla_x\phi\|_{L^2_{v,x}}^2$ is bounded above by $\E_{K,l}$. For the term $J_6$, since $L_\pm$ commutes with $e^{\pm\phi}$, by Lemma \ref{lemmaL}, we have 
\begin{align*}
	J_6 &= - \sum_{\pm}\Big(\partial^{\alpha}_\beta L_\pm f,\psi_{2|\alpha|+2|\beta|-6}e^{\pm\phi}w^{2l-2|\alpha|-2|\beta|}\partial^{\alpha}_\beta f_\pm\Big)_{L^2_{v,x}}\\&\ge \lambda\sum_{\pm} \|\psi_{|\alpha|+|\beta|-3}e^{\frac{\pm\phi}{2}}\partial^{\alpha}_\beta f_\pm\|_{\sigma,l-|\alpha|-|\beta|}^2-C_\eta\sum_{\pm}\|\partial^\alpha f_\pm\|^2_{\sigma,0}\\
	&\qquad -\eta\sum_{\pm}\sum_{|\beta_1|\le|\beta|}\|\psi_{|\alpha|+|\beta|-3}e^{\frac{\pm\phi}{2}}\partial^{\alpha}_{\beta_1} f_\pm\|^2_{\sigma,l-|\alpha|-|\beta_1|},
\end{align*}for any $\eta>0$. 
Here we use the fact that $\|w^{l-|\alpha|-|\beta|}(\cdot)\|_{L^2(B_{C_\eta})}\lesssim \|\cdot\|_{\sigma,0}$. 
The term $J_{7}$, by Lemma \ref{lemmag}, is bounded above by  $\E^{1/2}_{K,l}\D_{K,l}+\E_{K,l}\D^{1/2}_{K,l}\lesssim (\E^{1/2}_{K,l}+\E_{K,l})\D_{K,l}+\E_{K,l}$.

Combining all the above estimate, taking summation on $|\alpha|+|\beta|\le K$ and letting $\eta$ sufficiently small, we have 
\begin{align}\label{111ab}
	&\quad\,\frac{1}{2}\partial_t\sum_\pm\sum_{|\alpha|+|\beta|\le K}\|e^{\frac{\pm\phi}{2}}\psi_{|\alpha|+|\beta|-3}w^{l-|\alpha|-|\beta|}\partial^\alpha_\beta f_\pm\|^2_{L^2_{v,x}}+\lambda \sum_\pm\sum_{|\alpha|+|\beta|\le K}\|\psi_{|\alpha|+|\beta|-3}\partial^\alpha_\beta f_\pm\|^2_{\sigma,l-|\alpha|-|\beta|} \notag\\&\lesssim \|\partial_t\phi\|_{L^\infty_x}\E_{K,l}(t) + \sum_{|\alpha|+|\beta|\le K}\|\psi_{|\alpha|+|\beta|-3-\frac{1}{2N}}w^{l-|\alpha|-|\beta|}\partial^\alpha_\beta f\|^2_{L^2_{v,x}}+(\E^{1/2}_{K,l}+\E_{K,l})\D_{K,l}+\E_{K,l}.
\end{align}
Together with \eqref{priori1}, choosing $\delta_0>0$ sufficiently small and taking combination $\eqref{47}+\eqref{111ab}$, we have 
\begin{align}\label{65}
	\partial_t\E_{K,l}(t)+\lambda\D_{K,l}(t) \lesssim \|\partial_t\phi\|_{L^\infty_x}\E_{K,l}(t)+\E_{K,l} + \sum_{|\alpha|+|\beta|\le K}\|\psi_{|\alpha|+|\beta|-3-\frac{1}{2N}}w^{l-|\alpha|-|\beta|}\partial^\alpha_\beta f\|^2_{L^2_{v,x}},
\end{align}
where we let 
\begin{align*}
\E_{K,l}(t) = \sum_\pm\sum_{|\alpha|+|\beta|\le K}\|e^{\frac{\pm\phi}{2}}\psi_{|\alpha|+|\beta|-3}w^{l-|\alpha|-|\beta|}\partial^\alpha_\beta f_\pm\|^2_{L^2_{v,x}} + \sum_{|\alpha|\le K}
\|\psi_{|\alpha|-3}\partial^{\alpha}\nabla_x\phi\|_{L^2_x}^2.
\end{align*}
It's straightforward to show that $\E_{K,l}$ satisfies \eqref{Defe}.  
Notice that there's $\|\psi_{|\alpha|-3}\partial^\alpha E(t)\|^2_{L^2_x}$  in $\E_{K,l}$ on the right hand side of \eqref{65}, and hence we can put $\|\psi_{|\alpha|-3}\partial^\alpha E(t)\|^2_{L^2_x}$, which is in $\D_{K,l}$, on the left hand side of \eqref{65}. 
\qe\end{proof}

 Therefore, now it suffices to control the last term in \eqref{72}. 
\begin{Lem}\label{Lem53}	Let $f$ to be the solution to \eqref{7}-\eqref{9} and assume the same assumption as in Lemma \ref{thm41}. Let $0<\delta<1$ and multi-indices $|\alpha|+|\beta|\le K$. Then for $-1\le\gamma+2\le 0$, we have  
		\begin{align}\label{112}
		&\notag\quad\,\|\psi_{|\alpha|+|\beta|-3-\frac{1}{2N}}w^{l-|\alpha|-|\beta|}\partial^\alpha_\beta f\|^2_{L^2_{v,x}}\\
		&\lesssim \delta^2\partial_t\big(-\psi_{2|\alpha|+2|\beta|-6}w^{l-|\alpha|-|\beta|}\partial^\alpha_\beta f_\pm e^{\frac{\pm\phi}{2}},(\theta^ww^{l-|\alpha|-|\beta|}{(\partial^\alpha_\beta f_\pm e^{\frac{\pm\phi}{2}})^\wedge})^\vee\big)_{L^2_{v,x}}\\
		&\notag\qquad+\delta^2\big(\D_{K,l}+(\E^{1/2}_{K,l}+\E_{K,l})\D_{K,l}+\|\partial_t\phi\|_{L^\infty_x}\E_{K,l}(t)+\E_{K,l}\big)+C_\delta\|\<v\>^{C_{K,l}}f\|_{L^2_{v,x}}^2,
	\end{align}while for $\gamma+2>0$, we have 
\begin{align}
	\label{137}&\notag\quad\,\|\psi_{|\alpha|+|\beta|-3-\frac{1}{2N}}w^{l-|\alpha|-|\beta|}\partial^\alpha_\beta f\|^2_{L^2_{v,x}}\\
	&\lesssim \delta^2\partial_t\big(-\psi_{2|\alpha|+2|\beta|-6}w^{l-|\alpha|-|\beta|}\partial^\alpha_\beta f_\pm e^{\frac{\pm\phi}{2}},(\theta^ww^{l-|\alpha|-|\beta|}{(\partial^\alpha_\beta f_\pm e^{\frac{\pm\phi}{2}})^\wedge})^\vee\big)_{L^2_{v,x}}\\
	&\notag\qquad+\delta^2\big(\D_{K,l}+(\E^{1/2}_{K,l}+\E_{K,l})\D_{K,l}+\|\partial_t\phi\|_{L^\infty_x}\E_{K,l}(t)\big)+C_\delta\E_{K,l},
\end{align}
	where $\theta^w=\theta^w(v,D_v)$ and $\theta\in S(1)$ is defined by \eqref{107}. 
\end{Lem}

\begin{proof}
{\bf Step 1.}
Choose constants 
\begin{align}\label{106b}\notag
	\delta_1=\delta_1(\alpha,\beta)&\in\Big(0,\frac{1}{2}\Big],\\ 
	\delta_2= 1-\delta_1&\in\Big[\frac{1}{2},1\Big),\\ l_0=\gamma\delta_2<0\notag
\end{align} to be determined later. Let $\chi_0$ to be a smooth cutoff function such that $\chi_0(z)$ equal to $1$ when $|z|<\frac{1}{2}$ and equal to $0$ when $|z|\ge 1$. Define 
\begin{align}\label{98}
	\tilde{b}(v,y) &= \<v\>^{l_0}|y|^{\delta_1},\\
	\chi(v,\eta) &= \chi_0\bigg(\frac{\<\eta\>\<v\>^{l_0}}{|y|^{\delta_2}}\bigg),\notag
\end{align}and
\begin{align}\label{107}
	\theta(v,\eta) = \<v\>^{l_0}|y|^{-1-\delta_2}y\cdot\eta\,\chi(v,\eta).
\end{align}
Notice that for any multi-indices $\alpha,\beta$, 
\begin{align*}
	\psi_{|\alpha|+|\beta|-3-\frac{1}{2N}} = \psi_{|\alpha|-3-\frac{1}{2N}}\psi_{|\beta|-3-\frac{1}{2N}}.
\end{align*}
 If $|\alpha|>3$, we choose $N=N(\alpha)$ such that 
\begin{align}\label{106a}
	-\frac{2N(|\alpha|-3)-1}{2} = -\frac{|\alpha|}{\delta_1}.
\end{align}
Then by the definition \eqref{98} of $\tilde{b}$ and Young's inequality,
\begin{align}\notag
	\psi_{|\alpha|-3-\frac{1}{2N}}&\lesssim \delta \big((\tilde{b}^{1/2})^{\frac{|\alpha|-3-\frac{1}{2N}}{|\alpha|-3}}\psi_{|\alpha|-3-\frac{1}{2N}}\big)^{\frac{|\alpha|-3}{|\alpha|-3-\frac{1}{2N}}}+C_{0,\delta}\big((\tilde{b}^{-1/2})^{\frac{|\alpha|-3-\frac{1}{2N}}{|\alpha|-3}}\big)^{2N(|\alpha|-3)}\\
	&\lesssim \delta\, \tilde{b}^{1/2}\psi_{|\alpha|-3} + C_{0,\delta}(\<v\>^{-\frac{l_0|\alpha|}{\delta_1}}|y|^{-|\alpha|}),\label{100a}
\end{align}where $C_{0,\delta}$ is a large constant depending on $\delta>0$ and $|\alpha|$. 
If $|\alpha|\le 3$, we choose $\eta\in[0,1)$ such that $-\frac{\eta}{2(1-\eta)}=\frac{-|\alpha|}{\delta_1}$. Then 
\begin{align*}
	\psi_{|\alpha|-3-\frac{1}{2N}}=1&\lesssim (\delta^\eta\,\tilde{b}^{\eta/2})^{\frac{1}{\eta}} + (\delta^{-\eta}\tilde{b}^{-\eta/2})^{\frac{1}{1-\eta}}\\
	&\lesssim \delta\,\tilde{b}^{1/2}+C_{0,\delta}\<v\>^{\frac{-l_0|\alpha|}{\delta_1}}|y|^{-|\alpha|}.
\end{align*} 
Thus, taking the Fourier transform $(\cdot)^\wedge$ with respect to $x$, we have 
\begin{align}\label{101}
	&\quad\,\|\psi_{|\alpha|+|\beta|-3-\frac{1}{2N}}w^{l-|\alpha|-|\beta|}\partial^\alpha_\beta f\|_{L^2_{v,x}}\notag\\\notag&= \|\psi_{|\alpha|+|\beta|-3-\frac{1}{2N}}w^{l-|\alpha|-|\beta|}(\partial^\alpha_\beta f)^\wedge(v,y)\|_{L^2_{v,y}}\\
	&\lesssim \delta\|\psi_{|\alpha|+|\beta|-3}\tilde{b}^{1/2}w^{l-|\alpha|-|\beta|}(\partial^\alpha_\beta f)^\wedge(v,y)\|_{L^2_{v,y}} + C_{0,\delta} \|\psi_{|\beta|-3-\frac{1}{2N}}w^{l-|\alpha|-|\beta|}\<v\>^{\frac{-l_0|\alpha|}{\delta_1}}\partial_\beta f\|_{L^2_{v,x}}.
\end{align}
To deal with the second right-hand term of \eqref{101}, we use a similar interpolation on $\tilde{a}^{1/2}$. In fact, if $|\beta|>3$, we have 
\begin{align}\label{107b}\notag
	\psi_{|\beta|-3-\frac{1}{2N}}\<v\>^{\frac{-l_0|\alpha|}{\delta_1}}&\lesssim \Big(\psi_{|\beta|-3-\frac{1}{2N}}\big(\frac{\delta}{C_{0,\delta}}\tilde{a}^{1/2}\big)^{\frac{|\beta|-3-\frac{1}{2N}}{|\beta|-3}}\Big)^{\frac{|\beta|-3}{|\beta|-3-\frac{1}{2N}}}\\
	&\qquad\notag+((C_{0,\delta}\delta^{-1}\tilde{a}^{-1/2})^{\frac{|\beta|-3-\frac{1}{2N}}{|\beta|-3}}\<v\>^{\frac{-l_0|\alpha|}{\delta_1}})^{2N|\beta|-3}\\
	&\lesssim \frac{\delta}{C_{0,\delta}}\psi_{|\beta|-3}\tilde{a}^{1/2}+C_{\delta}\tilde{a}^{-\frac{1}{2}(2N(|\beta|-3)-1)}\<v\>^{C_K},
\end{align}where $C_{0,\delta}$ comes from \eqref{100a}.
When $|\alpha|>3$, recalling the definition \eqref{11a} of $\tilde{a}$, \eqref{106a} gives that 
\begin{align}\label{109}
	\tilde{a}^{-\frac{1}{2}(2N(|\beta|-3)-1)}\lesssim (\<v\>^{-\gamma}\<\eta\>^{-2})^{\frac{|\beta|-3}{|\alpha|-3}\big(\frac{2|\alpha|}{\delta_1}+1\big)-1}.
\end{align}
Now we choose $\delta_1=\delta_1(\alpha,\beta)>0$ sufficiently small such that 
\begin{align}\label{110}
	-2\Big(\frac{|\beta|-3}{|\alpha|-3}\Big(\frac{2|\alpha|}{\delta_1}+1\Big)-1\Big)\le -|\beta|.
\end{align}
Then, 
\begin{align*}
\tilde{a}^{-\frac{1}{2}(2N(|\beta|-3)-1)}\lesssim \<v\>^{C_{K}}\<\eta\>^{-|\beta|}.
\end{align*}
When $|\alpha|\le 3$, $N$ can be arbitrary large. Then we choose $N$ sufficiently large that 
\begin{align*}
	\tilde{a}^{-\frac{1}{2}(2N(|\beta|-3)-1)}\lesssim \<v\>^{C_{K}}\<\eta\>^{-|\beta|}. 
\end{align*}
Thus, \eqref{107b} becomes
\begin{align*}
	\psi_{|\beta|-3-\frac{1}{2N}}\<v\>^{\frac{-l_0|\alpha|}{\delta_1}}\lesssim \frac{\delta}{C_{0,\delta}}\psi_{|\beta|-3}\tilde{a}^{1/2}+C_{\delta}\<v\>^{C_{K}}\<\eta\>^{-|\beta|}.
\end{align*}
If $|\beta|\le 3$, we choose $\eta\in(0,1)$ such that $\frac{-\eta}{2(1-\eta)}=\frac{-|\beta|}{2}$. Then 
\begin{align}\label{142}\notag
	\psi_{|\beta|-3-\frac{1}{2N}}\<v\>^{\frac{-l_0|\alpha|}{\delta_1}}=\<v\>^{\frac{-l_0|\alpha|}{\delta_1}}&\lesssim \frac{\delta}{C_{0,\delta}}\,\tilde{a}^{1/2}+C_\delta(\tilde{a}^{-1/2}\<v\>^{\frac{-l_0|\alpha|}{\delta_1}})^{\frac{\eta}{(1-\eta)}}\\
	&\lesssim \frac{\delta}{C_{0,\delta}}\,\tilde{a}^{1/2}+C_\delta\<v\>^{C_K}\<\eta\>^{-|\beta|}.
\end{align}
Thus, whenever $|\beta|\le 3$ or $|\beta|>3$, we have $\psi_{|\beta|-3-\frac{1}{2N}}\<v\>^{\frac{-l_0|\alpha|}{\delta_1}}\in S(\frac{\delta}{C_{0,\delta}}\,\tilde{a}^{1/2}+C_\delta\<v\>^{C_K}\<\eta\>^{-|\beta|})$ uniformly in $\delta$, as a symbol in $(v,\eta)$. 
Then using Lemma \ref{bound_varepsilon} with respect to $v$, we have 
\begin{align*}
	\|\psi_{|\beta|-3-\frac{1}{2N}}w^{l-|\alpha|-|\beta|}\<v\>^{\frac{-l_0|\alpha|}{\delta_1}}\partial_\beta f\|_{L^2_{v,x}}&\lesssim \frac{\delta}{C_{0,\delta}}\|\psi_{|\beta|-3}(\tilde{a}^{1/2})^w\partial_\beta f\|_{L^2_{v,x}}+C_\delta\|\<v\>^{C_{K,l}}f\|_{L^2_{v,x}}\\
	&\lesssim \frac{\delta}{C_{0,\delta}}\D^{1/2}_{K,l}+C_\delta\|\<v\>^{C_{K,l}}f\|_{L^2_{v,x}},
\end{align*}by using \eqref{144}. 
Plugging this into \eqref{101}, we have 
\begin{align}\label{101a}\notag
	&\quad\,\|\psi_{|\alpha|+|\beta|-3-\frac{1}{2N}}w^{l-|\alpha|-|\beta|}\partial^\alpha_\beta f\|^2_{L^2_{v,x}}\\&\lesssim \delta^2\|\psi_{|\alpha|+|\beta|-3}\tilde{b}^{1/2}w^{l-|\alpha|-|\beta|}(\partial^\alpha_\beta f)^\wedge(v,y)\|^2_{L^2_{v,y}}+\delta^2\D_{K,l}+C_\delta\|\<v\>^{C_{K,l}}f\|_{L^2_{v,x}}^2.
\end{align}
Now it suffices to eliminate the first right-hand term of \eqref{101a}. 

In particular, for $\gamma+2> 0$, there exists $\lambda>0$ such that 
\begin{align*}
	\tilde{a}^{-1}\lesssim \<v\>^{-\lambda(\gamma+2)}\<\eta\>^{-\lambda}. 
\end{align*}
Thus in \eqref{109} and \eqref{110}, we instead let $\delta_1=\delta_1(\alpha,\beta)>0$ sufficiently small such that 
\begin{align*}
	\tilde{a}^{-\frac{1}{2}(2N(|\beta|-3)-1)}\lesssim \<v\>^{-C}\<\eta\>^{-|\beta|},
\end{align*}where $C>0$ can be chosen arbitrary large. Continuing the estimate from \eqref{107b} to \eqref{142}, we have that, whenever $|\beta|\le 3$ or $|\beta|>3$, 
\begin{align*}
	\psi_{|\beta|-3-\frac{1}{2N}}\<v\>^{\frac{-l_0|\alpha|}{\delta_1}}\in S(\frac{\delta}{C_{0,\delta}}\,\tilde{a}^{1/2}+C_\delta\<v\>^{-C}\<\eta\>^{-|\beta|}).
\end{align*}
Then using Lemma \ref{inverse_bounded_lemma} with respect to $v$, we have 
\begin{align*}
	\|\psi_{|\beta|-3-\frac{1}{2N}}w^{l-|\alpha|-|\beta|}\<v\>^{\frac{-l_0|\alpha|}{\delta_1}}\partial_\beta f\|_{L^2_{v,x}}&\lesssim \frac{\delta}{C_{0,\delta}}\|\psi_{|\beta|-3}\partial_\beta f\|_{\sigma,l-|\beta|}+C_\delta\|f\|_{L^2_{v,x}}\\
	&\lesssim \frac{\delta}{C_{0,\delta}}\D^{1/2}_{K,l}+C_\delta\E_{K,l}.
\end{align*}
Plugging this into \eqref{101}, we have 
\begin{align}\label{101b}\notag
	&\quad\,\|\psi_{|\alpha|+|\beta|-3-\frac{1}{2N}}w^{l-|\alpha|-|\beta|}\partial^\alpha_\beta f\|^2_{L^2_{v,x}}\\&\lesssim \delta^2\|\psi_{|\alpha|+|\beta|-3}\tilde{b}^{1/2}w^{l-|\alpha|-|\beta|}(\partial^\alpha_\beta f)^\wedge(v,y)\|^2_{L^2_{v,y}}+\delta^2\D_{K,l}+C_\delta\E_{K,l}.
\end{align}

{\bf Step 2.} 
Recalling \eqref{107}, we regard $\theta$ as a symbol in $(v,\eta)$ with parameter $y$. Then, 
\begin{align*}
	|\theta(v,\eta)| = \<v\>^{l_0}|y|^{-1-\delta_2}|y\cdot\eta|\,\chi(v,\eta)
	&\lesssim 1.
\end{align*}
Direct calculation gives that $\partial^\alpha_v\partial^\beta_\eta\theta\lesssim 1$ and hence $\theta\in S(1)$ as a symbol on $(v,\eta)$. 
On the other hand, regarding the Poisson bracket on $(v,\eta)$ we have 
\begin{align*}
	\{\theta,v\cdot y\}&= \<v\>^{l_0}|y|^{1-\delta_2}+\<v\>^{l_0}|y|^{1-\delta_2}(\chi(v,\eta)-1)+\<v\>^{l_0}|y|^{-1-\delta_2}y\cdot\eta\,\partial_\eta\chi\cdot y\\
	&=: \tilde{b} + R_1 + R_2. 
\end{align*}
Now we claim that $R_1,R_2\in S(\tilde{a})$. Indeed, noticing the support of $\chi-1$, we have 
\begin{align*}
	|R_1|\le \<v\>^{l_0}\<\eta\>^{\frac{1-\delta_2}{\delta_2}}\<v\>^{l_0\frac{1-\delta_2}{\delta_2}}
	&\le \<v\>^{\gamma}\<\eta\>^{2}\le \tilde{a}, 
\end{align*}by \eqref{106b}. For $R_2$, since $1-3\delta_2\le 0$, we have 
\begin{align*}
	|R_2|\le\<v\>^{2l_0}|y|^{1-2\delta_2}|\eta|\1_{\<\eta\>\<v\>^{l_0}\approx|y|^{\delta_2}}\le \<v\>^{\frac{l_0}{\delta_2}}\<\eta\>^{\frac{1-\delta_2}{\delta_2}}\le \tilde{a}.
\end{align*}Higher derivative estimate can be calculated by Leibniz's formula and hence, $R_1,R_2\in S(\tilde{a})$. Thus, by Lemma \ref{innerproduct} and \eqref{compostion}, we have 
\begin{align}\label{102}\notag
	\|\tilde{b}^{1/2}\widehat{g}(v,y)\|^2_{L^2_{v,y}} &=\big(\tilde{b}(v,y)\widehat{g},\widehat{g}\big)_{L^2_{v,y}} \\\notag
	&= \Re\big(\{\theta,v\cdot y\}^w(v,D_v)\widehat{g},\widehat{g}\big)_{L^2_{v,y}} + \Re ((R_1+R_2)^w(v,D_v)\widehat{g},\widehat{g})_{L^2_{v,y}}\\\notag
	&\lesssim \Re\big(\mathbf{i}v\cdot y\widehat{g},\theta^w(v,D_v)\widehat{g}\big)_{L^2_{v,y}}+\|(\tilde{a}^{1/2})^wg\|^2_{L^2_{v,x}}\\
	&\lesssim \Re\big(v\cdot \nabla_x{g},(\theta^w\widehat{g})^\vee\big)_{L^2_{v,x}}+\|g\|^2_{\sigma,0},
\end{align}
for any $g$ in a suitable smooth space. Here and after, we write  $\theta^w=\theta^w(v,D_v)$.

Now we let $g = \psi_{|\alpha|+|\beta|-3}w^{l-|\alpha|-|\beta|}\partial^\alpha_\beta f_\pm e^{\frac{\pm\phi}{2}}$ in \eqref{102}, then 
\begin{align}\label{104}
	&\notag\quad\,\|\tilde{b}^{1/2}\psi_{|\alpha|+|\beta|-3}w^{l-|\alpha|-|\beta|}(\partial^\alpha_\beta f_\pm)^\wedge(v,y) e^{\frac{\pm\phi}{2}}\|_{L^2_{v,x}}\\
	&\lesssim\Re\big(v\cdot \nabla_x{\psi_{|\alpha|+|\beta|-3}w^{l-|\alpha|-|\beta|}\partial^\alpha_\beta f}e^{\frac{\pm\phi}{2}},(\theta^w\psi_{|\alpha|+|\beta|-3}w^{l-|\alpha|-|\beta|}{(\partial^\alpha_\beta f_\pm e^{\frac{\pm\phi}{2}})^\wedge})^\vee\big)_{L^2_{v,x}}+\D_{K,l}\\
	&=: K_0 + \D_{K,l}.\notag
\end{align}
Here, by equation \eqref{7}, we have 
\begin{align*}
	&\quad\,v\cdot\nabla_x(\partial^\alpha_\beta f_\pm e^{\frac{\pm\phi}{2}})\\
	&= v_i\partial^{\alpha+e_i}_\beta f_\pm e^{\frac{\pm\phi}{2}} \pm \frac{1}{2}v_i\partial^{e_i}\phi e^{\frac{\pm\phi}{2}}\partial^\alpha_\beta f_\pm\\
	&= \partial_\beta\big(v_i\partial^{\alpha+e_i}f_\pm e^{\frac{\pm\phi}{2}}\big)-\sum_{0\neq \beta_1\le \beta}C^{\beta_1}_{\beta}\partial_{\beta_1}v_i\partial^{\alpha+e_i}_{\beta-\beta_1}f_\pm e^{\frac{\pm\phi}{2}}\pm \frac{1}{2}v_i\partial^{e_i}\phi e^{\frac{\pm\phi}{2}}\partial^\alpha_\beta f_\pm\\
	&=-\partial_t\partial^\alpha_\beta f_\pm e^{\frac{\pm\phi}{2}} \mp \frac{1}{2}\sum_{\alpha_1\le\alpha}\sum_{\beta_1\le\beta}C^{\alpha_1,\beta_1}_{\alpha,\beta}\partial^{e_i+\alpha_1}\phi\partial_{\beta_1}v_i\partial^{\alpha-\alpha_1}_{\beta-\beta_1}f_\pm e^{\frac{\pm\phi}{2}}\\
	&\qquad\pm \sum_{\alpha_1\le\alpha}C^{\alpha_1}_\alpha\partial^{e_i+\alpha_1}\phi\partial^{\alpha-\alpha_1}_{\beta+e_i}f_\pm e^{\frac{\pm\phi}{2}} 
	 \mp \partial^{e_i+\alpha}\phi\partial_\beta(v_i\mu^{1/2})e^{\frac{\pm\phi}{2}}+\partial^\alpha_\beta L_\pm fe^{\frac{\pm\phi}{2}}\\
	 &\qquad+\partial^\alpha_\beta\Gamma_\pm(f,f)e^{\frac{\pm\phi}{2}} -\sum_{0\neq \beta_1\le \beta}\partial_{\beta_1}C^{\beta_1}_{\beta}v_i\partial^{\alpha+e_i}_{\beta-\beta_1}f_\pm e^{\frac{\pm\phi}{2}}\pm \frac{1}{2}v_i\partial^{e_i}\phi e^{\frac{\pm\phi}{2}}\partial^\alpha_\beta f_\pm
\end{align*}
Thus, 
\begin{align*}
 &K_0 =
 \Re\big(-\psi_{2|\alpha|+2|\beta|-6}w^{l-|\alpha|-|\beta|}\partial_t\partial^\alpha_\beta f_\pm e^{\frac{\pm\phi}{2}},(\theta^ww^{l-|\alpha|-|\beta|}{(\partial^\alpha_\beta f_\pm e^{\frac{\pm\phi}{2}})^\wedge})^\vee\big)_{L^2_{v,x}} \\
&\mp \Re\big(\psi_{2|\alpha|+2|\beta|-6}w^{l-|\alpha|-|\beta|}\frac{1}{2}\sum_{\alpha_1\le\alpha}\sum_{\beta_1\le\beta}C^{\alpha_1,\beta_1}_{\alpha,\beta}\partial^{e_i+\alpha_1}\phi\partial_{\beta_1}v_i\partial^{\alpha-\alpha_1}_{\beta-\beta_1}f_\pm e^{\frac{\pm\phi}{2}},(\theta^ww^{l-|\alpha|-|\beta|}{(\partial^\alpha_\beta f_\pm e^{\frac{\pm\phi}{2}})^\wedge})^\vee\big)_{L^2_{v,x}} \\
& \pm \Re\big(\psi_{2|\alpha|+2|\beta|-6}w^{l-|\alpha|-|\beta|}\sum_{\alpha_1\le\alpha}C^{\alpha_1}_\alpha\partial^{e_i+\alpha_1}\phi\partial^{\alpha-\alpha_1}_{\beta+e_i}f_\pm e^{\frac{\pm\phi}{2}},(\theta^ww^{l-|\alpha|-|\beta|}{(\partial^\alpha_\beta f_\pm e^{\frac{\pm\phi}{2}})^\wedge})^\vee\big)_{L^2_{v,x}} \\
 &\mp \Re\big(\psi_{2|\alpha|+2|\beta|-6}w^{l-|\alpha|-|\beta|}\partial^{e_i+\alpha}\phi\partial_\beta(v_i\mu^{1/2})e^{\frac{\pm\phi}{2}},(\theta^ww^{l-|\alpha|-|\beta|}{(\partial^\alpha_\beta f_\pm e^{\frac{\pm\phi}{2}})^\wedge})^\vee\big)_{L^2_{v,x}}
 \\&+\Re\big(\psi_{2|\alpha|+2|\beta|-6}w^{l-|\alpha|-|\beta|}\partial^\alpha_\beta L_\pm f e^{\frac{\pm\phi}{2}}, (\theta^ww^{l-|\alpha|-|\beta|}{(\partial^\alpha_\beta f_\pm e^{\frac{\pm\phi}{2}})^\wedge})^\vee\big)_{L^2_{v,x}}  \\ &+\Re\big(\psi_{2|\alpha|+2|\beta|-6}w^{l-|\alpha|-|\beta|}\partial^\alpha_\beta\Gamma_\pm(f,f)e^{\frac{\pm\phi}{2}},(\theta^ww^{l-|\alpha|-|\beta|}{(\partial^\alpha_\beta f_\pm e^{\frac{\pm\phi}{2}})^\wedge})^\vee\big)_{L^2_{v,x}}  \\
 &-\Re\big(\psi_{2|\alpha|+2|\beta|-6}w^{l-|\alpha|-|\beta|}\sum_{0\neq \beta_1\le \beta}C^{\beta_1}_{\beta}\partial_{\beta_1}v_i\partial^{\alpha+e_i}_{\beta-\beta_1}f_\pm e^{\frac{\pm\phi}{2}},(\theta^ww^{l-|\alpha|-|\beta|}{(\partial^\alpha_\beta f_\pm e^{\frac{\pm\phi}{2}})^\wedge})^\vee\big)_{L^2_{v,x}}\\
 &\pm \Re\big(\psi_{2|\alpha|+2|\beta|-6}w^{l-|\alpha|-|\beta|}\frac{1}{2}v_i\partial^{e_i}\phi e^{\frac{\pm\phi}{2}}\partial^\alpha_\beta f_\pm,(\theta^ww^{l-|\alpha|-|\beta|}{(\partial^\alpha_\beta f_\pm e^{\frac{\pm\phi}{2}})^\wedge})^\vee\big)_{L^2_{v,x}}.
\end{align*}
Denote these terms by $K_1$ to $K_8$. Noticing that there's coefficient $\delta$ in \eqref{101a}, we only need to have a upper bound for these terms. 
For $K_1$, noticing that $\theta^w$ is self-adjoint, 
\begin{align*}
	K_1&\le \frac{1}{2}\partial_t\big(-\psi_{2|\alpha|+2|\beta|-6}w^{l-|\alpha|-|\beta|}\partial^\alpha_\beta f_\pm e^{\frac{\pm\phi}{2}},(\theta^ww^{l-|\alpha|-|\beta|}{(\partial^\alpha_\beta f_\pm e^{\frac{\pm\phi}{2}})^\wedge})^\vee\big)_{L^2_{v,x}}\\
	&\qquad+C\big|\big(-\psi_{2|\alpha|+2|\beta|-3-\frac{1}{N}}w^{l-|\alpha|-|\beta|}\partial^\alpha_\beta f_\pm e^{\frac{\pm\phi}{2}},(\theta^ww^{l-|\alpha|-|\beta|}{(\partial^\alpha_\beta f_\pm e^{\frac{\pm\phi}{2}})^\wedge})^\vee\big)_{L^2_{v,x}}\big|\\
	&\qquad+C\big|\big(\partial_t\phi\psi_{2|\alpha|+2|\beta|-6}w^{l-|\alpha|-|\beta|}\partial^\alpha_\beta f_\pm e^{\frac{\pm\phi}{2}},(\theta^ww^{l-|\alpha|-|\beta|}{(\partial^\alpha_\beta f_\pm e^{\frac{\pm\phi}{2}})^\wedge})^\vee\big)_{L^2_{v,x}}\big|.
\end{align*}
We denote the second and third term on the right hand side by $K_{1,1}$ and $K_{1,2}$. Since $\theta\in S(1)$, $\theta^w$ is a bounded operator on $L^2_{v,y}$. Using the trick from \eqref{100a}-\eqref{101b} to the first $f_\pm$ in $K_{1,1}$, we have 
\begin{align*}
	K_{1,1}\lesssim \delta^2\|\psi_{|\alpha|+|\beta|-3}\tilde{b}^{1/2}w^{l-|\alpha|-|\beta|}(\partial^\alpha_\beta f)^\wedge\|_{L^2_{v,y}}^2+\delta^2\D_{K,l}+C_\delta\|\<v\>^{C_{K,l}}f\|_{L^2_{v,x}}^2+\E_{K,l},
\end{align*}
when $-1\le\gamma+2\le 0$ and 
\begin{align*}
	K_{1,1}\lesssim \delta^2\|\psi_{|\alpha|+|\beta|-3}\tilde{b}^{1/2}w^{l-|\alpha|-|\beta|}(\partial^\alpha_\beta f)^\wedge\|_{L^2_{v,y}}^2+\delta^2\D_{K,l}+C_\delta\E_{K,l},
\end{align*}when $\gamma+2>0$. 
The boundedness of $\theta^w$ will be frequently used in the following without further mentioned. The term $K_{1,2}$ is similar to the case $I_1$, i.e.
\begin{align*}
	K_{1,2} \lesssim \|\partial_t\phi\|_{L^\infty_x}\|\psi_{|\alpha|+|\beta|-3}w^{l-|\alpha|-|\beta|}\partial^\alpha_\beta f\|_{L^2_{v,x}}^2\lesssim \|\partial_t\phi\|_{L^\infty_x}\E_{K,l}(t). 
\end{align*}
For the term $K_2$ with $\alpha_1=\beta_1=0$, a nice observation is that it's the same as $K_8$ except the sign and hence, they are eliminated. For $K_2$ with $\alpha_1+\beta_1\neq 0$, the order of derivatives for the first $f_\pm$ is less or equal to $K-1$ and hence, the weight can be controlled as $w^{l-|\alpha|-|\beta|}\partial_{\beta_1}v_i\lesssim \<v\>^{\gamma}w^{l-|\alpha-\alpha_1|-|\beta-\beta_1|}$. Then similar to Lemma \ref{Lem26}, by noticing $\theta\in S(1)$, 
\begin{align*}
	|K_2+K_8|\lesssim \E^{1/2}_{K,l}\D_{K,l}.
\end{align*}
For $K_3$, when $\alpha_1=0$, noticing $\theta^w$ is self-adjoint, we use integration by parts over $v$ to obtain 
\begin{align*}
|K_3| &= \big|\big(\psi_{2|\alpha|+2|\beta|-6}w^{l-|\alpha|-|\beta|}\partial^{e_i}\phi\partial^{\alpha}_{\beta+e_i}f_\pm e^{\frac{\pm\phi}{2}},(\theta^ww^{l-|\alpha|-|\beta|}{(\partial^\alpha_\beta f_\pm e^{\frac{\pm\phi}{2}})^\wedge})^\vee\big)_{L^2_{v,x}}\big|\\
&\lesssim \big|\big(\psi_{2|\alpha|+2|\beta|-6}\partial_{e_i}(w^{l-|\alpha|-|\beta|})\partial^{e_i}\phi\partial^{\alpha}_{\beta}f_\pm e^{\frac{\pm\phi}{2}},(\theta^ww^{l-|\alpha|-|\beta|}{(\partial^\alpha_\beta f_\pm e^{\frac{\pm\phi}{2}})^\wedge})^\vee\big)_{L^2_{v,x}}\big|\\
&\quad+\big|\big(\psi_{2|\alpha|+2|\beta|-6}w^{l-|\alpha|-|\beta|}\partial^{e_i}\phi\partial^{\alpha}_{\beta}f_\pm e^{\frac{\pm\phi}{2}},(\underbrace{[\partial_{e_i},\theta^w]}_{\in S(1)}w^{l-|\alpha|-|\beta|}{(\partial^\alpha_\beta f_\pm e^{\frac{\pm\phi}{2}})^\wedge})^\vee\big)_{L^2_{v,x}}\big|\\
&\quad+\big|\big(\psi_{2|\alpha|+2|\beta|-6}w^{l-|\alpha|-|\beta|}\partial^{e_i}\phi\partial^{\alpha}_{\beta}f_\pm e^{\frac{\pm\phi}{2}},(\theta^w\partial_{e_i}(w^{l-|\alpha|-|\beta|}){(\partial^\alpha_\beta f_\pm e^{\frac{\pm\phi}{2}})^\wedge})^\vee\big)_{L^2_{v,x}}\big|\\
&\lesssim \|\partial^{e_i}\phi\|_{H^2_x}\|\psi_{|\alpha|+|\beta|-3}w^{l-|\alpha|-|\beta|}\partial^{\alpha}_{\beta}f_\pm\|_{L^2_{v,x}}^2\\
&\lesssim \delta_0\E_{K,l}(t),
\end{align*}by \eqref{13} and $\theta\in S(1)$. 
If $\alpha_1\neq 0$, then $\alpha\neq 0$, the total number of derivatives on the first $f_\pm$ is less or equal to $K$ and there's at least one derivative on the second $f_\pm$ with respect to $x$. Thus, using the discussion on $|\alpha_1|$ as \eqref{33aa}-\eqref{33bb}, we have 
\begin{align*}
	|K_3|\lesssim \E^{1/2}_{K,l}\D_{K,l}.
\end{align*}
For $K_4$, there's exponential decay in $v$ and hence 
	$|K_4|\lesssim \E_{K,l}. $
For $K_5$, recalling that we only need upper bound, using Lemma \ref{lemmat}, we have $|K_5|\lesssim\E_{K,l}.$
For $K_6$, we use Lemma \ref{lemmag} to obtain
\begin{align*}
	|K_6|\lesssim\E^{1/2}_{K,l}\D_{K,l}+\E_{K,l}\D^{1/2}_{K,l}\lesssim (\E^{1/2}_{K,l}+\E_{K,l})\D_{K,l}+\E_{K,l}.
\end{align*}
For $K_7$, since $\beta_1\neq 0$, $|\partial_{\beta_1}v_i|\lesssim 1$ and the total number of derivatives on the first $f_\pm$ is less or equal to $K$. This yields that $|K_7|\lesssim \E_{K,l}$. Combining the above estimate with \eqref{104} and choosing $\delta_0>0$ sufficiently small, we have 
	\begin{align*}
		&\notag\quad\,\|\psi_{|\alpha|+|\beta|-3}\tilde{b}^{1/2}w^{l-|\alpha|-|\beta|}(\partial^\alpha_\beta f)^\wedge(v,y)\|^2_{L^2_{v,y}}\\
		&\lesssim \frac{1}{2}\partial_t\big(-\psi_{2|\alpha|+2|\beta|-6}w^{l-|\alpha|-|\beta|}\partial^\alpha_\beta f_\pm e^{\frac{\pm\phi}{2}},(\theta^ww^{l-|\alpha|-|\beta|}{(\partial^\alpha_\beta f_\pm e^{\frac{\pm\phi}{2}})^\wedge})^\vee\big)_{L^2_{v,x}}\\
		&\notag\qquad+(\E^{1/2}_{K,l}+\E_{K,l})\D_{K,l}+\|\<v\>^{C_{K,l}}f\|_{L^2_{v,x}}^2+\|\partial_t\phi\|_{L^\infty_x}\E_{K,l}(t)+\D_{K,l}+\E_{K,l},
	\end{align*}when $-1\le\gamma+2\le 0$, and 
\begin{align*}
	&\notag\quad\,\|\psi_{|\alpha|+|\beta|-3}\tilde{b}^{1/2}w^{l-|\alpha|-|\beta|}(\partial^\alpha_\beta f)^\wedge(v,y)\|^2_{L^2_{v,y}}\\
	&\lesssim \frac{1}{2}\partial_t\big(-\psi_{2|\alpha|+2|\beta|-6}w^{l-|\alpha|-|\beta|}\partial^\alpha_\beta f_\pm e^{\frac{\pm\phi}{2}},(\theta^ww^{l-|\alpha|-|\beta|}{(\partial^\alpha_\beta f_\pm e^{\frac{\pm\phi}{2}})^\wedge})^\vee\big)_{L^2_{v,x}}\\
	&\notag\qquad+(\E^{1/2}_{K,l}+\E_{K,l})\D_{K,l}+\|\partial_t\phi\|_{L^\infty_x}\E_{K,l}(t)+\D_{K,l}+\E_{K,l},
\end{align*}when $\gamma+2>0$. 
Substituting these into \eqref{101a} and \eqref{101b} respectively, we have the desired estimate. 

\qe\end{proof}

\begin{proof}[Proof of Theorem \ref{lem51}]
	We only prove the case of $-1\le\gamma+2\le 0$. The case $\gamma+2>0$ is similar and the only difference in this case is that there's no $\|\<v\>^{C_{K,l}}f\|_{L^2_{v,x}}$ in the estimate \eqref{137}. 
Substituting \eqref{112} into \eqref{72}, we have that for $0<\delta<1$, 
\begin{align*}
	\partial_t\E_{K,l}(t)+\lambda D_{K,l}(t)&\lesssim \delta^2\sum_{|\alpha|+|\beta|\le K}\partial_t\big(-\psi_{2|\alpha|+2|\beta|-6}w^{l-|\alpha|-|\beta|}\partial^\alpha_\beta f_\pm e^{\frac{\pm\phi}{2}},(\theta^ww^{l-|\alpha|-|\beta|}{(\partial^\alpha_\beta f_\pm e^{\frac{\pm\phi}{2}})^\wedge})^\vee\big)_{L^2_{v,x}}\\
	&\notag\qquad+ \|\partial_t\phi\|_{L^\infty_x}\E_{K,l}(t) +\delta^2\big(\D_{K,l}+(\E^{1/2}_{K,l}+\E_{K,l})\D_{K,l}+\E_{K,l}\big)+C_\delta\|\<v\>^{C_{K,l}}f\|_{L^2_{v,x}}^2,
\end{align*}
Notice that $\|\partial_t\phi\|_{L^\infty_x}\lesssim \E^{1/2}_{K,l}\lesssim \delta^{1/2}_0$ by \eqref{34} and \eqref{priori1}.
 Using the $a$ $priori$ assumption \eqref{priori1} and choosing $\delta,\delta_0>0$ sufficiently small, we have 
\begin{align*}
	\partial_t\E_{K,l}(t)+\lambda \D_{K,l}(t)\lesssim\delta^2\sum_{|\alpha|+|\beta|\le K}\partial_t\big(-\psi_{2|\alpha|+2|\beta|-6}w^{l-|\alpha|-|\beta|}\partial^\alpha_\beta f_\pm e^{\frac{\pm\phi}{2}},(\theta^ww^{l-|\alpha|-|\beta|}{(\partial^\alpha_\beta f_\pm e^{\frac{\pm\phi}{2}})^\wedge})^\vee\big)_{L^2_{v,x}}\\
	+\E_{K,l}(t)+\|\<v\>^{C_{K,l}}f\|_{L^2_{v,x}}^2.\qquad
\end{align*}
By solving this ODE with neglecting $\lambda\D_{K,l}(t)$ and noticing 
\begin{align*}
	\big|\big(-\psi_{2|\alpha|+2|\beta|-6}w^{l-|\alpha|-|\beta|}\partial^\alpha_\beta f_\pm e^{\frac{\pm\phi}{2}},(\theta^ww^{l-|\alpha|-|\beta|}{(\partial^\alpha_\beta f_\pm e^{\frac{\pm\phi}{2}})^\wedge})^\vee\big)_{L^2_{v,x}}\big|
	\lesssim \E_{K,l}(t),
\end{align*}
 we have that for $0\le t\le t_0$,
\begin{align}\notag
	\E_{K,l}(t) &\lesssim \E_{K,l}(0)+\delta^2\E_{K,l}(t)+\delta^2\E_{K,l}(0)+\int^t_0(\E_{K,l}+\|\<v\>^{C_{K,l}}f\|_{L^2_{v,x}})\,d\tau,\\
	\E_{K,l}(t)&\lesssim \epsilon^2_1,\label{76}
\end{align}by choosing $\delta>0$ and $t_0=t_0(\epsilon_1,\|\<v\>^{C_{K,l}}f\|_{L^2_{v,x}})>0$ sufficiently small. 
Here we used $\E_{K,l}(0)\le \E_{3,l}(0)$.

\qe\end{proof}

\begin{proof}
	[Proof of Theorem \ref{main2}]
It follows immediately from the $a$ $priori$ estimate \eqref{priori1} and Theorem \ref{lem51} that $\sup_{0\le t\le t_0}\E_{K,l}\lesssim \epsilon^2_1$ holds true for some small $t_0>0$, as long as $\epsilon_1$ is sufficiently small. 
The rest is to prove the local existence and uniqueness of solutions in terms of the energy norm $\E_{K,l}$. The details of proof is the same as Theorem \ref{main1} and is omitted for brevity; see \cite{Guo2012, Strain2013} and \cite{Gressman2011}. 
		
		Notice that the constants in Lemma \ref{lem51} are independent of time $t$ and hence, we can apply Theorem \ref{lem51} to any time interval with length less than $t_0$ to obtain that for $0<\tau<T$, 
		\begin{align}\label{106}
			\sup_{\tau\le t\le T}\E_{K,l}(t)\lesssim \epsilon_1^2C_{T}. 
		\end{align}
	Recalling Definition \eqref{Defe} of $\E_{K,l}$ and the choice \eqref{93} of $\psi$, we have that for any $0<\tau<T$ and $l\ge K\ge 0$,
	\begin{align}\label{79}
	\sup_{\tau\le t\le T}\sum_{|\alpha|+|\beta|\le K}\|w^{l-|\alpha|-|\beta|}\partial^\alpha_\beta f\|^2_{L^2_{v,x}}+\sup_{\tau\le t\le T}\sum_{|\alpha|\le K}\|\partial^\alpha\nabla_x\phi\|_{L^2_x}^2\le  C_{\tau,T}<\infty.
\end{align}	Notice that $\psi_{|\alpha|+|\beta|-3}^{-1}$ is singular near $t=0$ when $|\alpha|+|\beta|>3$, so the constant is necessarily depending on $\tau$. 
This proves \eqref{19a}.

If additionally $\sup_{l^*}\E_{4,l^*}(0)$ is sufficiently small. Then for $l_0\ge K\ge 3$, by \eqref{19a}, we have 
\begin{align*}
\sup_{\tau\le t\le T}\sum_{|\alpha|+|\beta|\le K}\|w^{l_0-|\alpha|-|\beta|}\partial^\alpha_\beta f\|^2_{L^2_{v,x}}\le C_{\tau,T}.
\end{align*}
For the regularity on $t$, the technique above is not applicable and we only make a rough estimate. For any $t>0$, applying $w^l\partial^k_t\partial^\alpha_\beta$ with $k,l\ge 0$, $|\alpha|+|\beta|\le K$ to equation \eqref{7} and taking $L^2_{v,x}$ norms, we have   
\begin{align}\label{105}
\|w^l\partial^{k+1}_t\partial^\alpha_\beta f_\pm\|^2_{L^2_{v,x}}
&\notag\lesssim \|w^lv\cdot\nabla_x\partial^k_t\partial^\alpha_\beta f_\pm\|^2_{L^2_{v,x}}+\|w^l\sum_{k_1\le k}\partial^{\alpha}_\beta\big(\partial^{k_1}_t\nabla_x\phi\cdot v\partial^{k-k_1}_tf_\pm\big)\|^2_{L^2_{v,x}}
\\
&\qquad+\|w^l\sum_{k_1\le k}\partial^\alpha\big(\partial^{k_1}_t\nabla_x\phi\cdot\nabla_v\partial^{k-k_1}_t\partial_\beta f_\pm\big)\|_{L^2_{v,x}}^2+\|w^l\partial^k_t\partial^\alpha\nabla_x\phi\cdot \partial_\beta(v\mu^{1/2})\|^2_{L^2_{v,x}}\\
&\qquad+\|w^l\partial^\alpha_\beta L_\pm \partial^k_tf_\pm\|^2_{L^2_{v,x}} + \|w^l\sum_{k_1\le k}\partial^\alpha_\beta\Gamma_\pm(\partial^{k_1}_tf,\partial^{k-k_1}_tf)\|_{L^2_{v,x}}^2. \notag
\end{align}
Let $l_3=1$ for $\gamma+2\ge0$ and $l_3=-\gamma$ for $-1\le\gamma+2<0$. 
Denoting $\E_{K,l,k}=\sum_{|\alpha|+|\beta|\le K,k_1\le k}\|w^l\partial^\alpha_\beta\partial^{k_1}_tf\|_{L^2_{v,x}}$, we estimate the right-hand terms one by one. The first term on the right hand is bounded above by $\E_{K+1,l+\frac{1}{l_3},k}$.
For terms involving both $\phi$ and $f_\pm$, we use \eqref{13} to generate one more $x$ derivative on $\phi$. Applying the trick in Lemma \ref{Lem26}, the second term is bounded above by 
\begin{align*}
\sum_{|\alpha|+|\beta|\le K+1,\,k_1\le k}\|\partial^{k_1}_t\partial^\alpha_\beta\nabla_x\phi\|^2_{L^2_x}\sum_{|\alpha|+|\beta|\le K+1,\,k_1\le k}\|w^{l+\frac{1}{l_3}}\partial^{k_1}_t\partial^\alpha_\beta f_\pm\|^2_{L^2_{v,x}}\lesssim \E_{K+1,l+\frac{1}{l_3},k}^2.
\end{align*}
Similarly, applying the trick in Lemma \ref{Lem27}, the third term is bounded above by $\E_{K+1,l+\frac{1}{l_3},k}^2.$ For the fourth term, when $k=0$, it's bounded above by $\E_{K,l,0}$. When $k\ge 1$, by using \eqref{34a}, it's bounded above by $\E_{K,l,k-1}$. For the fifth term, noticing $L_\pm\in S(\tilde{a})\subset S(\<v\>^{\gamma+2}\<\eta\>^{2})$ and $s\in(0,1)$, we have 
\begin{align*}
\|w^l\partial^\alpha_\beta L_\pm \partial^k_tf_\pm\|^2_{L^2_{v,x}}\lesssim \|w^{l-\frac{\gamma+2}{\gamma}}\<D_v\>^2\<(D_x,D_v)\>^K \partial^k_tf_\pm\|^2_{L^2_{v,x}}\lesssim \E_{K+2,l-\frac{\gamma+2}{\gamma},k}.
\end{align*}
For the last term, using \eqref{12a} and trick \eqref{12b}, it's bounded above by 
\begin{align*}
\sum_{|\alpha|+|\beta|\le K+2,\,k_1\le k}\|w^{l+\frac{\gamma+2}{2l_3}}\partial^\alpha_\beta\partial^{k_1}_tf\|^2_{L^2_{v,x}}\lesssim \E_{K+2,l+\frac{\gamma+2}{2l_3},k}^2. 
\end{align*} 
Combining the above estimate and taking summation $|\alpha|+|\beta|\le K$, $k\le k_0$ for any $k_0\ge 0$, we have 
\begin{align*}
\E_{K,l,k_0+1}(t)\lesssim \E_{K,l,0} +\E_{K,l,k_0-1}+\E_{K+1,l+\frac{1}{l_3},k}+\E_{K+1,l+\frac{1}{l_3},k}^2+\E_{K+2,l+\frac{\gamma+2}{l_3},k}+ \E_{K+2,l+\frac{\gamma+2}{2l_3},k_0}^2.
\end{align*} 
The $t$ derivative on the right hand is less than the left hand. Then by induction, noticing \eqref{106}, we have
\begin{align*}
\sup_{\tau\le t\le T}\E_{K,l_0,k_0}(t)\le C_{\tau,T,l_0,k_0},
\end{align*}for any $T>\tau>0$.
The same standard argument for obtaining the local solution gives the result of Theorem \ref{main2}(1) and the details are omitted for brevity; see \cite{Guo2012,Gressman2011}. Consequently, by Sobolev embedding, $f\in C^\infty(\R^+_t;C^\infty(\R^3_x;\mathscr{S}(\R^3_v)))$.

Now we additionally assume \eqref{19aa}.
Noticing $\psi=1$ in Theorem \ref{main1}, \eqref{15a} shows that for any $\tau_0\ge \tau$, 
\begin{align*}
\sum_{|\alpha|\le 3}\|\partial^\alpha E(\tau_0)\|^2_{L^2_x}+\sum_{|\alpha|\le 3}\|\partial^\alpha\P f(\tau_0)\|^2_{L^2_{v,x}}+\sum_{\substack{|\alpha|+|\beta|\le 3}}\|w^{l-|\alpha|-|\beta|}\partial^\alpha_\beta(\I-\P) f(\tau_0)\|^2_{L^2_{v,x}}
\lesssim \epsilon_0^2.
\end{align*}
Using this as the initial data instead of \eqref{15aa}, we can apply the above calculation on any time interval $[\tau_0,\tau_0+t_0]$ to obtain the same estimate with constants independent of $T$. In this case, we use 
\begin{align*}
\sup_{\tau_0\le t\le \tau_0+t_0}\E_{K,l}(t)\lesssim \epsilon_1^2C_{\tau}
\end{align*}instead of \eqref{106}, where the constant $C_\tau$ is independent of $\tau_0$ and $T$. Since the choice $t_0>0$ in \eqref{76} is uniform in $t$, we can obtain a uniform estimate independent of time $T$ and this completes the proof of Theorem \ref{main2}(2). Notice that the estimate of \eqref{19a} and \eqref{19b} are necessarily depending on $\tau$ since $\psi_{|\alpha|+|\beta|-3}^{-1}$ is singular near $t=0$ when $|\alpha|+|\beta|>3$.  
 
\qe\end{proof}

\section{Appendix}

\paragraph{Pseudo-differential calculus}

We recall some notation and theorem of pseudo differential calculus. For details, one may refer to Chapter 2 in the book \cite{Lerner2010} for details. Set $\Gamma=|dv|^2+|d\eta|^2$, but also note that the following are also valid for general admissible metric.
Let $M$ be an $\Gamma$-admissible weight function. That is, $M:\R^{2d}\to (0,+\infty)$ satisfies the following conditions:\\
(a). (slowly varying) there exists $\delta>0$ such that for any $X,Y\in\R^{2d}$, $|X-Y|\le \delta$ implies
\begin{align*}
	M(X)\approx M(Y);
\end{align*}
(b) (temperance) there exists $C>0$, $N\in\R$, such that for $X,Y\in \R^{2d}$,
\begin{align*}
	\frac{M(X)}{M(Y)}\le C\<X-Y\>^N.
\end{align*}
Consider symbols $a(v,\eta,\xi)$ as a function of $(v,\eta)$ with parameters $\xi$. We say that
$a\in S(M)=S(M,\Gamma)$ uniformly in $\xi$, if for $\alpha,\beta\in \N^d$, $v,\eta\in\Rd$,
\begin{align*}
	|\partial^\alpha_v\partial^\beta_\eta a(v,\eta,\xi)|\le C_{\alpha,\beta}M,
\end{align*}with $C_{\alpha,\beta}$ a constant depending only on $\alpha$ and $\beta$, but independent of $\xi$. The space $S(M,\Gamma)$ endowed with the seminorms
\begin{align*}
	\|a\|_{k;S(M,\Gamma)} = \max_{0\le|\alpha|+|\beta|\le k}\sup_{(v,\eta)\in\R^{2d}}
	|M(v,\eta)^{-1}\partial^\alpha_v\partial^\beta_\eta a(v,\eta,\xi)|,
\end{align*}becomes a Fr\'{e}chet space.
Sometimes we write $\partial_\eta a\in S(M,\Gamma)$ to mean that $\partial_{\eta_j} a\in S(M,\Gamma)$ $(1\le j\le d)$ equipped with the same seminorms.
We formally define the pseudo-differential operator by
\begin{align*}
	(op_ta)u(x)=\int_\Rd\int_\Rd e^{2\pi i (x-y)\cdot\xi}a((1-t)x+ty,\xi)u(y)\,dyd\xi,
\end{align*}for $t\in\R$, $f\in\S$.
In particular, denote $a(v,D_v)=op_0a$ to be the standard pseudo-differential operator and
$a^w(v,D_v)=op_{1/2}a$ to be the Weyl quantization of symbol $a$. We write $A\in Op(M,\Gamma)$ to represent that $A$ is a Weyl quantization with symbol belongs to class $S(M,\Gamma)$. One important property for Weyl quantization of a real-valued symbol is the self-adjoint on $L^2$ with domain $\S$. 

For composition of pseudodifferential operator we have $a^wb^w= (a\#b)^w$ with 
\begin{align}\label{compostion}
	a\#b = ab + \frac{1}{4\pi i}\{a,b\} + \sum_{2\le k\le \nu}2^{-k}\sum_{|\alpha|+|\beta|=k}\frac{(-1)^{|\beta|}}{\alpha!\beta!}D^\alpha_\eta\partial^\beta_xaD^{\beta}_\eta\partial^\alpha_xb+r_\nu(a,b),
\end{align}where $X=(v,\eta)$,
\begin{align*}
	r_\nu(a,b)(X) & = R_\nu(a(X)\otimes b(Y))|_{X=Y},\\
	R_\nu &= \int^1_0\frac{(1-\theta)^{\nu-1}}{(\nu-1)!}\exp\Big(\frac{\theta}{4\pi i}\<\sigma\partial_X,\partial_Y\>\Big)\,d\theta\Big(\frac{1}{4\pi i}\<\sigma\partial_X,\partial_Y\>\Big)^\nu.
\end{align*}
Thus if $\partial_{\eta}a_1,\partial_{\eta}a_2\in S(M'_1,\Gamma)$ and $\partial_{v}a_1,\partial_{v}a_2\in S(M'_2,\Gamma)$, then $[a_1,a_2]\in S(M'_1M'_2,\Gamma)$, where $[\cdot,\cdot]$ is the commutator defined by $[A,B]:=AB-BA$.

We can define a Hilbert space $H(M,\Gamma):=\{u\in\S':\|u\|_{H(M,\Gamma)}<\infty\}$, where
\begin{align}\label{sobolev_space}
	\|u\|_{H(M,\Gamma)}:=\int M(Y)^2\|\varphi^w_Yu\|^2_{L^2}|g_Y|^{1/2}\,dY<\infty,
\end{align}and $(\varphi_Y)_{Y\in\R^{2d}}$ is any uniformly confined family of symbols which is a partition of unity. If $a\in S(M)$ is a isomorphism from $H(M')$ to $H(M'M^{-1})$, then $(a^wu,a^wv)$ is an equivalent Hilbertian structure on $H(M)$. Moreover, the space $\S(\Rd)$ is dense in $H(M)$ and $H(1)=L^2$.


The following Lemmas come from \cite{Deng2020a}. 
\begin{Lem}\label{inverse_bounded_lemma}Let $m,c$ be $\Gamma$-admissible weight and $a\in S(m)$.
	Assume $a^w:H(mc)\to H(c)$ is invertible.
	If $b\in S(m)$, then there exists $C>0$, depending only on the seminorms of symbols to $(a^w)^{-1}$ and $b^w$, such that for $f\in H(mc)$,
	\begin{align*}
		\|b(v,D_v)f\|_{H(c)}+\|b^w(v,D_v)f\|_{H(c)}\le C\|a^w(v,D_v)f\|_{H(c)}.
	\end{align*}
	Consequently, if $a^w:H(m_1)\to L^2\in Op(m_1)$, $b^w:H(m_2)\to L^2\in Op(m_2)$ are invertible, then for $f\in\S$, 
	\begin{align*}
		\|b^wa^wf\|_{L^2}\lesssim \|a^wb^wf\|_{L^2},
	\end{align*}where the constant depends only on seminorms of symbols to $a^w,b^w,(a^w)^{-1},(b^w)^{-1}$.
\end{Lem}

\begin{Lem}\label{bound_varepsilon}
	Denote $a_{K,l}:=a+Kl$, $m_{K,l}:=m+Kl$ for $K>1$, where $m,l$ are $\Gamma$-admissible weights. Assume $a\in S(m)$, $\partial_\eta (a_{K,l})\in S(K^{-\kappa}m_{K,l})$ uniformly in $K$ and
	$a_{K,l}\gtrsim m_{K,l}$.
	Let $\rho>0$ and $b\in S(\varepsilon m_{K,l}+\varepsilon^{-\rho}l)$, uniformly in $\varepsilon\in(0,1)$. Then there exists $K_0>0$, such that for $f\in H(mc)$, $\varepsilon\in(0,1)$,
	\begin{align}
		\|b(v,D_v)f\|_{H(c)}+\|b^w(v,D_v)f\|_{H(c)}\le C_{K,l}\left(\varepsilon\|a^w(v,D_v)f\|_{H(c)}+\varepsilon^{-\rho}\|l^wf\|_{H(c)}\right).
	\end{align}
\end{Lem}
Notice that the condition $l\le m$ in \cite{Deng2020a} is unnecessary. 

\begin{Lem}\label{innerproduct}
	Let $m,c$ be $\Gamma$-admissible weight and $a^{1/2}\in S(m^{1/2})$.
	Assume $(a^{1/2})^w:H(mc)\to H(c)$ is invertible and $b\in S(m)$. Then 
	\begin{align*}
		(b^wf,f)_{L^2} = (\underbrace{((a^{1/2})^w)^{-1}b}_{\in S(m^{1/2})}f,(a^{1/2})^wf)_{L^2}\lesssim \|(a^{1/2})^wf\|^2_{L^2}.
	\end{align*}
\end{Lem}

%
%
%


\paragraph{Acknowledgments:} The author would thank Prof. Tong Yang for the valuable comments on the manuscript. 

\small
\bibliographystyle{plain}
\bibliography{1}

\end{document}